\documentclass[aoas]{imsart}

\RequirePackage{amsthm,amsmath,amsfonts,amssymb}
\RequirePackage{bbm}
\RequirePackage[authoryear]{natbib}
\RequirePackage[colorlinks,citecolor=blue,urlcolor=blue]{hyperref}
\RequirePackage{graphicx}
\RequirePackage{scrextend}
\RequirePackage[ruled,norelsize,vlined]{algorithm2e}
\RequirePackage{colortbl} 
\RequirePackage{xcolor}
\RequirePackage{booktabs}
\usepackage{caption}
\usepackage{subcaption}
\usepackage{adjustbox}
\usepackage{etoolbox} 
\RequirePackage{xr-hyper} 
\apptocmd{\thebibliography}{\setlength{\itemsep}{5pt}}{}{}

\changefontsizes{14pt}

\startlocaldefs

\newcommand{\cc}{\mathcal{C}}
\newcommand{\E}{\mathbb{E}}
\newcommand{\R}{\mathbb{R}}

\newcommand{\argmin}{\mbox{argmin}} 
\newcommand{\argmax}{\mbox{argmax}} 

\newcommand\scalemath[2]{\scalebox{#1}{\mbox{\ensuremath{\displaystyle #2}}}}

\theoremstyle{plain}

\newtheorem{theorem}{Theorem}[section]
\newtheorem{lemma}[theorem]{Lemma}

\newtheorem{remark}[theorem]{Remark}
\newtheorem{corollary}[theorem]{Corollary}

\theoremstyle{remark}

\newtheorem{assumption}{Assumption}

\endlocaldefs
\renewcommand{\baselinestretch}{0.95}

\begin{document}

\begin{frontmatter}
\title{\small Skewed Bernstein--von Mises theorem \\ and skew--modal approximations}
\runtitle{Skewed Bernstein--von Mises theorem and skew--modal approximations}

\begin{aug}
\author[A]{\fnms{Daniele} \snm{Durante}\ead[label=e1,mark]{daniele.durante@unibocconi.it}},
\author[B]{\fnms{Francesco} \snm{Pozza}\thanksref{t2}\ead[label=e2,mark]{francesco.pozza2@unibocconi.it}}
\and
\author[A]{\fnms{Botond} \snm{Szabo}\thanksref{t1}\ead[label=e3,mark]{botond.szabo@unibocconi.it}}
\address[A]{Department  of Decision Sciences and Institute for Data Science and Analytics,
Bocconi University,\\
\printead{e1,e3}}

\thankstext{t1}{\scriptsize Co--funded by the European Union (ERC, BigBayesUQ, project number: 101041064). Views and opinions expressed are however those of the author(s) only and do not necessarily reflect those of the European Union or the European Research Council. Neither the European Union nor the granting authority can be held responsible for them.}
\vspace{-5pt}

\address[B]{Institute for Data Science and Analytics,
	Bocconi University,
\printead{e2}}

\thankstext{t2}{\scriptsize Funded by the European Union (ERC, PrSc-HDBayLe, project number: 101076564). Views and opinions expressed are however those of the author(s) only and do not necessarily reflect those of the European Union or the European Research Council. Neither the European Union nor the granting authority can be held responsible for them.}
\vspace{-5pt}

\end{aug}
\begin{abstract}
Gaussian deterministic approximations are routinely employed in Bayesian statistics to ease inference when the target posterior of direct interest is intractable. Although these approximations are justified, in asymptotic regimes,  by Bernstein--von Mises type results, in practice the expected Gaussian behavior may poorly represent the actual shape~of~the~target posterior, thereby affecting approximation accuracy. Motivated by these considerations, we derive an improved class of closed--form and valid deterministic approximations of posterior distributions which arise from a novel treatment of a third--order version of the Laplace method yielding approximations within a tractable family of skew--symmetric distributions. Under general assumptions which  also account for misspecified models and non--i.i.d. settings, this novel family of approximations~is~shown~to have a total variation distance from~the target posterior whose rate of convergence improves by at least one order of magnitude the one achieved by the Gaussian from the classical Bernstein--von Mises theorem. Specializing such a general result to the case of regular parametric models shows that the same improvement in approximation accuracy can be also established for polynomially bounded posterior functionals, including moments.  Unlike other higher--order approximations based on, e.g., Edgeworth expansions, our results prove that it is possible to derive closed--form and valid densities~which are expected to provide, in practice, a more accurate, yet similarly--tractable, alternative to Gaussian approximations of the target posterior of direct interest, while inheriting its limiting  frequentist properties. We strengthen  these arguments by developing a practical skew--modal approximation for both joint and marginal posteriors which achieves the same~theoretical guarantees of its theoretical counterpart by replacing the unknown model parameters with the corresponding maximum a posteriori estimate. Simulation studies and real--data applications confirm that our theoretical results closely match the remarkable empirical performance observed in practice, even in finite, possibly small, sample regimes. 
\end{abstract}



\end{frontmatter}


\section{Introduction} \label{sec_1}
Modern Bayesian statistics relies extensively on deterministic approximations  to facilitate inference  in those challenging, yet routine, situations when the target posterior of direct interest is intractable \citep[e.g.,][]{tierney1986accurate,minka2013expectation,inla_paper,blei2017variational}.  A natural option to enforce the desired tractability is to constrain~the approximating distribution within a suitable family which facilitates the evaluation of functionals of interest for inference. To this end, both classical solutions, such as the approximation of posterior distributions induced by the Laplace method \citep[e.g.,][Ch.\ 4.4]{bishop2006pattern}, and state--of--the--art strategies, including, for example, Gaussian variational Bayes \citep{opper2009variational} and standard implementations of expectation--propagation \citep{minka2013expectation}, employ Gaussian approximations. These further appear, either as the final solution or as a recurring building--block, also within several routinely--implemented alternatives, such as mean--field variational Bayes \citep{blei2017variational} and integrated nested Laplace approximation (\textsc{inla})  \citep{inla_paper}. See also \citet{wang2013variational}, \citet{Chopin_2017}, \citet{durante2019conditionally}, \citet{ray2022variational} and \citet{vehtari2020expectation}, among others, for further  examples illustrating the relevance of Gaussian approximations.

From a theoretical perspective, the choice of the Gaussian family to approximate the posterior distribution is justified, in asymptotic regimes, by Bernstein--von Mises type results. In its classical formulation  \citep{laplace1810theorie,bernstein1917,vonM1931,le1953some,le1990asymptotics,van2000asymptotic}, the Bernstein--von Mises theorem~states~that, in sufficiently regular parametric models, the posterior distribution converges in total variation (\textsc{tv}) distance, with probability tending to one under the law of the data, to a Gaussian distribution. The mean of such a limiting Gaussian is a known function of the true data--generative parameter, or any efficient estimator of this quantity, such as the maximum likelihood estimator, while the variance is the inverse of the Fisher information. Extensions of the Bernstein--von Mises theorem to more complex settings have also been made in recent years. Relevant contributions along these directions include, among others, generalizations to high-dimensional regimes \citep{boucheron2009bernstein,spok2021}, along with in--depth treatments of misspecified  \citep{kleijn2012bernstein}  and irregular   \citep{bochkina2014bernstein} models. Semiparametric settings have also been addressed \citep{bickel2012semiparametric,castillo2015bernstein}.  In the nonparametric context,  Bernstein--von Mises type results do not hold in general, but the asymptotic Gaussianity can be still proved for weak Sobolev spaces via a multiscale analysis \citep{castillo2014bernstein}. 

Besides providing  crucial advances in the understanding of the limiting frequentist properties of posterior distributions, the above Bernstein--von Mises type results  have also substantial implications in the design and in the theoretical justification of practical Gaussian deterministic approximations for intractable posterior distributions from, e.g., the Laplace method  \citep{kasprzak2022good}, variational Bayes (\textsc{vb}) \citep{wang2019frequentist} and expectation--propagation (\textsc{ep})  \citep{dehaene2018expectation}. Such a direction has led to important results. Nonetheless, in practical situations the Gaussian approximation may lack  the required flexibility to closely match the actual shape of the posterior distribution of direct interest, thereby undermining accuracy when inference is based on such  an approximation. In fact, as illustrated via two representative real--data clinical applications~in Sections~\ref{sec_322}--\ref{sec_hd_logistic}, the error in posterior mean estimation of the Gaussian approximation supported by the classical Bernstein--von Mises theorem is non--negligible not only in  a study with low sample size $n=27$ and $d=3$ parameters, but also in a higher--dimensional application with $n=333$ and $d \approx n/2.5$. Both regimes often occur in routine implementations. The results in Sections~\ref{sec_322}--\ref{sec_hd_logistic}  further clarify that the issues encountered by the Gaussian approximation are mainly due to the inability of capturing the non--negligible skewness often displayed by the actual posterior in these settings. Such an asymmetric shape is inherent to routinely--studied posterior distributions. For example,  \citet{durante2019conjugate},  \citet{fasano2022class} and \citet{anceschi2022bayesian} have recently proved that, under a broad class of priors which  includes multivariate normals, the posterior distribution induced by probit, multinomial probit and tobit models belongs to a  skewed generalization of  the Gaussian distribution known as unified skew--normal (\textsc{sun}) \citep{arellano2006unification}. More generally, available extensions of Gaussian deterministic approximations which account, either explicitly or implicitly, for skewness \citep[see e.g.,][]{inla_paper,challis2012affine,fasanoscalable} have shown evidence of improved empirical accuracy relative to their Gaussian counterparts. Nonetheless, these approximations are often model--specific and general  justifications relying on Bernstein--von Mises type results are not available yet. In fact, in--depth theory and methods for skewed approximations are either lacking or are tailored to specific  models and priors \citep[][]{fasanoscalable}.

In this article, we cover the aforementioned gaps by deriving an improved class~of~closed--form, valid and theoretically--supported skewed approximations of generic posterior distributions. Such a class
 arises from a novel treatment of a higher--order version of the Laplace method which replaces the third--order term with a suitable univariate cumulative distribution function (cdf) satisfying mild regularity conditions. As clarified in Section~\ref{sec_21}, this new perspective yields tractable approximations that crucially belong to the broad and known skew--symmetric family \citep[see e.g.,][]{ma2004flexible}. More specifically, these approximations can be readily~obtained by direct perturbation of the density~of a multivariate Gaussian via a suitably--defined univariate cdf evaluated at a cubic function of the parameter. This implies that the proposed class of approximations admits straightforward i.i.d.\ sampling schemes facilitating direct Monte Carlo evaluation of any functional of interest for posterior inference. These~are~crucial advancements relative to other higher--order studies relying on Edgeworth or other types of representations \citep[see e.g.,][and  references therein]{johnson1970asymptotic,weng2010bayesian,kolassa2020validity}, which consider arbitrarily truncated versions of infinite~expansions~that~do not necessarily correspond to closed--form valid densities, even after normalization --- e.g., the density approximation is not guaranteed to be non--negative \citep[e.g.,][Remark 11]{kolassa2020validity}. This undermines the methodological and practical impact of current higher--order results which still fail to provide a natural, valid and general alternative to Gaussian deterministic approximations that can be readily employed in practice. In contrast, our novel results prove that a previously--unexplored treatment of specific~higher--order expansions can actually yield valid, practical and theoretically--supported approximations, thereby opening the avenues to extend such a perspective to orders even higher than the third one; see also our final discussion in Section~\ref{sec_4}.

Section~\ref{sec_general_thm} clarifies that the proposed class of skew--symmetric approximations has also strong theoretical support in terms of accuracy improvements relative to its Gaussian counterpart.  More specifically, in Theorem~\ref{thm:1} we prove that the newly--proposed class of skew--symmetric approximations has a total variation distance from the target posterior distribution whose rate of convergence improves by at least one order of magnitude the one attained by the Gaussian from the classical Bernstein--von Mises theorem. Crucially, this result~is~derived under general assumptions which account for both misspecified models and non--i.i.d. settings. This yields an important refinement of standard Bernstein--von Mises type results clarifying that it is possible to derive closed--form and valid~densities which are expected to provide, in practice, a more accurate, yet similarly--tractable, alternative to Gaussian approximations of the target posterior of direct interest, while inheriting its limiting frequentist properties. In Section~\ref{sec_fixd_thm} these general results are further specialized to, possibly non--i.i.d. and misspecified, regular parametric models,  where  $n \to \infty$ and the dimension~$d$ of the parametric space is fixed. Under such a practically--relevant setting, we show that~the proposed~skew--symmetric approximation can be explicitly derived as a function of the log--prior and log--likelihood derivatives. Moreover, we prove that  by~replacing the classical Gaussian~approximation from the Bernstein--von Mises theorem with such a newly--derived alternative yields a remarkable improvement~in the rates of order $\sqrt{n}$, up to a poly--log term. This gain is shown to hold not only for the \textsc{tv} distance from the target posterior, but also for the error in approximating polynomially bounded posterior functionals (e.g., moments).

The methodological impact of the theory in Section~\ref{sec_2} is further strengthened in Section~\ref{sec_3} through the development of a readily--applicable plug--in~version for the proposed class of skew--symmetric approximations derived in Section~\ref{sec_21}. This is obtained by replacing the unknown true data--generating parameter in the theoretical construction with the corresponding maximum a posteriori estimate, or any other efficient estimator. The resulting solution~is named skew--modal approximation and, under mild conditions, is shown to achieve the same improved rates~of~its~theoretical counterpart, both in terms of the \textsc{tv} distance from the target posterior distribution and with respect to the approximation  error for polynomially bounded posterior functionals. In such a practically--relevant setting, we further refine the theoretical analysis through the derivation of non--asymptotic bounds for the \textsc{tv} distance among the skew--modal approximation and the target posterior. These bounds are guaranteed to~vanish also when the dimension $d$ grows with $n$, as long as  $d \ll n^{1/3}$ up to~a poly--log~term. Interestingly, this condition is related to those required either for $d$ \citep[see e.g.,][]{panov2015finite} or for the notion of effective dimension \smash{$\tilde{d}$}  \citep[see][]{spok2021,spokoiny2023inexact} in recent high--dimensional studies of the Gaussian Laplace approximation. However, unlike these studies,  the  bounds we derive vanish with $n$,~up~to a poly--log term, rather than $\sqrt{n}$, for~any given dimensions. These advancements enable also~the~derivation of a novel lower bound for the \textsc{tv} distance among the Gaussian Laplace approximation and the target posterior, which is shown to still vanish with~$\sqrt{n}$.  This result strengthens~the proposed skew--modal solution whose associated upper bound vanishes with $n$, up to a poly--log term. When the focus of inference is not on the joint posterior but rather on its marginals,~we~further derive  in Section~\ref{sec_marginal}  accurate skew--modal approximations for such marginals that inherit the same theory guarantees while scaling~up~computation.

The superior empirical performance of the newly--proposed class of skew--symmetric approximations and the practical consequences of our theoretical results on the  improved rates are~illustrated through both simulation studies and two real--data applications in Sections~\ref{sec_24} and~\ref{sec_32}. All these numerical analyses demonstrate that the remarkable~theoretical improvements encoded within the rates we derive closely match the empirical behavior observed in practice even in finite, possibly small, sample regimes. This translates into noticeable empirical accuracy gains relative to the classical Gaussian--modal approximation from the Laplace method. Even more, in the real--data application the proposed skew–modal approximation also displays a highly competitive performance with respect to more sophisticated state--of--the--art Gaussian and non--Gaussian approximations from mean--field \textsc{vb} and expectation--propagation  (\textsc{ep}) \citep{minka2013expectation,blei2017variational,Chopin_2017,durante2019conditionally,fasanoscalable}. 

As discussed in the concluding remarks in Section~\ref{sec_4}, the above results stimulate future advancements aimed at refining the accuracy of other Gaussian approximations from, e.g.,  \textsc{vb}  and \textsc{ep},  via the inclusion of skewness. To this end, our contribution provides the foundations to achieve this goal, and suggests that a natural and tractable class where~to seek these improved approximations would be still the skew--symmetric family. Extensions to~higher--order expansions beyond the third term are also discussed as directions of future research. Finally, notice that although the non--asymptotic bounds we derive for the skew--modal approximation in Section~\ref{sec_3} yield refined theoretical results that can be readily proved for the general skew--symmetric class in Section~\ref{sec_2}, the  practical consequences of non--asymptotic bounds and the associated constants is an ongoing area of research even for basic Gaussian approximations \citep[see e.g.,][and references therein]{kasprzak2022good}. In this sense, it shall be emphasized that, in our case, even the asymptotic theoretical support encoded in the rates we derive finds empirical evidence in  the finite--sample studies~considered in Sections~\ref{sec_24} and~\ref{sec_32}.


\subsection{Notation} \label{sec_11}
We denote with $\{X_i\}_{i = 1}^n$, $n \in \mathbbm{N}$, a sequence of random variables with unknown true distribution $P_0^n$. Moreover, let $\mathcal{P}_{\Theta} = \left\{ P_\theta^{n}, \theta \in \Theta \right\}$, with $\Theta\subseteq \mathbb{R}^d$, be a parametric family of distributions. In the following, we assume that there exists a common $\sigma$--finite measure $\mu^{n}$ which dominates $P_0^n$ as well as all measures $P_\theta^{n}$, and we denote by~$p_0^n$~and~$p_{\theta}^{n}$ the corresponding density functions. The Kullback--Leibler projection \smash{$P_{\theta_*}^n$ of $P_0^n$ on $\mathcal{P}_{\Theta}$} is defined as \smash{$P_{\theta_*}^n = \argmin_{P_\theta^n \in \mathcal{P}_{\Theta} }\textsc{kl}(P_{0}^{n} \| P_{\theta}^{n})$} where 
$\textsc{kl}(P_{0}^{n} \| P_{\theta}^{n})$ denotes the Kullback--Leibler divergence between  $P_{0}^{n}$ and $P_{\theta}^{n}$. The log--likelihood of the, possibly misspecified, model is $ \ell(\theta ) = \ell(\theta, X^{n} ) =  \log p_{\theta}^{n}\left(X^{n}\right)$. Prior and posterior distributions are denoted by $\Pi(\cdot)$ and $\Pi_n(\cdot)$, whereas the corresponding densities are indicated with $\pi(\cdot)$ and $\pi_n(\cdot)$, respectively.
 
As mentioned in Section~\ref{sec_1}, our results rely on higher--order expansions and derivatives.~To this end,  we characterize operations among vectors, matrices and arrays in a  compact manner by adopting the index notation along with the Einstein's summation convention  \citep[see e.g.,][p. 335]{salvan1997}. More specifically, the inner product $Z^{\intercal}a$ between the  generic random vector $Z \in \mathbb{R}^d,$ with components $Z_s$ for $s = 1,\dots, d$, and the vector of coefficients $ a \in \mathbb{R}^d$  having elements $a_s$ for $s = 1,\dots,d,$ is expressed as $ a_s Z_s$, with the sum being implicit in the repetition of the indexes.  Similarly, if $B$ is a $d \times d$ matrix with entries $b_{st}$ for $s,t = 1,\dots,d,$ the quadratic form  $Z^\intercal B Z$ is expressed as $ b_{st} Z_{s} Z_{t}.$ The generalization to operations involving arrays with higher dimensions is obtained under the same reasoning.
   	
	Leveraging the above notation, the score vector evaluated at $\theta_*$ is defined as
	 \begin{eqnarray*}
	\ell^{(1)}_{\theta_{*}} = [ \ell^{(1)}_s(\theta) ]_{\mid \theta = \theta_*} \, = \, \left[ (\partial/\partial \theta_{s})\ell(\theta) \right]_{\mid \theta = \theta_*}\in\mathbb{R}^{d}, 
			 \end{eqnarray*}
	whereas, the second, third and fourth order derivatives of $\ell(\theta)$, still evaluated at $\theta_*$, are 
	 \begin{eqnarray*}
	\begin{split}
	&	\ell^{(2)}_{\theta_{*}} = [ \ell^{(2)}_{st}(\theta)]_{\mid \theta = \theta_*} \,  \,  = \, [ \partial/(\partial \theta_{s} \partial \theta_t )\ell(\theta) ]_{\mid \theta = \theta_*}\in\mathbb{R}^{d\times d},\\
	&	\ell^{(3)}_{\theta_{*}} = [ \ell^{(3)}_{stl}(\theta) ]_{\mid \theta = \theta_*} \,  \,  = \, [ \partial/(\partial \theta_{s} \partial \theta_t \partial \theta_l  )\ell(\theta) ]_{\mid \theta = \theta_*}\in\mathbb{R}^{d\times d\times d}, \\
	&	\ell^{(4)}_{\theta_{*}} = [ \ell^{(4)}_{stlk}(\theta) ]_{\mid \theta = \theta_*} \, = \, [ \partial/(\partial \theta_{s} \partial \theta_t \partial \theta_l \partial \theta_{k} )\ell(\theta) ]_{\mid \theta = \theta_*}\in\mathbb{R}^{d\times d \times d \times d }, 
	\end{split}
		 \end{eqnarray*}
		\vspace{5pt} 
\noindent	where all the indexes in the above definitions and in the subsequent ones go from $1$ to $d$. The observed and expected Fisher information are denoted by \smash{$J_{\theta_*} = [j_{st}] = -[\ell^{(2)}_{\theta_{*},st}] \in\mathbb{R}^{d\times d}$} and \smash{$I_{\theta_*} = [i_{st}] = [\mathbb{E}_{0}^{n} j_{st}]\in\mathbb{R}^{d\times d}$}, respectively, where $\mathbb{E}_{0}^{n}$ is the expectation with respect to $P_{0}^{n}$. In addition, 
	 \begin{eqnarray*}
	 	\begin{aligned}
	 		\log \pi^{(1)}_{\theta_*}  = &[ 	\log \pi (\theta)_{s}^{(1)} ]_{\mid \theta = \theta_*}  = \, [ \partial/(\partial \theta_{s} )\log \pi (\theta) ]_{\mid \theta = \theta_*}\in\mathbb{R}^{d},  \\
	 	 \log \pi^{(2)}_{\theta_*}  = &[ 	\log \pi (\theta)_{st}^{(2)} ]_{\mid \theta = \theta_*}  = \, [ \partial/(\partial \theta_{s} \partial \theta_{t} )\log \pi (\theta) ]_{\mid \theta = \theta_*}\in\mathbb{R}^{d \times d}, \\
	 	\end{aligned}
	 \end{eqnarray*}
	 represent the first two derivatives of the log--prior density, evaluated at $\theta_*.$
	 
	The Euclidean norm of a vector $a \in \mathbb{R}^d$ is denoted by $\|a\|$, whereas, for a generic $d \times d$ matrix $B$, the notation $|B|$ indicates its determinant, while $\lambda_{\textsc{min}}(B)$ and $\lambda_{\textsc{max}}(B)$~its~minimum and maximum eigenvalue, respectively. Furthermore, $u \wedge v$ and $u \vee v$ correspond to $\min\{u, v\}$ and $\max\{u, v\}$. For two positive sequences $u_n,v_n$ we employ $u_n\lesssim v_n$ if there exists a universal positive constant $C$ such that $u_n\leq C v_n$. When $u_n\lesssim v_n$ and $v_n\lesssim u_n$ are satisfied simultaneously, we write $u_n\asymp v_n$.  


\section{A skewed Bernstein--von Mises theorem} \label{sec_2}
This section presents our first important contribution. In particular, Section~\ref{sec_21} shows that, for Bayesian models satisfying a refined version of the local asymptotic normality (\textsc{lan}) condition \citep[see e.g.,][]{van2000asymptotic,kleijn2012bernstein}, a previously--unexplored treatment of a third--order version of~the Laplace method can yield a novel, closed--form and valid approximation of~the posterior distribution. Crucially, this approximation is further shown to belong to the tractable skew--symmetric (\textsc{sks}) family  \citep[e.g.,][]{ma2004flexible}.   Focusing on this newly--derived class of \textsc{sks} approximations, we then prove in Section~\ref{sec_general_thm} that the $n$--indexed sequence of \textsc{tv} distances~between such a class and the target posterior has a rate which improves by~at least one order of magnitude the one achieved  under the classical Bernstein--von Mises theorem based on Gaussian approximations. Such a skewed Bernstein--von Mises type result is proved under general assumptions which account for misspecified models in non--i.i.d. settings. Section~\ref{sec_fixd_thm}  then specializes this result to the practically--relevant context of regular parametric models with $n \to \infty$ and fixed $d$. In this setting we prove that the improvement~in rates over the Bernstein--von Mises theorem is by a multiplicative factor of order $\sqrt{n}$, up to a poly--log term. This result is shown to hold not only for the \textsc{tv} distance from the posterior, but also for the  error in approximating polynomially bounded posterior functionals.

Let $\delta_n \to 0$ be a generic norming rate governing the posterior contraction toward $\theta_*$.
Consistent with standard Bernstein--von Mises  type theory \citep[see e.g.,][]{van2000asymptotic,kleijn2012bernstein}, consider the re--parametrization $h = \delta_n^{-1}(\theta-\theta_*) \in \mathbb{R}^d$. Moreover, let $F(\cdot) \, : \, \R \to [0{,}1]$ denote any~univariate cumulative distribution function which satisfies $F(-x) = 1-F(x)$ and~$F(x)\ {=}  1/2 + \eta x + O(x^2), \, x \to 0$, for some $ \eta \in \R$. Then, the class of \textsc{sks} approximating  densities $p_{\textsc{sks}}^n(h)$ we derive and study has the  general form
	\begin{eqnarray}\label{limiting:distribution}
		p_{\textsc{sks}}^n( h) \, = \, 2 \phi_d \left( h ; \xi, \Omega \right)w(h-\xi) = \, 2 \phi_d \left( h ; \xi, \Omega \right)F( \alpha_{\eta}(h-\xi)),
	\end{eqnarray}
    with $P_{\textsc{sks}}^n(S) = \int_S p_{\textsc{sks}}^n(h) dh$ denoting the associated cumulative distribution function.~In~\eqref{limiting:distribution}, $\phi_d \left( \cdot ; \xi, \Omega \right)$ is the density of a $d$--variate Gaussian  with mean vector $\xi$ and covariance matrix $\Omega$, while the function $ w(h-\xi) \in (0,1)$ is responsible for inducing~skewness, and takes the form $w(h-\xi) = F \left( \alpha_{\eta}(h-\xi) \right)$, where  $\alpha_{\eta}({\cdot})  : \R^d \to \R$ denotes a third order odd polynomial depending on the parameter that regulates the expansion of $F(\cdot)$, i.e., $\eta$.  

Crucially, Equation \eqref{limiting:distribution} not only ensures that $p_{\textsc{sks}}^n(h)$ is a valid and closed--form density, but also that such a density belongs to the tractable and known skew--symmetric class  \citep{azzalini2003distributions,ma2004flexible}. This follows directly by the definition of \textsc{sks} densities, provided that  $\alpha_{\eta}(\cdot)$ is an odd function, and $\phi_d \left( \cdot ; \xi, \Omega \right)$ is symmetric about $\xi$  \citep[see e.g.,][Proposition 1]{azzalini2003distributions}. Therefore, in contrast to available higher--order studies of posterior distributions based on Edgeworth~or~other type of expansions \citep[see e.g.,][]{johnson1970asymptotic,weng2010bayesian,kolassa2020validity}, our~theoretical and methodological results focus on a family of closed--form and valid approximating densities which are essentially as tractable as multivariate Gaussians,~both~in~terms~of~evaluation of the corresponding density, and i.i.d.\ sampling.~More~specifically,~let~$z_0 \in  \mathbb{R}^d$~and $z_1 \in [0,1]$ denote samples from a $d$--variate Gaussian having density $\phi_d( z_0 ; 0, \Omega )$ and from a~uniform with support $[0,1]$, respectively. Then, adapting results in, e.g., \citet{wang2004skew}, a sample from the \textsc{sks} distribution with density  as in \eqref{limiting:distribution} can be readily~obtained~via
\begin{eqnarray*}
\xi+\mbox{sgn}( F ( \alpha_{\eta}(z_0) )-z_1)z_0.
\end{eqnarray*}
Therefore, sampling from the proposed \textsc{sks} approximation~simply reduces to drawing values from a $d$–variate Gaussian and then changing or not the sign of the sampled value via a straightforward perturbation scheme.

As clarified within Sections \ref{sec_fixd_thm} and \ref{sec_3}, the general \textsc{sks} approximation in \eqref{limiting:distribution}~is~not~only interesting from a theoretical perspective, but has also relevant methodological consequences and direct applicability. This is because, when specializing  the general theory~in Section~\ref{sec_general_thm} to, possibly misspecified and non--i.i.d., regular parametric models where $n \to \infty$, $d$ is fixed and $\delta_n^{-1} = \sqrt{n}$, we can show that the quantities defining $p_{\textsc{sks}}^n(h)$ in \eqref{limiting:distribution} can be expressed as closed--form functions of the log--prior and log--likelihood derivatives  at $\theta_*$. In particular, let $\smash{u_t=(\ell^{(1)}_{\theta_{*}} +  \log \pi^{(1)}_{\theta_{*}} )_t/\sqrt{n}}$ for $t=1, \ldots, d$, then, as clarified in Section \ref{sec_fixd_thm}, we have
	\begin{eqnarray}
\scalemath{1}{	\begin{split}
		\qquad &\xi=\smash{[n (J^{-1}_{\theta_*})_{st}u_t]}, \qquad \Omega^{-1} = [ j_{st}/n-(\xi_{l}\ell^{(3)}_{\theta_*,stl}/n)/\sqrt{n} ],\\
	&\alpha_{\eta}(h-\xi) ={\{}1/(12 \eta \sqrt{n})\}  (\ell^{(3)}_{\theta_*,stl}/n) \{(h-\xi)_{s}(h-\xi)_{t}(h-\xi)_{l}+3(h-\xi)_{s} \xi_{t} \xi_{l}\}.
	\end{split}}
	\label{eq_as_1}
	\end{eqnarray}
Interestingly, in this case, the first factor on the right hand side of \eqref{limiting:distribution} closely resembles~the limiting Gaussian density with mean vector \smash{$\ell^{(1)}_{\theta_{*}}/\sqrt{n}$} and covariance matrix \smash{$(I_{\theta_{*}}/n)^{-1}$} from the classical Bernstein--von Mises theorem which, however, fails to incorporate skewness. To this end, the symmetric component in \eqref{limiting:distribution} is perturbed via a skewness--inducing mechanism regulated by $w(h-\xi)$ to obtain a valid  asymmetric density with tractable normalizing constant. As shown in Section~\ref{sec_3}, this solution admits a directly--applicable practical counterpart, which can be obtained by replacing $F(\cdot)$ and $\theta_*$  in \eqref{limiting:distribution}--\eqref{eq_as_1}, with routine--use univariate cdfs such as, e.g., $\Phi(\cdot)$, and with the maximum a posteriori estimate \smash{$\hat{\theta}$} of $\theta$, respectively. This results in a practical and novel skew--modal approximation that can be shown to have the same theoretical guarantees of improved accuracy of its theoretical counterpart.


\subsection{Derivation of the skew--symmetric approximating density} \label{sec_21}
Prior to~stating and proving within Section~\ref{sec_general_thm}  the general skewed Bernstein--von Mises theorem which supports the proposed class of \textsc{sks} approximations, let us focus on providing a constructive derivation of~such~a class via a novel treatment of a third--order extension of the Laplace method. To simplify  notation, we consider the simple univariate case with $d=1$ and $\delta_n^{-1}= \sqrt{n}$.  The extension of these derivations to $d>1$ and to the general setting we consider in Theorem~\ref{thm:1} follow as a direct adaptation of the reasoning for the~univariate case; see Sections \ref{sec_general_thm}--\ref{sec_fixd_thm}. 

As a first step towards deriving the approximating density  $p_{\textsc{sks}}^n(h)$, notice that the posterior for $h = \sqrt{n}(\theta-\theta_*) $ can be expressed as
		\begin{eqnarray} \label{un:post:1d}
		\pi_n(h)  \propto \frac{p^{n}_{\theta_* + h/\sqrt{n}}}{p^{n}_{\theta_*}}(X^{n}) \frac{\pi ( \theta_* + h/\sqrt{n})}{\pi(\theta_*)},
	\end{eqnarray}
	since $p^{n}_{\theta_*}(X^{n})$ and $\pi(\theta_*)$ do not depend on $h$, and $\theta=\theta_*+h/\sqrt{n}$.
	
	Under suitable regularity conditions discussed in Sections~\ref{sec_general_thm}--\ref{sec_fixd_thm} below, the third--order Taylor's expansion for the logarithm of the likelihood ratio in Equation~\eqref{un:post:1d} is
	\begin{eqnarray} \label{lr:1d}
			\log  \frac{p^{n}_{\theta_* + h/\sqrt{n}}}{p^{n}_{\theta_*}}(X^{n})  = \frac{\ell^{(1)}_{\theta_*}}{\sqrt{n}}  h - \frac{1}{2} \frac{j_{\theta_*}}{n}  h^2 + \frac{1}{6 \sqrt{n}}\frac{\ell^{(3)}_{\theta_*}}{n} h^3 + O_{P^{n}_{0}}\big( n^{-1} \big),
	\end{eqnarray}
whereas the first order Taylor's expansion of the log--prior ratio is
		\begin{eqnarray}\label{taylor:prior1}
			\log \frac{\pi ( \theta_* + h/\sqrt{n})}{\pi(\theta_*)} = \frac{\log \pi^{(1)}_{\theta_*}}{\sqrt{n}} h + O\big( n^{-1} \big).
	\end{eqnarray}  
	Combining~\eqref{lr:1d} and \eqref{taylor:prior1} it is possible to reformulate the right--hand--side of Equation~\eqref{un:post:1d} as
		\begin{eqnarray} \label{eq:help00}
	\quad \frac{p^{n}_{\theta_* + h/\sqrt{n}}}{p^{n}_{\theta_*}}(X^{n}) \frac{\pi ( \theta_* + h/\sqrt{n})}{\pi(\theta_*)}= \exp \Big(  u h - \frac{1}{2} \frac{j_{\theta_*}}{n}  h^2 + \frac{1}{6 \sqrt{n}}\frac{\ell^{(3)}_{\theta_*}}{n} h^3 \Big) + O_{P^{n}_{0}}(n^{-1}),	
	\end{eqnarray}
	where $u = ( \ell^{(1)}_{\theta_*} +\log \pi^{(1)}_{\theta_*})/\sqrt{n}.$ 
	
	Notice that, up to a multiplicative constant, the Gaussian density arising from the classical Bernstein--von Mises theorem can be obtained by neglecting all terms in (\ref{lr:1d})--(\ref{taylor:prior1}) which converge to zero in probability. These correspond to the contribution of the prior, the difference between the observed and expected Fisher information, and the term associated to the third--order log--likelihood derivative. Maintaining these quantities would surely yield improved accuracy, but it is unclear whether a valid and similarly--tractable density can be still identified. In fact, current solutions  \citep[e.g.,][]{johnson1970asymptotic} consider approximations based on the sum among a Gaussian density and additional terms in the Taylor's expansion. However, as for related alternatives arising from Edgeworth--type  expansions \citep[e.g.,][]{weng2010bayesian,kolassa2020validity}, there is no guarantee that such constructions provide valid densities. 
	
	As a first key contribution we prove below that a valid and tractable approximating density can be, in fact, derived from the above expansions and belongs to the \textsc{sks} class. To this end, let  $\omega=1/v$ with \smash{$ v =   j_{\theta_*}/n-(\xi \ell^{(3)}_{\theta_*}/n)/\sqrt{n}$} and \smash{$\xi = n(j_{\theta_*})^{-1} u $}, and note that, by replacing $h^3$  in the right hand side of Equation~\eqref{eq:help00} with $(h-\xi+\xi)^3$, the exponential term in   \eqref{eq:help00} can be rewritten as proportional to
	\begin{eqnarray}\label{eq:help00_0}
		\phi(h;\xi , \omega ) \exp	( \{1/(6 \sqrt{n})\}(\ell^{(3)}_{\theta_*}/n)\big\{  (h-\xi)^3 + 3(h-\xi)\xi^2\big \} ).
	\end{eqnarray}
	At this point, recall that,  for $x \to 0$, we can write  $ \exp(x) = 1 + x + O(x^2)$ and $ 2F(x) =~1 +2 \eta x + O(x^2)$, for some $\eta \in \mathbb{R}$, where $F(\cdot)$ is the univariate cumulative distribution function introduced in Equation~\eqref{limiting:distribution}. Therefore, leveraging the similarity among these  two expansions and the fact that the exponent in Equation~\eqref{eq:help00_0} is an odd function of $(h-\xi)$ about $0$, of order \smash{$O_{P_0^{n}}(n^{-1/2})$}, it follows that \eqref{eq:help00_0}  is equal to
	\begin{eqnarray*} 2 \phi(h; \xi , \omega) F( \alpha_{\eta}(h-\xi) ) + O_{P_0^{n}}(n^{-1}),\end{eqnarray*}
	with $ \alpha_{\eta}(h-\xi)$ defined as in Equation~\eqref{eq_as_1}, for a univariate setting. The above expression~coincides with the univariate case of the skew--symmetric density in~ Equation~\eqref{limiting:distribution}, up~to~an~additive $O_{P^{n}_{0}}(n^{-1})$ term. The direct extension of the above derivations to the multivariate case  provides the general form~of  \smash{$p_{\textsc{sks}}^n(h)$} in  Equation \eqref{limiting:distribution} with parameters as in  \eqref{eq_as_1}. Section~\ref{sec_general_thm} further extends, and supports theoretically, such a construction in more general settings.


\subsection{The general theorem} \label{sec_general_thm}
The core message of Section~\ref{sec_21} is that a suitable treatment of the cubic terms in the Taylor expansion of the log--posterior can yield a higher--order,~yet valid, \textsc{sks} approximating density. This solution is expected to improve the accuracy of the classical second--order Gaussian approximation, while avoiding known issues of  polynomial approximations, such as regions with negative mass \citep[e.g.,][p.~154]{mccullagh2018tensor}.

In this section, we provide theoretical support to the above arguments. More specifically, we clarify that the  derivations in Section~\ref{sec_21}  can be applied generally to obtain  \textsc{sks} approximations in~a~variety of different settings, provided that the posterior contraction is~governed by a generic norming rate $\delta_n \to 0$, and that some general, yet reasonable, regularity conditions are met.   In particular, Theorem \ref{thm:1} requires Assumptions \ref{cond:uni}--\ref{cond:concentration} below. For convenience, let us introduce the notation \smash{$M_n = \sqrt{c_0 \log \delta_n^{-1}}$}, with $c_0>0$ a constant to be specified~later.

\begin{assumption} \label{cond:uni}
 The Kullback--Leibler projection $\theta_* \in \Theta$ is unique.  
\end{assumption}
\vspace{-17pt}
\begin{assumption} \label{cond:LAN}  
	There exists a sequence of $d$-dimensional random vectors $\Delta^n_{\theta_*} {=} \ O_{P_{0}^n}({1})$, a sequence of $d \times d$ random matrices $V_{\theta_*}^n = [v_{st}^n]$ with $v_{st}^n = O_{P_0^n}(1)$, and also a sequence of $d \times d \times d$  random arrays \smash{$a^{(3),n}_{\theta_*}  = [a^{(3),n}_{\theta_*,stl}]$} with  \smash{$a^{(3),n}_{\theta_*,stl} = O_{P_0^n}(1)$}, so that
		\vspace{-2pt}   
			\begin{eqnarray*}
		\log \frac{p_{\theta_* + \delta_n h }^n}{p_{\theta_*}^n}(X^n) - h_s v_{st}^n \Delta_{\theta_*, t}^n + \frac{1}{2}  v_{st}^n h_s h_{t} - \frac{\delta_n}{6} a^{(3),n}_{\theta_*,stl} h_s h_t h_l  = r_{n,1}(h),
		\vspace{-5pt}
	\end{eqnarray*}
    with $r_{n,1} := \sup_{h \in K_n}\left| r_{n,1}(h) \right| = O_{P_0^n}(\delta_n^2 M_n^{c_1})$, for some positive constant $c_1>0$, where~$K_n = \{ \| \theta - \theta_* \| \leq M_n \delta_n \}$. In addition, there are two positive constants $\eta^*_1$ and $\eta^*_2$ such that the event $A_{n,0} = \{ \lambda_{\textsc{min}}( V_{\theta_*}^n ) > \eta^*_1 \} \cap \{ \lambda_{\textsc{max}}( V_{\theta_*}^n ) < \eta^*_2\},$ holds with $P_{0}^n A_{n,0}  = 1-o(1)$. 
\end{assumption} 
\vspace{-17pt}
\begin{assumption} \label{cond:prior} 
	There exists a $d$--dimensional vector $\log \pi^{(1)} =[\log \pi^{(1)}_s]$ such that 
	\vspace{-2pt}
		\begin{eqnarray*}  \log \pi(\theta_* + \delta_n h )/\pi(\theta_*) - \delta_n h_s \log \pi^{(1)}_s = r_{n,2}(h), 
				\vspace{-9pt}
			\end{eqnarray*}
		with $\log \pi^{(1)}_s = O(1) $ and $r_{n,2} := \sup_{h \in K_n}\left| r_{n,2}(h) \right| = O(\delta_n^2 M_n^{c_2})$ for some constant $c_2>0$. 
\end{assumption}
\vspace{-17pt}
\begin{assumption} \label{cond:concentration} 
	It holds $\lim_{\delta_n \to 0}	P_0^n \{ \Pi_n( \| \theta - \theta_*\| > M_n \delta_n ) < \delta_n^2 \} = 1.$
\end{assumption}

Albeit general, Assumptions \ref{cond:uni}--\ref{cond:concentration} provide reasonable conditions that extend those commonly considered to derive classical Bernstein--von Mises type results. Moreover, as clarified in Section~\ref{sec_fixd_thm}, these assumptions directly translate, under regular parametric models, into natural and explicit conditions on the behavior of the log--likelihood and log--prior.~In particular, Assumption \ref{cond:uni} is mild and can be found, for example, in \citet{kleijn2012bernstein}. Together with Assumption \ref{cond:concentration}, it guarantees that asymptotically the posterior distribution concentrates in the region where the two expansions in Assumptions~\ref{cond:LAN}--\ref{cond:prior} hold with~negligible remainders. Notice that, Assumptions \ref{cond:LAN} and~\ref{cond:concentration} naturally extend those found in~general theoretical studies of Bernstein--von Mises  type results  \citep[e.g.,][]{kleijn2012bernstein} to a third--order construction, which further requires quantification of rates. Assumption~\ref{cond:prior} provides, instead, an additional condition relative to those found in the classical theory. This is because, unlike for  second--order Gaussian approximations, the log--prior enters the  \textsc{sks}  construction through its first derivative; see Section~\ref{sec_21}. To this~end, Assumption~\ref{cond:prior} imposes suitable and natural smoothness conditions on the prior. Interestingly, such a need to include a careful  study for the behavior of~the prior density is also useful in forming~the bases to naturally extend our proofs and theory  to the general high--dimensional settings considered in \citet{spok2021} and \citet{spokoiny2023inexact} for the classical Gaussian Laplace approximation, where the prior effect has a critical role in controlling the behavior of the third and fourth--order components of the log--posterior. Motivated by these results,~Section~\ref{sec_3}~further derives  non--asymptotic bounds  for the practical skew--modal~approximation, which are guaranteed to vanish  also when $d$ grows with $n$, as long as this growth in the dimension is such that $d \ll n^{1/3}$ up to~a poly--log~term.  See Remark \ref{rem_31} for a detailed discussion and~Section~\ref{sec_4} for future research directions in high dimensions motivated by such a result.

Under the above assumptions, Theorem~\ref{thm:1} below provides theoretical support to the proposed \textsc{sks} approximation by stating a novel skewed Bernstein--von Mises  type result. More specifically, this result establishes that in general contexts, covering also misspecified models and non--i.i.d. settings, it is possible to derive a  \textsc{sks}  approximation, with density as in~\eqref{limiting:distribution}, whose \textsc{tv} distance from the target posterior has a rate which improves by at least one order of magnitude the one achieved~by the classical Gaussian counterpart from the Bernstein--von Mises theorem. By approaching the target posterior at a~provably--faster rate,~the proposed solution  is therefore expected to provide, in practice, a more accurate alternative to Gaussian approximations of the target posterior, while inheriting the corresponding limiting frequentist~properties. To this end, Theorem~\ref{thm:1} shall not be interpreted as a theoretical result aimed at providing novel or alternative frequentist support to Bayesian inference. Rather,~it represents an important refinement of the  classical Bernstein--von Mises theorem which guides and supports the derivation of improved deterministic approximations to be used in practice as tractable, yet accurate, alternatives to the intractable posterior of direct interest. Figure~\ref{fig:tv_intui} provides a graphical intuition for such an argument.  The practical impact of these results is  illustrated in  the empirical studies within Sections~\ref{sec_24} and~\ref{sec_32}. Such studies clarify that the theoretical improvements encoded in the rates we derive directly translate into the remarkable accuracy gains of the proposed class of \textsc{sks}  approximations observed in finite--sample studies.

\begin{theorem} \label{thm:1}
	Let $h = \delta_n^{-1}(\theta - \theta_*),$ and define  $M_n = \sqrt{c_0 \log \delta_n^{-1}}$, with $c_0>0$. Then, under Assumptions  \ref{cond:uni}--\ref{cond:concentration}, if $\theta_*$ is an inner point of $\Theta$, it holds
	\begin{eqnarray} \label{thm1:out}
		\| \Pi_n(\cdot) - P_{\textsc{sks}}^n(\cdot) \|_{\textsc{tv}} = O_{P_{0}^n}(M_n^{c_3} \delta_n^2),
	\end{eqnarray}
	where $c_3>0$, and $P_{\textsc{sks}}^n(\cdot)$ is the cdf of the \textsc{sks} density $p_{\textsc{sks}}^n(h)$ in  \eqref{limiting:distribution} with parameters 
\begin{eqnarray*}
	\begin{split}
		&\xi = \Delta_{\theta_*}^n + \delta_n (V^n_{\theta_*})^{-1} \log \pi^{(1)}, \qquad \Omega^{-1} = [v_{st}^n -  \delta_n \smash{a^{(3),n}_{\theta_*,stl}} \xi_l],\\
	&\alpha_{\eta}(h-\xi) = (\delta_n/12 \eta)\smash{a^{(3),n}_{\theta_*, stl}} \{  (h-\xi)_s   (h-\xi)_t  (h-\xi)_l + 3  (h-\xi)_s\xi_t \xi_l    \}.
	\end{split}
	\end{eqnarray*}	
The function $F(\cdot)$ entering the definition of $p_{\textsc{sks}}^n(h)$  in  \eqref{limiting:distribution}  is any univariate cdf  which satisfies $F(-x) = 1 - F(x) $ and $F(x) = 1/2 + \eta x + O(x^2)$, for some $\eta \in \mathbb{R}$, when $x \to 0$.
\end{theorem}

\begin{remark} \label{re1}
Under related conditions and a simpler proof, it is possible to show that the order of convergence for the Bernstein--von Mises theorem based on  limiting~Gaussians is $O_{P_{0}^{n}}(M_n^{c_4} \delta_n)$, for some $c_4>0$. Thus, Theorem~\ref{thm:1} guarantees that by~relying on suitably--derived \textsc{sks} approximations with density as in \eqref{limiting:distribution}, it is possible to improve~the~rates of the classical Bernstein--von Mises result by a multiplicative factor of $\delta_n$. This follows directly from the fact that the \textsc{sks} approximation is able to include terms~of~order $\delta_n$ that are present in the Taylor expansion of the log--posterior but are neglected in~the Gaussian limit. This allows an improved redistribution of the mass in the high posterior probability region, thereby yielding increased accuracy in characterizing the shape of the target posterior. As~illustrated in Sections \ref{sec_24} and \ref{sec_32}, this correction yields remarkable accuracy improvements in practice. 
\end{remark}

\begin{remark} \label{reF}
Theorem~\ref{thm:1} holds for a broad class of \textsc{sks} approximating distributions as long as the univariate cdf $F(\cdot)$ which enters the skewing factor in~\eqref{limiting:distribution} satisfies  $F(-x) = 1 - F(x) $ and $F(x) = 1/2 + \eta x + O(x^2)$ for some $\eta \in \mathbb{R}$ when $x \to 0$. These conditions are mild and  add flexibility in the selection of  $F(\cdot)$. Relevant and practical examples of possibile choices for $F(\cdot)$ are the cdf of the standard Gaussian distribution, $\Phi(\cdot)$, and the inverse logit function, $g(\cdot) = \exp(\cdot)/\{1 + \exp(\cdot)\} $. Both satisfy $F(-x) = 1 - F(x) $,~and the associated Taylor expansions are \smash{$\Phi(x) = 1/2 + x/\sqrt{2 \pi} + O(x^3)$} and \smash{$g(x) = 1/2 + x/4 + O(x^3)$}, respectively, for $x \to 0$. Interestingly, when $F(\cdot)=\Phi(\cdot)$, the resulting skew--symmetric approximation belongs to the well--studied sub--class of generalized skew--normal (\textsc{gsn}) distributions \citep{ma2004flexible}, which provide the most natural extension of multivariate skew--normals \citep{azzalini2003distributions} to more flexible skew--symmetric representations. Due to this,  Sections~\ref{sec_24} and~\ref{sec_32} focus on assessing the empirical performance of such a noticeable example.
\end{remark} 

Before entering the details of the proof of Theorem~\ref{thm:1}, let us highlight an interesting~aspect regarding the interplay between skew--symmetric and Gaussian approximations that~can be deduced from our theoretical studies. In particular, notice that Theorem~\ref{thm:1} states results in terms of approximation of the whole posterior distribution under the \textsc{tv} distance. This implies, as a direct consequence of the definition of such a distance, that the same rates hold also for the absolute error in approximating the posterior expectation of any bounded function. As shown later in Corollary \ref{corol:1}, such an improvement can also be proved, under mild additional conditions, for the approximation of the posterior expectation of any~function bounded by a polynomial (e.g., posterior moments). According to Remark~\ref{re1} these rates cannot be achieved in general under a Gaussian approximation. Nonetheless, as stated in  Lemma \ref{lemma:distr:inv}, for some specific functionals the classical Gaussian approximation can actually attain the same rates of its skewed version. This result follows from the skew--symmetric distributional invariance with respect to even functions \citep{wang2004skew}. Such a property implies that the \textsc{sks} approximation $2 \phi_d \left( h ; \xi, \Omega \right)F( \alpha_{\eta}(h-\xi))$ and its Gaussian component $\phi_d \left( h ; \xi, \Omega \right)$ yield the same level of accuracy in estimating the posterior~expected value of functions that are symmetric with respect to the location parameter $\xi$. Thus, our results  provide also a novel explanation of the phenomenon observed in \citet{spok2021} and \citet{spokoiny2023inexact}, where the quality of the Gaussian approximation, within high--dimensional models, increases by an order of magnitude when evaluated on Borel sets which are centrally symmetric with respect to the location $\xi$ \citep[e.g.,][Theorem 3.4]{spokoiny2023inexact}. Nonetheless, as clarified in Theorem~\ref{thm:1} and Remark~\ref{re1}, Gaussian approximations remain still unable to attain the same rates of the \textsc{sks} counterparts in the estimation of generic functionals. Relevant examples are highest posterior density intervals which are often studied in practice and will be non--symmetric by definition whenever the posterior is skewed. 
\vspace{-2pt}

\begin{lemma} \label{lemma:distr:inv}
	Let $2 \phi_d ( h ; \xi, \Omega )F( \alpha_{\eta}(h-\xi))$ with $ \xi \in \mathbb{R}^d$ and $ \Omega \in\mathbb{R}^d \times \mathbb{R}^d$, be a skew-symmetric approximation of $\pi_n(h)$ and let $G\,: \, \mathbb{R}^d \to \mathbb{R}$ be an even function. If both $\int G(h -\xi ) \pi_n(h) d h$  and  $\int G(h - \xi ) 2 \phi_d(h; \xi , \Omega) F( \alpha_{\eta}(h-\xi)) d h$ are finite,  it holds  
	\begin{eqnarray*}
		 {\int} G(h - \xi )\{ \pi_n(h) - 2 \phi_d(h; \xi , \Omega) F( \alpha_{\eta}(h-\xi)) \}dh  \ {=}  {\int} G(h - \xi )\{ \pi_n(h) - \phi_d(h; \xi ,  \Omega) \}d h .
	\end{eqnarray*} 
\end{lemma}   
As clarified in the Supplementary Materials, Lemma \ref{lemma:distr:inv} follows as a direct consequence of Proposition 6 in \citet{wang2004skew}. 

The proof of Theorem~\ref{thm:1} is reported below and extends to \textsc{sks} approximating distributions the  reasoning behind the derivation of general Bernstein–von Mises type results~\citep[][]{kleijn2012bernstein}. Nonetheless, as mentioned before, the need to derive sharper rates which establish a higher approximation accuracy, relative to  Gaussian limiting distributions, requires a number of additional technical lemmas and refined arguments ensuring a tighter control of the error terms  in the expansions behind Theorem~\ref{thm:1}. Note~also~that,~in addressing these aspects, it is not sufficient to rely  on standard theory for higher--order approximations. In fact, unlike for current results, Theorem~\ref{thm:1} establishes improved rates for a valid class of approximating densities. This means that, beside replacing the second--order expansion of the log--posterior with a third--order one, it is also necessary to carefully control the difference between such an expansion and the class of \textsc{sks} distributions.

\begin{proof} Let $p_{\textsc{sks}}^n(h)$ denote the \textsc{sks}  density in  \eqref{limiting:distribution} with parameters derived in Lemma~\ref{lemma:1} under Assumptions \ref{cond:LAN} and \ref{cond:prior} (see the Supplementary Materials). In addition, let $\pi_n^{K_n}(h)$ and \smash{$p_{\textsc{sks}}^{n,K_n}(h)$} be the constrained versions of  $\pi_{n}(h)$ and \smash{$p_{\textsc{sks}}^{n}(h)$} to the set $K_n=\{h{:} \  ||h|| <~M_n \}$, i.e,~$\smash{\pi_n^{K_n}(h)}= \pi_{n}(h)\mathbbm{1}_{h \in K_n}/ \Pi_n(K_n)$ and  \smash{$p_{\textsc{sks}}^{n,K_n}(h)= p_{\textsc{sks}}^n(h)\mathbbm{1}_{h \in K_n}/ P_{\textsc{sks}}^n(K_n)$}. Using triangle inequality	
\begin{eqnarray} \label{triangle:tv}
		\begin{aligned}
			\quad &\| \Pi_n(\cdot) - P_{\textsc{sks}}^n(\cdot) \|_{\textsc{tv}} =\frac{1}{2}\int| \pi_n(h) - p_{\textsc{sks}}^n(h)| dh \leq  \\
			& \quad \ \	 {\int}| \pi_n(h) -  \pi_n^{K_n}(h) | dh 
		\ {+} 	{\int}| \pi_n^{K_n}(h) - p_{\textsc{sks}}^{n,K_n}(h)| dh 
			\  {+}  {\int}|  p_{\textsc{sks}}^n(h)-p_{\textsc{sks}}^{n,K_n}(h)| dh,  
		\end{aligned}
	\end{eqnarray}
	 the proof of Theorem~\ref{thm:1} reduces to study the behavior of the three above summands.
	
	As for the first term, Assumption \ref{cond:concentration} and a standard inequality of the \textsc{tv} norm yield 
	\begin{eqnarray} \label{bound:out:pi}
		\int| \pi_n(h) -  \pi_n^{K_n}(h) | dh \leq 2 \int_{h\, : \, \|h\| > M_n } \pi_n(h) dh = O_{P_0^n}(\delta_n^2).
	\end{eqnarray}
	
	Let us now deal with the third term via a similar reasoning. More specifically, leveraging the same \textsc{tv} inequality as above and the fact that $F( \alpha_{\eta}(h-\xi))$ within the expression for   $p_{\textsc{sks}}^{n}(h)$ in \eqref{limiting:distribution} is a univariate cdf meeting the condition $|F( \alpha_{\eta}(h-\xi))| \leq 1$, we have
	\begin{eqnarray*}
					 {\int}|  p_{\textsc{sks}}^n(h)-p_{\textsc{sks}}^{n,K_n}(h)| dh 
			\leq   4 \int_{ h \,:\, \|h\|>M_n } \phi_d(h;\xi,\Omega) dh.
	\end{eqnarray*}
	Define $P_{\xi,\Omega}(\|h\| > M_n) := $ $\smash{\int_{ h \in K_n^c } \phi_d(h;\xi,\Omega)dh}$ and let $A_{n,1} =  \{ \lambda_{\textsc{min}}( \Omega ) > \eta_1 \} \cap \{ \lambda_{\textsc{max}}( \Omega )$ $< \eta_2\} \cap \{ \|\xi\| < \tilde M_n\} $ for some sequence $\tilde M_n=o(M_n)$ going to infinity arbitrary slowly and some $\eta_1,\eta_2>0$. Moreover, notice that $V_{\theta_{*}}^n - \Omega^{-1}$ has entries of order $O_{P_{0}^n}(\delta_n)$. As a consequence, in view of Assumptions \ref{cond:LAN}--\ref{cond:prior} and Lemma \ref{lemma:eigen:J} in the Supplementary Materials, it follows that $P_0^n A_{n,1} = 1-o(1)$. 
	Conditioned on $A_{n,1}$, the eigenvalues of $\Omega$ lay on a positive bounded range. This, together with  \smash{$\tilde M_n/M_n \to 0$}, imply
	\begin{eqnarray} \label{thm1:help1}
		\begin{aligned}
			\quad P_0^n ( P_{\xi,\Omega}(\|h\| > M_n)/\delta_n^2 > \epsilon  ) \,  = &\, P_0^n ( \{ P_{\xi,\Omega}(\|h\| > M_n)/\delta_n^2 > \epsilon \} \cap A_{n,1}  ) + o(1)\\
			&\leq P_0^n ( e^{- \tilde c_1 M_n^2 }/\delta_n^2 > \epsilon  \mid A_{n,1}  ) + o(1) = o(1),
		\end{aligned}
	\end{eqnarray} 
	for every $\epsilon>0$, where $\tilde{c}_1$ is a sufficiently small positive constant and the last inequality follows from the tail behavior of the multivariate Gaussian for a sufficiently large choice of $c_0$ in $M_n = \sqrt{ c_0 \log \delta_n^{-1} }$. 
	This gives 
	\begin{eqnarray} \label{bound:out:skw} 
		  {\int}|  p_{\textsc{sks}}^n(h)-p_{\textsc{sks}}^{n,K_n}(h)| dh 
			\leq   4 \int_{ h \,:\, \|h\|>M_n } \phi_d(h;\xi,\Omega) dh = o_{P_0^n}(\delta_n^2).
	\end{eqnarray}
We are left to study the last summand ${\int}| \pi_n^{K_n}(h) - p_{\textsc{sks}}^{n,K_n}(h)| dh$. To this end, let us consider the event
	$ A_{n,2} =  A_{n,1} \cap \{ \Pi_n(K_h) > 0\} \cap\{ P_{\textsc{sks}}^n(K_n) > 0 \}.$
	Notice that $$P_0^n ( \Pi_n(K_h) > 0 ) = 1-o(1)$$ by Assumption \ref{cond:concentration}. Moreover, in view of \eqref{thm1:help1}, it follows that
 $P_0^n  ( P_{\textsc{sks}}^n(K_n) > 0) = P_0^n   (1 -   P_{\textsc{sks}}^n(K^c_n) > 0  ) \geq P_0^n   ( 1 - 2 P_{\xi,\Omega}(\|h\| > M_n) > 0  )  = 1 - o(1),$
    implying, in turn, $P_0^n A_{n,2} = 1 - o(1)$. 
	As a consequence, we can restrict our attention to
	\begin{eqnarray}
		\begin{aligned}
	\qquad		& \int| \pi_n^{K_n}(h) - p_{\textsc{sks}}^{n,K_n}(h)| dh \mathbbm{1}_{A_{n,2}} \\
			&=  \int \Big| 1 - \int_{K_n} \frac{p_{\textsc{sks}}^{n,K_n}(h)}{p_{\textsc{sks}}^{n,K_n}(h')} \frac{ p_{\theta_* + \delta_n h' }(X^n)\pi(\theta_* + \delta_n h')}{ p_{\theta_* + \delta_n h }(X^n)\pi(\theta_* + \delta_n h)} p_{\textsc{sks}}^{n,K_n}(h') dh'\Big| \pi_n^{K_n}(h) dh  \mathbbm{1}_{A_{n,2}}. 
		\end{aligned}
		\label{equ_13}
	\end{eqnarray}
	Now, note that, by the definition of $p_{\textsc{sks}}^{n,K_n}(h)$, we have  $ p_{\textsc{sks}}^{n,K_n}(h)/p_{\textsc{sks}}^{n,K_n}(h')= p_{\textsc{sks}}^{n}(h)/p_{\textsc{sks}}^{n}(h')$. This fact, together with an application of Jensen's inequality, implies that the quantity on~the right hand side  of Equation~\eqref{equ_13} can be upper bounded by
		\begin{eqnarray*}
			\begin{aligned}
				&\int_{K_n \times K_n} \Big| 1 - \frac{p_{\textsc{sks}}^{n}(h)}{p_{\textsc{sks}}^{n}(h')} \frac{ p_{\theta_* + \delta_n h' }(X^n)\pi(\theta_* + \delta_n h')}{ p_{\theta_* + \delta_n h }(X^n)\pi(\theta_* + \delta_n h)} \Big| \pi_n^{K_n}(h)p_{\textsc{sks}}^{n,K_n}(h') dh dh'  \mathbbm{1}_{A_{n,2}}. \\
			\end{aligned}
	\end{eqnarray*} 
	At this point, it is sufficient to recall Lemma \ref{lemma:1} within the Supplementary Materials and~the fact that $e^x = 1 + x + e^{\beta x}x^2/2$, for some $\beta \in (0,1)$, to obtain  
	\begin{eqnarray} 
			{\int}| \pi_n^{K_n}(h) - p_{\textsc{sks}}^{n,K_n}(h)| dh \mathbbm{1}_{A_{n,2}} &\leq& {\int_{K_n \times K_n}} | 1 - e^{r_{n,4}(h') - r_{n,4}(h)} | \pi_n^{K_n}(h)p_{\textsc{sks}}^{n,K_n}(h') dh dh'  \mathbbm{1}_{A_{n,2}} \nonumber \\
			& \leq& 2 |r_{n,4}| + 2 \exp(2\beta|r_{n,4}|)r_{n,4}^2 = O_{P_0^n}(\delta_n^2 M_n^{c_3}),	\label{upper:cond:Kn}	
	\end{eqnarray}  
	where $r_{n,4}  = \sup_{h \in K_n} r_{n,4}(h) $, and
	$c_3$ is some constant defined in Lemma \ref{lemma:1}. We conclude the proof by noticing that the combination of \eqref{triangle:tv},  \eqref{bound:out:pi}, \eqref{bound:out:skw} and \eqref{upper:cond:Kn}, yields Equation \eqref{thm1:out} in Theorem~\ref{thm:1}.  
\end{proof}


\subsection{Skew--symmetric approximations in regular parametric models} \label{sec_fixd_thm}
Theorem~\ref{thm:1}~states a general result under broad assumptions. In this section, we strengthen the methodological and practical impact of such a result by specializing the analysis to the context of, possibly misspecified and non--i.i.d., regular parametric models with $d$ fixed and $\delta_n = n^{-1/2}$.~The~focus on this practically--relevant setting crucially clarifies that Assumptions \ref{cond:uni}--\ref{cond:concentration} can be~readily verified under standard explicit conditions on the log--likelihood and log--prior derivatives, which in turn enter the definition of the \textsc{sks} parameters $\xi$, $\Omega$ and $\alpha_{\eta}(h-\xi)$. This~allows direct and closed--form derivation of $p_{\textsc{sks}}^n(h)$ in routine implementations of deterministic approximations for intractable posteriors induced by broad classes of parametric models. As stated in Corollary~\ref{corol:1}, in this setting the resulting \textsc{sks} approximating density achieves a remarkable improvement in the rates by a $\sqrt{n}$ factor, up to a poly--log term, relative to the classical Gaussian approximation. This accuracy gain can be proved both for the \textsc{tv} distance from the target posterior and also for the absolute error in the approximation for the posterior expectation of general polynomially bounded functions, with finite prior expectation.

Prior to stating Corollary~\ref{corol:1}, let us introduce a number of explicit assumptions which allow to specialize the general theory in Section \ref{sec_general_thm} to the setting with $d$ fixed and $\delta_n = n^{-1/2}$. As discussed in the following, Assumptions \ref{cond:1}--\ref{cond:2} provide natural and verifiable conditions which ensure that the general Assumptions \ref{cond:LAN}--\ref{cond:concentration}  are met, thereby allowing Theorem~\ref{thm:1} to be applied, and specialized, to the regular parametric models setting. 

\begin{assumption}	\label{cond:1} Define $ \ell^{(4)}_{\theta_*,stlk}(h):=\ell_{stlk}^{(4)}(\theta_*+h/\sqrt{n}) $. Then, the log--likelihood~of~the, possibly misspecified,  model is four times differentiable at $\theta_*$ with 
	\vspace{-5pt}
	\begin{eqnarray*}
		\ell^{(1)}_{\theta_*,s} =O_{P_{0}^{n}}(n^{1/2}), \quad  \ell^{(2)}_{\theta_*,st} = O_{P_{0}^{n}}(n),  \quad \ell^{(3)}_{\theta_*,stl} = O_{P_{0}^{n}}(n), \quad \mbox{for } \ s,t,l=1, \ldots, d,
	\end{eqnarray*}
	and \smash{$\sup_{h \in K_n} |\ell^{(4)}_{\theta_*,stlk}(h)|= O_{P_{0}^{n}}(n)$},  for $s,t,l, k=1, \ldots, d$. 
\end{assumption}
\begin{assumption}
	\label{cond:2}   The entries of the expected Fisher information matrix satisfy $i_{st}=~O(n)$ while $j_{st}/n - i_{st}/n =\smash{ O_{P_{0}^{n}}(n^{-1/2})}$, for $s,t = 1,\dots,d.$ 
	Moreover, there exist two positive constants  $\eta_1$ and $\eta_2$ such that $  \lambda_{\textsc{min}}(I_{\theta_*} / n) > \eta_1$ and $  \lambda_{\textsc{max}}(I_{\theta_*} / n) < \eta_2$.
\end{assumption}
\begin{assumption}
	\label{cond:3}  The log--prior density $ \log \pi(\theta)$ is two times continuously differentiable in a neighborhood of $\theta_*$, and $0 < \pi(\theta_*) < \infty$. 	
\end{assumption}
\begin{assumption}
	\label{cond:4} For every sequence $M_n \to \infty$ there exists a constant $c_5>0$ such that 
	$  \lim_{n \to \infty } P_0^{n}  \lbrace \sup_{\|\theta - \theta_*\|> M_n/\sqrt{n}}\{ (\ell(\theta) - \ell(\theta_*))/n  \} < - c_5 M_n^2/n \rbrace = 1.$
\end{assumption}

Assumptions \ref{cond:1}--\ref{cond:2} are mild and considered standard in classical frequentist theory \citep[see~e.g.,][p. 347]{salvan1997}. In Lemma \ref{lemma:taylor:fixed:d} we show that these conditions allow to control with precision the error in the Taylor approximation of the log--likelihood.
 Assumption \ref{cond:3} is also mild and is satisfied by several priors that are commonly used in practice. Such~a~condition allows to consider a first order Taylor expansion for the log--prior of the form
 \begin{eqnarray} \label{taylor:lprior:star}
 	\log \pi(\theta) = 	\log \pi(\theta_*) + \log \pi^{(1)}_{\theta_*,s} h_s/\sqrt{n} + r_{n,2}(h), 
 \end{eqnarray}
 with $  r_{n,2} :=  \sup_{h \in K_n} r_{n,2}(h) = O(M_n^2/n)$.
Finally, Assumption \ref{cond:4} is required to control the 
the rate of contraction of the, possibly misspecified, posterior distribution into $K_n$. In other modern versions of Berstein--Von Mises type results, such an assumption is usually replaced by conditions on the existence of a suitable sequence of tests. Sufficient conditions for the correctly--specified case can be found, for example, in \citet{van2000asymptotic}. In the misspecified setting, assumptions ensuring the existence of these tests have been derived by \citet{kleijn2012bernstein}. Another possible option is to assume, for every $\delta>0$, the presence of a positive constant $c_{\delta}$ such that
\begin{eqnarray} \label{cond:koers2023} 
	\lim_{n \to \infty } P_0^{n} \lbrace {\textstyle \sup_{\|\theta - \theta_*\|> \delta}}\{ (\ell(\theta) - \ell(\theta_*))/n  \} < - c_{\delta} \rbrace = 1.
\end{eqnarray}
In the misspecified setting, condition \eqref{cond:koers2023} is considered by, e.g., \citet{koers2023}.
Assumption \ref{cond:4} is a slightly more restrictive version of \eqref{cond:koers2023}. In fact, Lemma \ref{lemma:cond4} below shows that it is implied by mild and readily--verifiable sufficient conditions. 

Under Assumption \ref{cond:uni} and  \ref{cond:1}--\ref{cond:4}, Corollary \ref{corol:1} below clarifies that Theorem \ref{thm:1} holds for a general class of \textsc{sks} distributions yielding \textsc{tv} rates in approximating the exact posterior of order $O_{P_0^n}(M_n^{c_6}/n )$, with $c_6>0$ and $M_n = \sqrt{c_0 \log n}$. Furthermore,  the  \textsc{sks} parameters are  defined~as explicit functions of the log--prior and log--likelihood derivatives. As stated in Equation~\eqref{corol1:out2} of Corollary \ref{corol:1}, the same rates can be derived also for the absolute error in  approximating the posterior expectation of general polynomially bounded functions.

\begin{corollary} \label{corol:1}
	Let $h = \sqrt{n}(\theta - \theta_*),$ and define  $M_n = \sqrt{c_0 \log n}$, with $c_0>0$. Then, under Assumptions \ref{cond:uni} and \ref{cond:1}--\ref{cond:4} it holds
	\begin{eqnarray} \label{corol1:out1} 
		\| \Pi_n(\cdot) - P_{\textsc{sks}}^n(\cdot) \|_{\textsc{tv}} = O_{P_{0}^n}(M_n^{c_6}/n).
	\end{eqnarray}
	where $c_6 >0$, and $P_{\textsc{sks}}^n(\cdot)$ is the cdf of the \textsc{sks} density $p_{\textsc{sks}}^n(h)$ in  \eqref{limiting:distribution} with parameters 
		\begin{eqnarray*}
	\begin{split}
		&\xi=\smash{[n (J^{-1}_{\theta_*})_{st}u_t]}, \qquad \Omega^{-1} = [ j_{st}/n-(\xi_{l}\ell^{(3)}_{\theta_*,stl}/n)/\sqrt{n} ],\\
	&\alpha_{\eta}(h-\xi) ={\{}1/(12 \eta \sqrt{n})\}  (\ell^{(3)}_{\theta_*,stl}/n) \{(h-\xi)_{s}(h-\xi)_{t}(h-\xi)_{l}+3(h-\xi)_{s} \xi_{t} \xi_{l}\}.
	\end{split}
	\end{eqnarray*}
where 	$\smash{u_t=(\ell^{(1)}_{\theta_{*}} +  \log \pi^{(1)}_{\theta_{*}} )_t/\sqrt{n}}$ for $t=1, \ldots, d$.	 The function $F(\cdot)$ entering the definition of $p_{\textsc{sks}}^n(h)$  in  \eqref{limiting:distribution}  denotes any univariate cdf  which satisfies $F(-x) = 1 - F(x) $ and $F(x) = 1/2 + \eta x + O(x^2)$, for some $\eta \in \mathbb{R}$, when $x \to 0$. In addition, let $G \, : \, \mathbbm{R}^d \to R $ be a function satisfying $ | G | \lesssim \|h\|^r $. If the prior is such that $\int \|h\|^r \pi(\theta + h/\sqrt{n}) dh < \infty$
	then 
	\begin{eqnarray} \label{corol1:out2}
	\int  G(h) | \pi_n(h) - p_{\textsc{sks}}^n(h) | dh = O_{P_0^n}(M_n^{c_6 + r}/n).
	\end{eqnarray} 
 with $p_{\textsc{sks}}^n(h)$ denoting  the skew--symmetric approximating density defined above.
\end{corollary}

\vspace{-15pt}

\begin{remark}\label{rem_0}
	As for Theorem \ref{thm:1}, under conditions similar to those required by Corollary \ref{corol:1}, it is possible to show that the \textsc{tv} distance between the posterior and the Gaussian approximation dictated by the classical Bernstein--von Mises theorem is $O_{P_{0}^n}(M_n^{c_7}/\sqrt{n})$ for some fixed $c_7 >0$. Therefore, the improvement in rates achieved by the proposed \textsc{sks} approximation is by a remarkable $\sqrt{n}$ factor, up to a poly--log term. As~illustrated in Figure~\ref{fig:tv_intui}, this implies that the \textsc{sks} solution is expected to substantially improve, in practice, the accuracy of the classical Gaussian in approximating the target posterior, while inheriting its limiting frequentist properties. Intuitively, the rates we derive suggest that the proposed \textsc{sks} approximation can possibly attain with a $n \approx \sqrt{\bar{n}}$ sample~size the same accuracy obtained by its Gaussian counterpart with a sample size of $\bar{n}$. The empirical studies in Sections~\ref{sec_24} and~\ref{sec_32} confirm this intuition, which is further strengthened in Section~\ref{sec_3} through the derivation of non--asymptotic upper bounds for the practical skew--modal approximation, along with novel lower bounds for the classical Gaussian from the Laplace method.
\end{remark}

\begin{remark}
Equation~\eqref{corol1:out2} confirms that the improved rates  hold also when the focus is on the error in approximating the posterior expectation of generic polynomially bounded functions. More specifically, notice that, by direct application of standard properties of integrals, the proof of Equation~\eqref{corol1:out2} in the Supplementary Materials, implies
	\begin{eqnarray} \label{corol1:out3}
	\textstyle \big| \int  G(h)\pi_n(h) dh - \int  G(h)p_{\textsc{sks}}^n(h) dh | = O_{P_0^n}(M_n^{c_6 + r}/n).
	\end{eqnarray} 
This clarifies that the skewed Bernstein--von Mises type result in~\eqref{corol1:out1} has important methodological and practical consequences  that point toward remarkable improvements in the approximation of posterior functionals of direct interest for inference (e.g., moments).
\end{remark}

\begin{figure}[t!]
	\centering
		\includegraphics[trim={0.98cm 0 0 0},clip, width=0.86\textwidth]{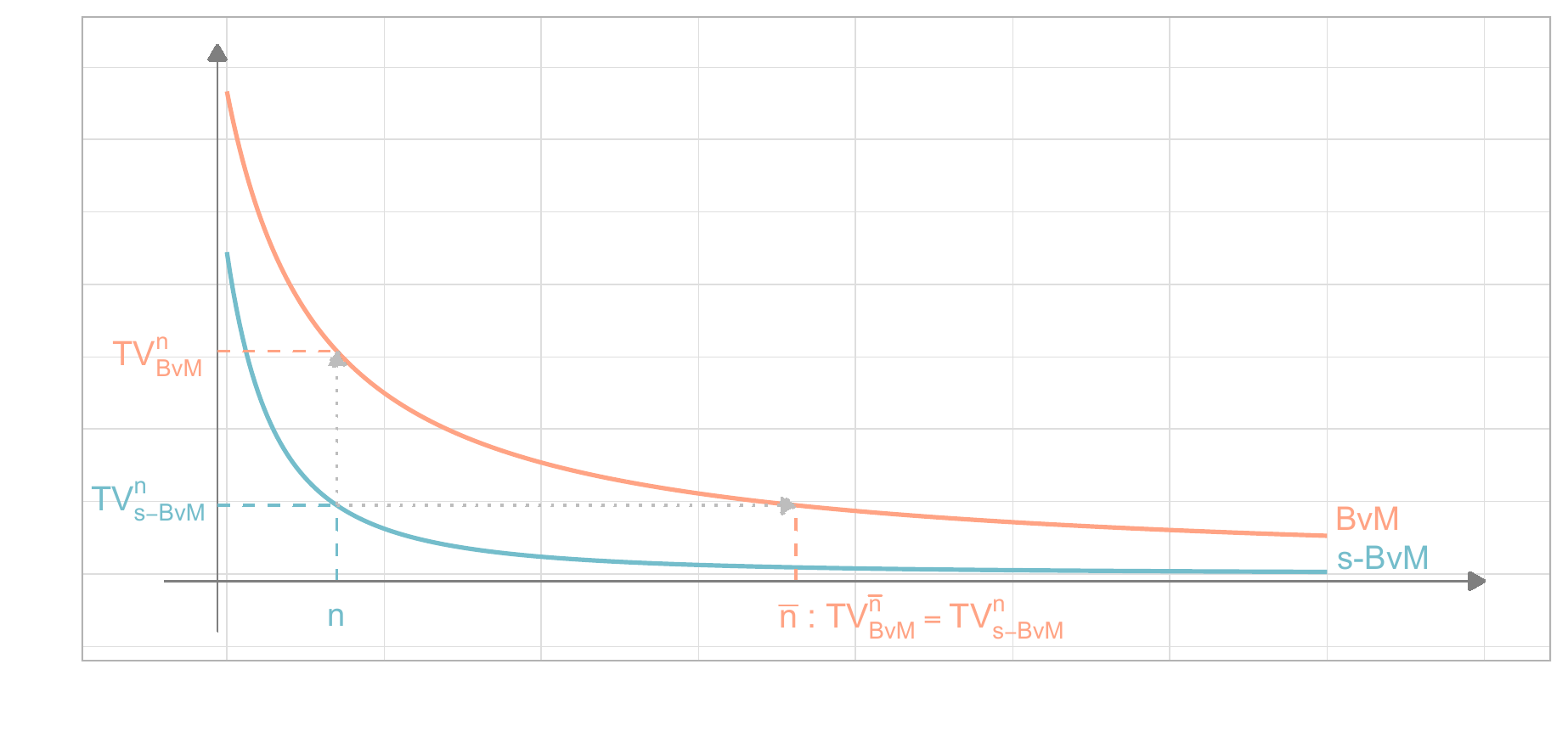}
		\vspace{-20pt}
		\caption{ \footnotesize Illustrative graphical comparison among the \textsc{tv}--rates achieved by the proposed skew--symmetric approximation (\textsc{s--BvM}) and those derived under the classical Bernstein--von Mises theorem based on Gaussians (\textsc{BvM}). For a given $n$, we show the  \textsc{tv} distances $\textsc{tv}_{\textsc{s--BvM}}^n$ and $\textsc{tv}_{\textsc{BvM}}^n$ expected under the  \textsc{s--BvM} and  \textsc{BvM} rates, respectively. We further highlight the expected sample size $\bar{n} \gg n$  required by \textsc{BvM} to attain the same \textsc{tv} distance as the one achieved by \textsc{s--BvM} with the original $n$, under the derived rates. The empirical studies in Sections~\ref{sec_24} and~\ref{sec_32} support this graphical intuition.}	
		\vspace{-10pt}
	\label{fig:tv_intui}
\end{figure}

The proof of Equation~\eqref{corol1:out2} can be found in the Supplementary Materials. As for the~main result in Equation~\eqref{corol1:out1}, it is sufficient to apply  Theorem \ref{thm:1}, after ensuring that its Assumptions \ref{cond:uni}--\ref{cond:concentration} are implied by \ref{cond:uni} and \ref{cond:1}--\ref{cond:4}. Section~\ref{sec_logp} presents two key results (see Lemma~\ref{lemma:taylor:fixed:d} and Lemma \ref{lemma:post:contr}) which address this point. Section~\ref{sec_logp1} introduces instead simple and verifiable conditions which ensure the validity of Assumption~\ref{cond:4}.

\subsubsection{Log--posterior asymptotics and posterior contraction}\label{sec_logp}
In order to move from the general theory within Theorem \ref{thm:1} to the specialized setting considered in Corollary \ref{corol:1}, two key points are the rate at which the target posterior concentrates in $K_n = \{h \, : \, \|h\| < M_n\}$ and the behavior of the Taylor expansion of the log--likelihood within $K_n$. Here we show~that Assumptions \ref{cond:uni} and \ref{cond:1}--\ref{cond:4} are indeed sufficient to obtain the required guarantees.
Lemma~\ref{lemma:taylor:fixed:d} below establishes that, under Assumptions \ref{cond:1}--\ref{cond:2}, the error introduced by replacing the log--likelihood with its third--order Taylor approximation is uniformly of order $M_n^4/n$ on $K_n$.

\begin{lemma} \label{lemma:taylor:fixed:d}
	Under Assumptions \ref{cond:1}--\ref{cond:2}, it holds in $K_n = \{ h \, : \,  \|h\| \leq M_n \}$ that  
	\begin{eqnarray} \label{taylor:llik:star} 
		\log \frac{p_{\theta_* + h/\sqrt{n} }^n}{p_{\theta_*}^n}(X^n) - h_s v_{st}^n \Delta_{\theta_*, t}^n + \frac{1}{2}  v_{st}^n h_s h_{t} - \frac{1}{6 \sqrt{n}} a^{(3),n}_{\theta_*,stl} h_s h_t h_l  = r_{n,1}(h),
	\end{eqnarray}
	with $\Delta_{\theta_*, t}^n = j^{-1}_{st} \sqrt{n} \ell^{(1)}_{\theta_{*},s} = O_{P_0^n}(1),$ $v_{st}^n = j_{st}/n = O_{P_0^n}(1)$ and $a^{(3),n}_{\theta_*,stl} = \ell^{(3)}_{\theta_*,stl}/n= O_{P_0^n}(1)$. Moreover,	  \smash{$r_{n,1} := \sup_{h \in K_n}\left| r_{n,1}(h) \right|=O_{P_0^n}(M_n^{4}/n)$}.
\end{lemma}

Lemma~\ref{lemma:post:contr} shows that under the above conditions it is then possible to select $c_0$ in $M_n = \sqrt{c_0 \log n}$, and hence $K_n$, such that the posterior distribution concentrates its mass in $K_n$, at any polynomial rate, with $P_0^n$ probability tending to 1 as $n \to \infty$. This in turn implies that Assumption \ref{cond:concentration} in Section \ref{sec_general_thm} is satisfied.

\begin{lemma}[Posterior contraction] \label{lemma:post:contr}
	Under Assumptions \ref{cond:1}--\ref{cond:4}, there exists a sufficiently--large $c_0> 0 $  in $M_n = \sqrt{c_0 \log n}$, such that for every $D>0$ it holds 
	\begin{eqnarray*}
		\lim_{n \to \infty } P_0^n \{ \Pi_n( K_n^c ) < n^{-D} \} = 1,  
	\end{eqnarray*} 
where  $K_n^c$ is the complement of $K_n$.
\end{lemma}

To complete the connection between Theorem \ref{thm:1} and Corollary \ref{corol:1}, note that Assumption \ref{cond:3} implies that Assumption \ref{cond:prior} in Section \ref{sec_general_thm}  is verified with \smash{$\log \pi^{(1)} = (\partial/\partial \theta ) \log \pi(\theta)_{| \theta = \theta_*}$}. Refer to the Supplementary Materials for the proofs of Lemma~\ref{lemma:taylor:fixed:d} and Lemma \ref{lemma:post:contr}.
 

\subsubsection{Sufficient conditions for Assumption \ref{cond:4}}\label{sec_logp1}
Lemma \ref{lemma:post:contr} requires the fulfillment of Assumption \ref{cond:4} which allows precise control on the behavior of the log--likelihood ratio outside the set $K_n$. In addition, Assumption \ref{cond:4} further plays a crucial role in the development of the practical skew--modal approximation presented in Section \ref{sec_3}. Given the relevance of such an assumption, Lemma \ref{lemma:cond4} provides a set of natural and verifiable  sufficient conditions that guarantee its validity; see the Supplementary Materials for the proof of  Lemma \ref{lemma:cond4}.
      
\begin{lemma}\label{lemma:cond4}
	Suppose that Assumptions \ref{cond:uni} and \ref{cond:2} hold, and that for every $\delta > 0$ there exists a positive constant $c_{\delta}$ such that 
	\begin{eqnarray} \label{lemma:5:cond:1} 
	\textstyle	\lim_{n \to \infty } P_0^{n} \lbrace \sup_{\|\theta - \theta_*\|> \delta}\{(\ell(\theta) - \ell(\theta_*))/n  \} < - c_{\delta} \rbrace = 1.
	\end{eqnarray}
	If there exist $\bar{n} \in \mathbbm{N}$ and $\delta_1>0$ such that, for all $n > \bar n$, it holds
	\begin{itemize}
		\item R1)  \label{cond:lemma5:1} $ \E_{0}^n \{ \ell(\theta) - \ell(\theta_*)  \} / n $ is concave in $\{ \theta \, : \, \| \theta - \theta_* \|< \delta_1 \}$, two times differentiable at $\theta_*$ with negative Hessian equal to the expected Fisher information matrix 
		$I_{\theta_*}/n,$ 
		\vspace{-7pt}
		\item R2) \label{cond:lemma5:2} and
		$\sup_{0<\|\theta - \theta_*\| < \delta_1} [ \{ \ell(\theta) - \ell(\theta_*)\} - \E_{0}^n \{ \ell(\theta) - \ell(\theta_*)  \}  ]/(n\|\theta - \theta_*\|)    =  O_{P_0^n}(n^{-1/2}),$	
	\end{itemize}
	then there is a constant $c_5>0$ such that 
	$  \lim_{n \to \infty } P_0^{n} \lbrace \sup_{\|\theta - \theta_*\|> M_n/\sqrt{n}}\{ (\ell(\theta) - \ell(\theta_*))/n  \} < - c_5 M_n^2/n \rbrace = 1$,  for any $M_n \to \infty$.
\end{lemma}

As highlighted at the beginning of Section~\ref{sec_fixd_thm}, condition \eqref{lemma:5:cond:1} is mild and can be~found~both in classical  \citep[e.g.,][]{lehmann2006theory}  and modern \citep{koers2023}  Bernstein--von Mises type results. Condition R1 requires the expected log--likelihood to be sufficiently regular within a neighborhood of $\theta_*$ and is closely related to standard~assumptions on M--estimators \citep[e.g.,][Ch. 5]{van2000asymptotic}. Finally, among the assumptions of Lemma \ref{lemma:cond4}, R2 is arguably the most specific. It requires that, for all $\theta$ such that $0<\|\theta-\theta_*\|<\delta_1$, the quantity $ [ \{ \ell(\theta) - \ell(\theta_*)  \} - \E_{0}^n \{ \ell(\theta) - \ell(\theta_*)  \}  ]/(n\|\theta - \theta_{*}\|)$ converges uniformly to zero in probability with rate $n^{-1/2}$. This type of behavior is common in several routinely--implemented statistical models, such as generalized linear models.   


\section{Empirical results} \label{sec_24}
Sections~\ref{sec_241}--\ref{sec_242} provide empirical evidence of the improved accuracy achieved by the \textsc{sks} approximation~in~Corollary \ref{corol:1} (\textsc{s--BvM}), relative to its Gaussian counterpart (\textsc{BvM}) arising from the classical Bernstein--von Mises theorem~in regular parametric models. We study both correctly--specified and misspecified~settings, and focus not only on assessing the superior performance of the \textsc{sks} approximation, with~$F(\cdot)=\Phi(\cdot)$, but~also on quantifying whether the remarkable improvements encoded in the rates we~derived under asymptotic arguments find empirical evidence also in finite--sample studies.~To this~end,  \textsc{s--BvM} and \textsc{BvM} are compared both in terms of \textsc{tv} distance from the target posterior and also with respect to the absolute error in approximating the posterior mean. These two measures illustrate the practical implications of the~rates~derived in Equation \eqref{corol1:out1} and \eqref{corol1:out2}, respectively. Since for the illustrative studies in Sections~\ref{sec_241} and \ref{sec_242} the target posterior can be derived in closed form, the \textsc{tv} distances $ \textsc{tv}^{n}_{\textsc{BvM}} =(1/2)\int_{\mathbbm{R}}|\pi_n(h)-p^{n}_{\textsc{gauss}}(h)| dh$ and $\textsc{tv}^{n}_{\textsc{s--BvM}}=(1/2)\int_{\mathbbm{R}}|\pi_n(h)-p^{n}_{\textsc{sks}}(h)| dh$ can be evaluated numerically, for every $n$, via standard routines in $\texttt{R}$. The same holds for the  errors in posterior mean approximation $ \textsc{fmae}^{n}_{\textsc{BvM}} =|\int_{\mathbbm{R}} h \{ \pi_n(h)-p^{n}_{\textsc{gauss}}(h)\}dh|$ and $\textsc{fmae}^{n}_{\textsc{s--BvM}}=|\int_{\mathbbm{R}} h\{\pi_n(h)-p^{n}_{\textsc{sks}}(h)\}dh|$. 

Note that, as for other versions of the classical Bernstein--von Mises theorem, also our~theoretical results in~Sections~\ref{sec_general_thm} and \ref{sec_fixd_thm} require knowledge of the Kullback--Leibler minimizer between the true data--generating process and the parametric family $\mathcal{P}_{\Theta}$. Since $\theta_*$ is  clearly unknown in practice, in Section~\ref{sec_3} we address this aspect via a practical plug--in version of the \textsc{sks} approximation in~Corollary \ref{corol:1}, which replaces  $\theta_*$ with its maximum  a posteriori estimate. This yields a readily--applicable skew--modal approximation with similar theoretical and empirical support; see the additional simulations and real--data applications in~Section~\ref{sec_32}.

\vspace{-8pt}
\subsection{Exponential model} \label{sec_241}
Let $X_i\stackrel{iid}{\sim}\textsc{exp}(\theta_0)$, for $i=1, \ldots, n$, where $\textsc{exp}(\theta_0)$ denotes the exponential distribution with rate parameter $\theta_0 = 2.$ In the following, we consider a correctly specified model having exponential likelihood and a $\textsc{exp}(1)$ prior for $\theta$.  To obtain the skew--symmetric approximation for the posterior distribution induced by such a Bayesian model, let us first  verify that all conditions of Corollary \ref{corol:1} hold.

To address this goal first note that, as the model is correctly specified, Assumption~\ref{cond:uni}~holds with $\theta_* = \theta_0$. The first four derivates of the log--likelihood at $\theta$ are  \smash{$n/\theta - \sum_{i=1}^n x_i,$} \smash{$ -n/\theta^2$},  \smash{$ 2n/\theta^3$} and \smash{$ -6n/\theta^4$}, respectively. Hence, Assumptions \ref{cond:1}--\ref{cond:2} are both satisfied, even around a small neighborhood of $\theta_0$. Assumption \ref{cond:3} is met by a broad class of routinely--implemented priors. For example, $\textsc{exp}(1)$ can be considered in such a case. Finally, we need to check Assumption \ref{cond:4}. To this end, note that $\{\ell(\theta) - \ell(\theta_0)\}/n = \log \theta/\theta_0 + (\theta_0 - \theta)\sum_{i = 1}^n x_i/n$ which, by the law of large number, converges in probability to a negative constant for every fixed $\theta$ implying \eqref{lemma:5:cond:1}. Additionally, $\E_{0}^n \{\ell(\theta) - \ell(\theta_0)\}/n = \log \theta/\theta_0 + (1 - \theta/\theta_{0})$ is concave in $\theta$ and, therefore, it fulfills condition R1 of Lemma \ref{lemma:cond4}. Since $[\{\ell(\theta) - \ell(\theta_0)\} -\E_{0}^n \{\ell(\theta) - \ell(\theta_0)\}]/n = (\theta_0 - \theta)(\sum_{i = 1}^n x_i/n - 1/\theta_0)$ also condition R2 in Lemma \ref{lemma:cond4} is satisfied and, as a consequence, also Assumption \ref{cond:4}. 

\begin{table}[t] 
	\renewcommand{\arraystretch}{1}
	\centering
	\caption{\footnotesize Empirical comparison, averaged over $50$ replicated studies, between the  classical (\textsc{BvM}) and skewed (\textsc{s--BvM}) Bernstein--von Mises theorem in the correctly--specified exponential example. The first table shows, for different sample sizes from $n=10$ to $n=1500$, the  log--\textsc{tv} distances (\textsc{tv}) and log--approximation errors for the posterior mean  (\textsc{fmae}) under \textsc{BvM} and \textsc{s--BvM}. The bold values indicate the best performance for each $n$. The second table shows,  for $n$ from $n=10$ to $n=100$, the sample size $\bar{n}$ required by the classical Gaussian \textsc{BvM} to achieve the same \textsc{tv} and \textsc{fmae} attained by the proposed \textsc{sks} approximation with that $n$.} 
	\small
	\begin{adjustbox}{max width=\textwidth}
		\begin{tabular}{lrrrrrr}
			\hline
			\quad & $n=10$  &  $n=50$  &   $n=100$ & $n=500$  &  $n=1000$ &  $n=1500$  \\  
			\hline
			$\log \textsc{tv}^n_{\textsc{BvM}}$ \quad & $-1.67$& $-2.50$& $-2.82$& $-3.59$& $-3.98$& $-4.18$\\ 
			$\log \textsc{tv}^n_{\textsc{s--BvM}}$ \quad  &$\bf -2.53$ & $ \bf -3.86$& $\bf -4.41$& $\bf -5.76$& $\bf -5.58$&$\bf-6.58 $\\
			\hline
			$\log \textsc{fmae}^n_{\textsc{BvM}}$ \quad &$-0.90$& $-1.77$& $-1.97$& $-2.85$& $-3.21$& $-3.33$ \\ 
			$\log \textsc{fmae}^n_{\textsc{s--BvM}}$ \quad  &$\bf -1.07$ &$\bf -2.81$& $\bf -3.74$ &$\bf -6.14$ &$ \bf -7.09$ & $\bf -7.42$  \\ 
		\hline
		\end{tabular}
	\end{adjustbox}
	\vspace{-6pt}
	\label{tab}
\end{table}
\begin{table}[t]
	\centering
	\begin{tabular}{lrrrrrrr}
		\hline
		&  $n=10$ & $n=15$ & $n=20$& $n=25$ & $n=50$ & $n=75$ & $n=100$ \\ 
		\hline
	$\bar{n}: \ \textsc{tv}^{\bar{n}}_{\textsc{BvM}}=\textsc{tv}^{n}_{\textsc{s--BvM}} $& 55 & 120 & 250 & 350 & 820 & 1690 & 2470 \\ 
	$\bar{n}: \ \textsc{fmae}^{\bar{n}}_{\textsc{BvM}}=\textsc{fmae}^{n}_{\textsc{s--BvM}} $  & 15 & 25 & 70 & 110 & 380 & 1050 & 2280 \\ 
		\hline
	\end{tabular}
	\vspace{-4pt}
\end{table}

The above results ensure that Corollary \ref{corol:1} holds and can be leveraged to derive the~parameters of the \textsc{sks} approximation in \eqref{limiting:distribution} under this exponential example. To this end, first notice that, since the prior distribution is an  $\textsc{exp}(1)$, then  \smash{$\log \pi^{(1)}_{\theta_0}=-1$}. As a consequence, $\xi=\theta^2_0( n/\theta_0 - \sum_{i=1}^n x_i -1)/\sqrt{n}$ and \smash{$\Omega=1/(\theta_0^{-2}- 2{\theta_0^{-1}} \lbrace 1/\theta_{0} -(\sum_{i=1}^n x_i)/n -1/n \rbrace)$}. For what concerns the skewing factor, we choose $F(\cdot) = \Phi(\cdot)$ which implies a cubic function equal to \smash{$\alpha_{\eta}(h-\xi)=\smash{\{\sqrt{2\pi}/(6 \sqrt{n} \theta_0^3)\}\{(h-\xi)^3+3(h-\xi)\xi^2 \}}$}.

Table~\ref{tab} compares the accuracy of the skew--symmetric (\textsc{s--BvM}) and the Gaussian (\textsc{BvM}) approximations corresponding to the newly--derived and classical Bernstein--von Mises theorems, respectively, under growing sample size and replicated experiments. More specifically, we consider 50 different simulated datasets with $\theta_0 = 2$ and sample size $n_{\textsc{tot}}=1500$. Then, within each of these 50 experiments, we derive the target posterior under several subsets of data $x_1, \ldots, x_{n}$ with a growing sample size $n$, and then compare the accuracy of the two approximations under the $\textsc{tv}$ and $\textsc{fmae}$ measures discussed at the beginning of Section~\ref{sec_24}. The first part of Table~\ref{tab} displays, for each $n$, these two measures averaged across the 50 replicated experiments under both \textsc{s--BvM} and \textsc{BvM}. The empirical results confirm that the \textsc{sks} approximation yields remarkable accuracy improvements over the Gaussian counterpart for any $n$. Such an  empirical finding  clarifies that the $\sqrt{n}$ improvement encoded in the rates we derive, is visible also in finite, even small, sample size settings. This suggests that the theory in Sections~\ref{sec_general_thm}--\ref{sec_fixd_thm}  is informative also in practice, while motivating the adoption of the \textsc{sks}  approximation in place of the Gaussian one. Such a result is further strengthened in the second part of Table~\ref{tab}, which shows that to attain the same accuracy achieved by the proposed  \textsc{sks}  approximation  with a given $n$, the classical Gaussian counterpart requires a sample size $\bar{n}$ higher by approximately one order of magnitude.


\vspace{-5pt}

\subsection{Misspecified exponential model} \label{sec_242}
Section~\ref{sec_242} deals with a correctly--specified model where $P_0^n \in \mathcal{P}_{\Theta}$. Since  Corollary \ref{corol:1} holds even when the model $\mathcal{P}_{\Theta}$ is misspecified, it is of interest to compare the accuracy of the proposed \textsc{sks} approximation and the Gaussian one also within this context. To this end, let us consider the case \smash{$X_i\stackrel{iid}{\sim}\textsc{L--Norm}(-1.5,1)$}, for $i=1, \ldots, n$, where $\textsc{L--Norm}(-1.5,1)$ denotes the log--normal distribution with parameters $\mu = -1.5 $ and $\sigma = 1$. As in Section~\ref{sec_241}, an exponential likelihood is assumed, parameterized by the rate parameter $\theta$, and the prior is $\textsc{exp}(1)$. In this misspecified context, the minimizer of the Kullback--Leibler divergence between the log--normal distribution and the family of exponential distributions is unique and equal to $\theta_* \approx 2.71$. Similarly to Section \ref{sec_241} one~can show that the conditions of Corollary \ref{corol:1} are still satisfied, thus allowing the derivation of the same \textsc{sks}  approximating density with parameters evaluated at   $\theta_*$ instead of $\theta_0$. 

The quality of the \textsc{s--BvM} and \textsc{BvM} approximations is studied under the same measures and settings considered in Section~\ref{sec_242}. The results reported in Table~\ref{tab01} of the Supplementary Materials are in line with those for the correctly--specified case in Table~\ref{tab}.


\vspace{5pt}

\section{Skew--modal approximation} \label{sec_3}
As for standard theoretical derivations of Bernstein--von Mises type results, also our theory in Section~\ref{sec_2} studies approximating densities whose parameters depend on the minimizer $\theta_*$ of $\textsc{kl}(P_0^n || P_\theta^n)$ for $\theta \in \Theta$, which coincides with $\theta_0$ when the model is correctly--specified. Such a quantity is unknown in practice. Hence,~to provide an effective alternative to the classical Gaussian--modal approximation, which can be implemented in practical contexts, it is necessary to replace $\theta_*$ with a suitable estimate. To this end, in Section~\ref{sec_31} we consider a simple, yet effective, plug--in version of the \textsc{sks} density in Corollary \ref{corol:1} which replaces $\theta_*$ with its maximum a posteriori (\textsc{map}) estimator, without losing the theoretical accuracy guarantees. Note that, in general, $\theta_*$ can be replaced by any efficient estimator. However, by relying on the \textsc{map} several quantities simplify, giving raise to a highly tractable and accurate solution, which is named skew--modal approximation. When the focus is on approximating posterior marginals, Section~\ref{sec_marginal} further derives a theoretically supported, yet more scalable,  skew--modal approximation for such quantities.

\vspace{-5pt}

\subsection{Skew--modal approximation and theoretical guarantees} \label{sec_31}
Consistent with the above discussion, we consider the plug--in version \smash{$\hat{p}^n_{\textsc{sks}}(\hat h)$  of $p^n_{\textsc{sks}}(h)$} in Equation~\eqref{limiting:distribution}, where the unknown $\theta_*$ is replaced by the \textsc{map} \smash{$\hat{\theta}=\mbox{argmax}_{\theta \in \Theta} \{\ell(\theta)+\log \pi(\theta)\}$}. This yields the skew--symmetric density, for the rescaled parameter \smash{$\hat{h}=\sqrt{n}(\theta-\hat{\theta})\in \mathbb{R}^d$}, defined as
\begin{eqnarray}\label{limiting:distribution:map}
	\hat{p}_{\textsc{sks}}^n( \hat{h} ) \, = \, 2 \phi_d ( \hat{h} ; 0, \hat{\Omega} ) \hat{w}( \hat{h} )=2 \phi_d ( \hat{h} ; 0, \hat{\Omega} )F(\hat{\alpha}_{\eta}(\hat h)),
\end{eqnarray}
where \smash{$\hat{\Omega}=(\hat{V}^n)^{-1}$ with $\hat{V}^n = [ \hat{v}_{st}^n ] = [ {j}_{\hat{\theta},s t}/n ]\in\mathbb{R}^{d\times d},$} while the skewing function entering the univariate cdf $F(\cdot)$ is defined as \smash{$\hat{\alpha}_{\eta}(\hat h) \, = \{1/(12 \eta \sqrt{n})\}  (\ell^{(3)}_{\hat \theta,stl}/n) \hat{h}_{s}\hat{h}_{t}\hat{h}_{l} \in \mathbb{R}$}. 

Relative to the expression for \smash{$p^n_{\textsc{sks}}(h)$} in Equation~ \eqref{limiting:distribution}, the location parameter $\hat \xi$ is zero in \eqref{limiting:distribution:map}, since \smash{$\hat{\xi}$} is a function of  the quantity \smash{$(\ell^{(1)}_{\hat{\theta}} +  \log \pi^{(1)}_{\hat{\theta}} )/\sqrt{n}$} which is zero by definition when \smash{$\hat{\theta}$} is the \textsc{map}. For the same reason, unlike its population version defined below \eqref{corol1:out1}, in the expression for the precision matrix of the Gaussian density factor in \eqref{limiting:distribution:map} the additional term including the third order derivative disappears. Therefore, approximation~\eqref{limiting:distribution:map} does not introduce additional complications in terms of positive--definiteness~and non--negativity of the precision matrix relative to those of the classical Gaussian--modal approximation.

Equation~\eqref{limiting:distribution:map} provides a practical skewed approximation of the exact posterior centered at the mode. For this reason, such a solution is referred to as skew--modal approximation.  In order to provide theoretical guarantees for this practical version, similar to those in Corollary \ref{corol:1}, while further refining these guarantees through novel non--asymptotic bounds,  let~us~introduce two  mild assumptions in addition to those outlined in Section~\ref {sec_fixd_thm}. 
\vspace{-3pt}

\begin{assumption} \label{M1} 
For every $M_n \to \infty$, the event  \smash{$\hat{A}_{n,0} =\{\| \hat{\theta} - \theta_* \|\leq M_n\sqrt{d}/\sqrt{n}\} $} satisfies   \smash{$P_{0}^n\big(\hat{A}_{n,0}\big) >1-\hat{\epsilon}_{n,0}$} for some sequence \smash{$\{\hat{\epsilon}_{n,0}\}_{n=1}^\infty$} converging to zero.
\end{assumption}
\vspace{-20pt}
\begin{assumption} \label{M2}  There exist two positive constants $\bar \eta_1,\bar \eta_2$ such that the event $\hat A_{n,1}=\smash{ \{ \lambda_{\textsc{min}}(\hat \Omega^{-1})> \bar \eta_1 \} \cap\{ \lambda_{\textsc{max}}(\hat \Omega^{-1})< \bar \eta_2 \}}$ holds with a probability \smash{$P_0^n \big(\hat A_{n,1} \big)> 1 -\hat{\epsilon}_{n,1} $} for a suitable sequence $\{\hat{\epsilon}_{n,1}\}_{n=1}^\infty$ converging to zero as $n \to \infty$. 
Moreover, there exist positive constants \smash{$\delta>0$} and $L_3>0$, $L_4>0$, \smash{$L_{\pi,2}>0$} such that, for  \smash{$ B_{\delta}(\hat{\theta}) := \{ \theta  \in \Theta \,:\, \| \hat \theta - \theta \| < \delta \}$}, the joint event $\hat{A}_{n,2}=\smash{\{\sup_{\theta \in B_{\delta}(\hat{\theta}) } 	\|\log \pi^{(2)}(\theta) \| < L_{\pi,2}\}\cap \{ \sup_{\theta \in B_{\delta}(\hat{\theta}) } \|\ell^{(3)}(\theta)/n\| }< L_3  \}\cap\smash{\{\sup_{\theta \in B_{\delta}(\hat{\theta}) } \| \ell^{(4)}(\theta)/n\|  < L_4 \}},$ holds with a probability $P_{0}^{n}\big(\hat{A}_{n,2}\big)>  1-\hat{\epsilon}_{n,2}$, for some suitable sequence $\{\hat{\epsilon}_{n,2}\}_{n=1}^\infty$ converging to zero, where $\|\cdot\|$ represents the spectral norm.
\end{assumption}
\vspace{-2pt}

Assumption \ref{M1} is  mild and holds generally in regular parametric problems. This assumption ensures us that the \textsc{map} is in a suitably--small neighborhood of $\theta_{*}$, where the centering took place in Corollary \ref{corol:1}. Condition \ref{M2} is a similar and arguably not stronger  version~of~the analytical assumptions for the Laplace method described in  \citet{Kadane1990}. Notice also that under  Assumption \ref{M1}, condition \ref{M2} replaces Assumptions \ref{cond:1}--\ref{cond:2}, and requires the upper bound to hold in a neighborhood of $\theta_{*}.$ These conditions ensure uniform control on the difference between the log--likelihood ratio and its third order Taylor's expansion. Based on these additional conditions we provide an asymptotic result for the skew--modal approximation in \eqref{limiting:distribution:map}, similar to Corollary \ref{corol:1}. The proof can be found in the Supplementary Materials and follows as a direct consequence of a more refined non--asymptotic bound we derive for \smash{$\| \Pi_n( \cdot ) - \hat{P}^n_{\textsc{sks}}(\cdot)  \|_{\textsc{tv}}$}; see Remark~\ref{rem_31}.

\begin{theorem} \label{thm:modal}
	Let $\hat{h}= \sqrt{n}(\theta - \hat{\theta})$, and define $M_n = \sqrt{c_0 \log n}$, with $c_0>0$. If Assumptions \ref{cond:uni}, \ref{cond:3}--\ref{cond:4},  and \ref{M1}--\ref{M2} are met, then the posterior for \smash{$\hat{h}$} satisfies
	\begin{eqnarray} \label{tot:variatindist:map}
		\| \Pi_n( \cdot ) - \hat{P}^n_{\textsc{sks}}(\cdot)  \|_{\textsc{tv}} = O_{P_{0}^n}\big(M_n^{c_8}/n\big), 
	\end{eqnarray}
	for some $c_8 >0 $, where $\hat{P}_{\textsc{sks}}^n(S) \, = \, \int_{S} \hat{p}_{\textsc{sks}}^n(\hat{h}) d\hat{h}$ for $S \subset \mathbbm{R}^{d}$ with $\hat{p}_{\textsc{sks}}^n(\hat{h})$ defined as in~\eqref{limiting:distribution:map}. In addition, let \smash{$G \, : \, \mathbbm{R}^d \to  \mathbbm{R} $} be a function satisfying \smash{$ | G(\hat h) | \lesssim \| \hat h\|^r $}. If the prior is such that $\int \|\hat h\|^r \pi(\hat \theta + \hat h/\sqrt{n}) d\hat h < \infty$
	then 
	\begin{eqnarray} \label{result:corol:functionals}
		\int  G(\hat h) | \pi_n(\hat h) - \hat{p}_{\textsc{sks}}^n(\hat h) | d\hat h = O_{P_0^n}(M_n^{c_8 + r}/n).
	\end{eqnarray}  
\end{theorem} 
\begin{remark}\label{rem_31}
As discussed above, Theorem \ref{thm:modal} follows directly from a more refined non--asymptotic upper bound that we derive for \smash{$\| \Pi_n( \cdot ) - \hat{P}^n_{\textsc{sks}}(\cdot)  \|_{\textsc{tv}}$} in Section \ref{sec:proof:finite:sample} of the~Supplementary Materials. In particular, as in recently--derived non--asymptotic results for the Gaussian Laplace approximation \citep[e.g.,][]{spok2021,spokoiny2023inexact}, it is possible to keep track of the constants and the dimension dependence also within our derivations, to show that on an event with high probability (approaching $1$), it holds
\begin{eqnarray}
\| \Pi_n( \cdot ) - \hat{P}^n_{\textsc{sks}}(\cdot)  \|_{\textsc{tv}} \leq C M_n^{c_8}d^3/n,
\label{non_as_bound_sks_post}
\end{eqnarray}
for some constant $C>0$ not depending on $d$ and $n$, see Theorem \ref{thm:modal:fs}. As a consequence,~the rates in \eqref{tot:variatindist:map} follow directly from \eqref{non_as_bound_sks_post}, when keeping $d$ fixed and letting $n \to \infty$. More~importantly, the above bound vanishes also when the dimension $d$ grows with $n$, as long~as $d \ll n^{1/3}$ up to a poly--log~term. Although our original focus is not specific to these high--dimensional regimes,~it~shall be emphasized that such a growth for $d$ is interestingly in line with those required either for $d$ \citep[e.g.,][]{panov2015finite} or for the notion of effective dimension \smash{$\tilde{d}$}  \citep[][]{spok2021,spokoiny2023inexact} in recent high--dimensional studies of the~Gaussian Bernstein--von Mises theorem and the Laplace approximation. However, unlike the bounds derived in these studies, the one in \eqref{non_as_bound_sks_post} decays to zero with $n$, up to a poly--log term, rather than $\sqrt{n}$, for any given dimensions.
\end{remark}

\begin{remark}\label{rem_4}
Similarly to Remark \ref{rem_0} our proofs can be easily modified to show that the \textsc{tv} distance between the posterior and the classical Gaussian Laplace approximation~is, up to a poly--log term, of order $1/\sqrt{n}$. This is a substantially worse upper bound than those derived for the skew--modal approximation. Theorem \ref{thm:modal:gauss:lower}  in the Supplementary Materials further refines such a result by proving that, up to a poly--log term, this upper bound~is sharp, whenever the posterior displays local asymmetries; see condition \eqref{cond:asym}. More specifically, under \eqref{cond:asym}, we prove that, on an event with high probability (approaching $1$), the  \textsc{tv} distance between the posterior and the classical Laplace approximation (\textsc{gm}) is bounded from below by $C_d/\sqrt{n}+ O(M_n^{c_8}d^3/n)$,~for some constant $C_d>0$, possibly depending on  $d$. Crucially, the proof of this lower bound implies that \smash{$\| \Pi_n( \cdot ) - \hat{P}^n_{\textsc{gm}}(\cdot)  \|_{\textsc{tv}}-\| \Pi_n( \cdot ) - \hat{P}^n_{\textsc{sks}}(\cdot)  \|_{\textsc{tv}}$} is~also bounded from below by $C_d/\sqrt{n}+ O(M_n^{c_8}d^3/n)$. This result strengthens \eqref{tot:variatindist:map}--\eqref{non_as_bound_sks_post}.
\end{remark}

\begin{remark}\label{rem_2}
	Since the \textsc{tv} distance is invariant with respect to scale and location transformations, the above results can be stated also for the original parametrization $\theta$ of interest. Focusing, in particular, on the choice $F(\cdot) = \Phi(\cdot)$, this yields the density
	\begin{eqnarray} \label{post:theta:map}
		\hat{p}^n_{\textsc{sks}}(\theta) = 2 \phi_d(\theta; \hat{\theta}, J_{\hat{\theta}}^{-1}) \Phi( (\sqrt{2 \pi}/12) \ell^{(3)}_{\hat{\theta},stl} (\theta - \hat{\theta})_s  (\theta - \hat{\theta})_t  (\theta - \hat{\theta})_l), 
	\end{eqnarray}
which coincides with that of the well--studied sub--class of generalized skew--normal (\textsc{gsn}) distributions \citep{ma2004flexible} and is guaranteed to approximate the posterior density for $\theta$ with the rate derived in Theorem \ref{thm:modal}.
\end{remark}

Our novel skew--modal approximation provides, therefore, a similarly tractable, yet substantially more accurate, alternative to the classical Gaussian from the Laplace method. This is because, as discussed in Section~\ref{sec_2}, the closed--form skew--modal density can be~evaluated at a similar computational cost as the Gaussian one, when $d$ is not too large. Furthermore, it admits a straightforward i.i.d. sampling scheme that facilitates Monte Carlo estimation of any functional of interest. Recalling Section~\ref{sec_2}, such a scheme simply relies on sign perturbations of samples from a $d$--variate Gaussian and, hence, can be implemented via standard \texttt{R} packages for simulating from these~variables. Note that, although the non--asymptotic bound in \eqref{non_as_bound_sks_post} can be also derived for the theoretical skew--symmetric approximations  in Section~\ref{sec_2}, the focus on the skew--modal is motivated by the fact that such an approximation provides the solution implemented in practice. Section~\ref{sec_marginal} derives and studies an even more scalable, yet similarly--accurate, approximation when the focus is on posterior marginals.  

\subsection{Marginal skew--modal approximation and theoretical guarantees} \label{sec_marginal}
The skew--modal approximation  in Section \ref{sec_31} targets the joint posterior. In practice, the marginals of such a posterior are often the main object of interest \citep{inla_paper}. For studying these quantities,  it is possible to simulate i.i.d. values from the joint skew--modal approximation in~\eqref{limiting:distribution:map}, leveraging the sampling strategy discussed in Section \ref{sec_2}, and then retain only  samples from the marginals of direct interest. This requires, however, multiple evaluations of the cubic function in the skewness--inducing factor. In the following, we derive a closed--form skew--modal approximation for posterior marginals that mitigates this scalability issue.

To address the above goal,  denote with $\cc \subseteq \{1,\dots, d\}$ the set containing the indexes for the elements of $\theta$ on which we are interested in. Let $d_\cc$ be the cardinality of $\cc$, and~$\bar \cc = \cc^c$ the complement of $\cc$. Finally, write \smash{$\hat{h} = (\hat{h}_{\cc},\hat{h}_{ \bar\cc} )$}. Accordingly, the corresponding matrix~$\smash{\hat \Omega} = (J_{\hat \theta}/n)^{-1}$ can be partitioned in two diagonal blocks \smash{$\hat \Omega_{\cc\cc} $, $\hat \Omega_{ \bar \cc \bar \cc}  $,} and an off--diagonal one \smash{$\hat \Omega_{\bar \cc\cc} $}.

Under the regularity conditions stated in Section \ref{sec_31}, it is possible to write, for $n \to \infty$,
\begin{eqnarray*} 
	\pi_n(\hat \theta + \hat h/\sqrt{n}) \propto \exp( - j_{\hat \theta, st} \hat h_s \hat h_t/(2n) + (\ell_{\hat \theta, stl}^{(3)}/n)  \hat h_s \hat h_t \hat h_l /(6\sqrt{n}))  + O_{P_0^n}(n^{-1}).
\end{eqnarray*}  
The second order term in the above expression is proportional to the kernel of a Gaussian and, therefore, can be decomposed as 
\vspace{-3pt}
\begin{eqnarray*} 
\exp( - j_{\hat \theta, st} \hat h_s \hat h_t/(2n)) \propto \phi_d(\hat h; 0 , \hat \Omega ) =  \phi_{d_\cc}( \hat h_{\cc}; 0, \hat \Omega_{\cc\cc} ) \phi_{d - d_\cc}(\hat h_{\bar \cc};  \Lambda_{\cc} \hat h_{\cc}  , \bar \Omega), 
\end{eqnarray*}  
where~$\Lambda_{\cc} =  \smash{\hat \Omega_{\bar \cc\cc} \hat \Omega_{\cc\cc}^{-1}}$ and $ \bar \Omega =  \hat \Omega_{\bar \cc\bar \cc} -  \hat \Omega_{\bar \cc \cc} \hat \Omega_{\cc\cc}^{-1} \hat \Omega_{ \cc \bar \cc} $.

To obtain a marginal skew--modal approximation, let us leverage again the fact that~the~third order term converges to zero in probability, and that $e^x = 1 + x + O(x^2),$ for $x \to 0$. With these results, an approximation for the posterior marginal of \smash{$\hat h_{\cc}$} is, therefore, proportional~to
\begin{eqnarray}
\begin{split}
 \label{taylor:marginal}
	&\int \phi_{d_\cc}( \hat h_{\cc}; 0, \hat \Omega_{\cc\cc} ) \phi_{d - d_\cc}(\hat  h_{\bar \cc};  \Lambda_{\cc} \hat  h_{\cc}  , \bar \Omega)(1 +  (\ell_{\hat \theta, stl}^{(3)}/n)   \hat h_s \hat h_t \hat h_l /(6\sqrt{n}) ) d \hat  h_{\bar \cc}\\
	&\smash{= \phi_{d_\cc}( \hat  h_{\cc}; 0, \hat \Omega_{\cc\cc} ) [1+\{ (1/n)/(6\sqrt{n})\} \E_{\hat h_{\bar \cc }| \hat h_{\cc}}(\ell_{\hat \theta, stl}^{(3)} \hat h_s \hat h_t \hat h_l )]},
	\end{split}
\end{eqnarray}
where $\E_{\hat h_{\bar \cc }| \hat h_{\cc}}$ denotes the expectation with respect to $\phi_{d - d_\cc}(\hat h_{\bar \cc};  \Lambda_{\cc} \hat h_{\cc}  , \bar \Omega)$. Leveraging basic properties of the expected value, the term \smash{$ \E_{\hat h_{\bar \cc }| \hat h_{\cc}}(\ell_{\hat \theta, stl}^{(3)} \hat h_s \hat h_t \hat h_l )$} can be further decomposed as
\begin{eqnarray}
\qquad \ \ell_{\hat \theta, stl}^{(3)}  \hat h_{s} \hat h_{t} \hat h_{l} + 3 \ell_{\hat \theta, st r}^{(3)}  \hat h_{s} \hat h_{t}  \E_{\hat h_{\bar \cc }| \hat h_{\cc}}(\hat h_{r}) + 3 \ell_{\hat \theta, s r v }^{(3)}  \hat h_{s} \E_{\hat h_{\bar \cc }| \hat h_{\cc}}(\hat h_{r} \hat h_{v})  + \ell_{\hat \theta, r v k}^{(3)}  \E_{\hat h_{\bar \cc }| \hat h_{\cc}}(\hat h_{r} \hat h_{v} \hat h_{k}),
\label{eq_sum_margin}
\end{eqnarray}
with $s,t,l \in~\cc$ and $r,v,k \in \bar \cc$. Therefore, the above expected values simply require the first three non--central moments of the multivariate Gaussian having density \smash{$\phi_{d - d_\cc}(\hat h_{\bar \cc};  \Lambda_{\cc} \hat h_{\cc}  , \bar \Omega)$}. These are $\E_{\hat h_{\bar \cc }| \hat h_{\cc}}(\hat h_{r})= \Lambda_{\cc, r l}\hat{h}_l$, $\E_{\hat h_{\bar \cc }| \hat h_{\cc}}(\hat h_{r} \hat h_{v})= \bar{\Omega}_{rv}+ \Lambda_{\cc, r t} \Lambda_{\cc, v l}\hat{h}_t\hat{h}_l$ and $\E_{\hat h_{\bar \cc }| \hat h_{\cc}}(\hat h_{r} \hat h_{v} \hat h_{k})= 3\smash{\bar{\Omega}_{rv} \Lambda_{\cc, k s}\hat h_s+ \Lambda_{\cc, r s} \Lambda_{\cc, v t} \Lambda_{\cc, k l} \hat h_s  \hat h_t  \hat h_l}$. Hence, letting
\vspace{-5pt}
\begin{eqnarray}
\begin{split}
\nu^n_{1,s}&=3 \ell_{\hat \theta, srv}^{(3)} \bar{\Omega}_{rv}+3 \ell_{\hat \theta, rvk}^{(3)}\bar{\Omega}_{rv} \Lambda_{\cc, k s},  \\
\nu^n_{3,stl}&= \smash{\ell_{\hat \theta, stl}^{(3)} +3 \ell_{\hat \theta, str}^{(3)}\Lambda_{\cc, r l}+3 \ell_{\hat \theta, srv}^{(3)}\Lambda_{\cc, r t} \Lambda_{\cc, v l}+\ell_{\hat \theta, rvk}^{(3)}\Lambda_{\cc, r s} \Lambda_{\cc, v t} \Lambda_{\cc, k l}},
\label{nu1nu2}
\end{split}
\end{eqnarray}
the summation in \eqref{eq_sum_margin} can be written as $\nu^n_{1,s}\hat h_s+\nu^n_{3,stl} \hat h_s  \hat h_t  \hat h_l $, with $s,t,l \in~\cc$.  Replacing this quantity in \eqref{taylor:marginal}, yields \smash{$2 \phi_{d_\cc}( \hat h_{\cc}; 0, \hat \Omega_{\cc\cc} )( 1/2 + \eta \alpha_{\eta,\cc}(\hat h_{\cc}))$}, with
\begin{eqnarray} \label{def:param:ske:margin}
	 \alpha_{\eta,\cc}( \hat h_{\cc}) = \{1/(12 \eta \sqrt{n})\}(1/n)(\nu_{1,s}^n \hat h_s  +  \nu_{3,stl}^n \hat h_s \hat h_t  \hat h_l).
\end{eqnarray} 
Therefore, by leveraging the reasoning as in Section~\ref{sec_21}, we can write
$2 \phi_{d_\cc}( \hat h_{\cc}; 0, \hat \Omega_{\cc\cc} )( 1/2 + \smash{\eta \alpha_{\eta,\cc}(\hat h_{\cc})) =  2 \phi_{d_\cc}( \hat h_{\cc}; 0, \hat \Omega_{\cc\cc} ) F( \alpha_{\eta,\cc}( \hat h_{\cc}) ) + O_{P_0^n}(n^{-1})},$
where $\eta \in \mathbbm{R}$, and $F(\cdot): \mathbbm{R} \to [0,1]$ is a univariate cdf satisfying $F(-x) = 1-F(x)$, $F(0) =  1/2$ and $F(x) = F(0) + \eta x + O(x^2).$  As a result, the posterior marginal density of \smash{$\hat h_{\cc}$} can be approximated~by 
\begin{eqnarray} \label{skew:marginal}
	\hat p_{\textsc{sks},\cc}^n(\hat h_{\cc})  = 	2 \phi_{d_\cc}( \hat h_{\cc}; 0, \hat \Omega_{\cc\cc} )w_{\cc}(\hat h_{\cc})=2 \phi_{d_\cc}( \hat h_{\cc}; 0, \hat \Omega_{\cc\cc} ) F( \alpha_{\eta,\cc}( \hat h_{\cc}) ).
\end{eqnarray}
 Note that $ \alpha_{\eta,\cc}( \hat h_{\cc}) $ in \eqref{def:param:ske:margin} is an odd polynomial of $ \hat h_{\cc}$, and that $ \alpha_{\eta,\cc}( \hat h_{\cc})=  \E_{\hat h_{\bar \cc } | \hat h_{\cc} }\{\hat{\alpha}_{\eta}(\hat h)\} $. 

Equation \eqref{skew:marginal} shows that, once the quantities defining \smash{$\hat p_{\textsc{sks},\cc}^n(\hat h_{\cc})$} are pre--computed, then the cost of inference under such an approximating density scales with $d_{\cc}$, and no more~with $d$. As a consequence, when the focus is on the univariate marginals, i.e., $d_{\cc}=1$, the computational gains over the joint approximation in  \eqref{limiting:distribution:map} can be substantial, and calculation of functionals can be readily performed via  one--dimensional numerical integration methods.

Theorem~\ref{thm:marginal} below clarifies that, besides being effective from a computational perspective, the above solution preserves the same theoretical accuracy guarantees  in approximating the target marginal posterior density  \smash{$\pi_{n,\cc}(\hat h_{\cc}) = \int \pi_{n}(\hat h) 
d \hat h_{\bar \cc}$}.

\begin{theorem} \label{thm:marginal}
Let  \smash{$\Pi_{n,\cc}(S)= \, \int_{S}\pi_{n,\cc}(\hat h_{\cc}) d\hat h_\cc$}  for $S \subset \mathbbm{R}^{d_\cc}$. Then, under the assumptions of Theorem \ref{thm:modal}, we have that
	\vspace{-1pt}
	\begin{eqnarray} \label{tot:variatindist:marginal}
		\| \Pi_{n,\cc}( \cdot ) - \hat{P}^n_{\textsc{sks}, \cc}(\cdot)  \|_{\textsc{tv}} = O_{P_{0}^n}\big(M^{c_9}_n/n\big), 
	\end{eqnarray}
	for $ c_9>0 $, where $\hat{P}_{\textsc{sks},\cc}^n(S) \, = \, \int_{S} \hat{p}_{\textsc{sks},\cc}^n(\hat{h}_\cc) d\hat h_\cc$ with $\hat{p}_{\textsc{sks},\cc}^n(\hat{h}_{\cc})$ defined as in \eqref{skew:marginal}.
\end{theorem} 

\begin{remark}\label{rem_3}
	As for Remark \ref{rem_2}, Theorem~\ref{thm:marginal} holds also in the original~parametrization $\theta$. Considering, in particular, the \textsc{gsn} case with $F(\cdot)=\Phi(\cdot)$, this implies that
	\begin{equation} \label{post:theta:map:margin}
		\hat{p}^n_{\textsc{sks},\cc}(\theta_{\cc}) = 2 \phi_{d_{\cc}}( \theta_{\cc}; \hat{\theta}_{\cc}, J_{\hat{\theta}, \cc \cc}^{-1}) \Phi\Big( \frac{\sqrt{2 \pi}}{12 } \Big\{ \frac{\nu_{1,s}^n}{n}  (\theta - \hat \theta)_s  +  \nu_{3,stl}^n  (\theta - \hat \theta)_s  (\theta - \hat \theta)_t   (\theta - \hat \theta)_l\Big\} \Big), 
	\end{equation}
approximates $\pi_{n,\cc}(\hat h_{\cc})$ with rate as in Theorem \ref{thm:marginal}, $s,t,l \in \cc$ and $\nu_{1,s}^n$,$ \nu_{3,stl}^n$ defined in  
	\eqref{nu1nu2}.
\end{remark}


\section{Empirical analysis of skew--modal approximations} \label{sec_32}
Sections~\ref{sec_321}--\ref{sec_322} demonstrate on both synthetic datasets and real--data applications that the join and marginal skew--modal approximations  (\textsc{skew--m})  in Section~\ref{sec_3} achieve remarkable accuracy improvements relative to the Gaussian--modal counterpart  (\textsc{gm})  from the Laplace method. These improvements are again in line with the rates we derived theoretically. Comparisons against other state--of--the--art approximations from mean--field \textsc{vb} \citep[e.g.,][]{blei2017variational} and \textsc{ep} \citep[e.g.,][]{vehtari2020expectation} are also discussed. In the following, we focus, in particular, on assessing performance of the generalized skew--normal approximations in Remarks~\ref{rem_2}--\ref{rem_3}.

\vspace{-2pt}
\subsection{Exponential model revisited} \label{sec_321}
Let us first replicate the simulation studies in Sections \ref{sec_241}--\ref{sec_242} with  focus on the practical skew--modal approximation in Section~\ref{sec_31}, rather than its population version which assumes knowledge of $\theta_*$. Consistent with this focus, the performance of the \textsc{skew--m} approximation in Equation~\eqref{post:theta:map} is compared against the \textsc{gm} solution  \smash{$\mbox{N}(\hat{\theta}, J_{\hat{\theta}}^{-1})$} arising from the Laplace method \citep[see e.g.,][p. 318]{gelman2013bayesian}. Note that both the correctly--specified and misspecified models satisfy the additional Assumptions \ref{M1}--\ref{M2} required by Theorem~\ref{thm:modal} and Remark~\ref{rem_2}. In fact, \smash{$\hat{\theta}$} is asymptotically equivalent to the maximum likelihood estimator which implies that condition \ref{M1} is fulfilled. Moreover, in view of the expressions for the first three log--likelihood derivatives in Section \ref{sec_241} also \ref{M2} holds.

\begin{table}[t]
	\centering
	\caption{\footnotesize For each $n$ from $n=10$ to $n=50$, sample size $\bar{n}$ required by the classical Gaussian from the Laplace method (\textsc{gm}) to obtain the same \textsc{tv} and \textsc{fmae} achieved by our skew--modal  approximation (\textsc{skew--m}) with that~$n$.}
	\begin{tabular}{lrrrrr}
		\hline
		&  $ n=10$ & $ n=15$ & $ n=20$ & $ n=25$ & $ n=50$  \\ 
		\hline
		$\bar{n}: \ \textsc{tv}^{\bar{n}}_{\textsc{gm}}=\textsc{tv}^{n}_{\textsc{skew--m}} $ & 150 & 260 & 470 & 730 & $> 2500$  \\ 
		$\bar{n}: \ \textsc{fmae}^{\bar{n}}_{\textsc{gm}}=\textsc{fmae}^{n}_{\textsc{skew--m}} $   & 190 & 390 & 650 & 1030 & $> 2500$  \\ 
		\hline
	\end{tabular}
	\label{tab1}
	\vspace{-10pt}
\end{table}

Table~\ref{tab1} reports the same summaries as in the second part of Table~\ref{tab}, but now with a focus on comparing the  \textsc{skew--m}  approximation in \eqref{post:theta:map} and the  \textsc{gm}  \smash{$\mbox{N}(\hat{\theta}, J_{\hat{\theta}}^{-1})$}. Results are in line with those in Section \ref{sec_241}, and show, for example, that to achieve the same accuracy attained by the skew--modal with $n=20$, the Gaussian from the Laplace method requires a sample size of $\bar{n} \approx 500$. These results are strengthened in Tables~\ref{tab1_supp}--\ref{tab_exp_miss_n} in the Supplementary Materials which confirm the findings of Sections \ref{sec_241}--\ref{sec_242}. Also in this context, the asymptotic theory in Theorem~\ref{thm:modal} closely matches the empirical behavior observed in~practice.

\subsection{Probit and logistic regression model} \label{sec_322}
We consider now a real--data application~on~the Cushings dataset \citep{masslibrary}, openly--available in the \texttt{R} library \texttt{Mass}. In this case the true data--generative model is not known and, therefore, this analysis is useful to evaluate again performance in possibly misspecified contexts.  

The data are  obtained from a medical study on $n=27$ individuals, aimed at investigating the relationship between four different sub--types of Cushing's syndrome and two steroid metabolites, {\em Tetrahydrocortisone} and {\em Pregnanetriol} respectively. To simplify the analysis, we consider here the binary response variable $X_i \in \{0,1\}$ which takes value $1$ if patient $i$ is affected by bilateral hyperplasia, and $0$ otherwise, for $i = 1\dots,n.$ The observed covariates are $z_{i1}$ =  {\em “urinary excretion rate of Tetrahydrocortisone for patient $i$"} and {\em $z_{i2}$ = “urinary excretion rate of Pregnanetriol for patient $i$"}. In the following, we focus on the two most widely--implemented regression models for binary data, namely the probit regression $X_i \ \smash{\stackrel{ind}{\sim}} \ \mbox{Bern}( \Phi(\theta_0  + \theta_1 z_{i1} + \theta_2 z_{i2} ))$, and the logistic one $X_i \ \smash{\stackrel{ind}{\sim}} \ \mbox{Bern}( g(\theta_0 + \theta_1 z_{i1} + \theta_2 z_{i2} ))$ with $ g(\cdot)$ the inverse logit function defined in Remark~\ref{reF}.

Under both models, Bayesian inference proceeds via standard weakly informative Gaussian priors $\mbox{N}(0,25)$ for the three regression coefficients within $\theta=(\theta_0, \theta_1, \theta_2)^{\intercal}$. Such priors, combined with the likelihood of each model, yield a posterior for $\theta$ which we approximate under  both the joint and the marginal skew--modal approximations (\textsc{skew--m}). Table~\ref{table:tv} compares, via different measures, the accuracy of these solutions relative to the one obtained under   the classical Gaussian--modal approximation from the Laplace methods (\textsc{gm}) \citep[][pp. 318]{gelman2013bayesian}. Notice that, all these approximations can be readily derived from the closed--form derivatives of the log--likelihood and log--prior for both the probit and logistic regression. Moreover, since the prior is Gaussian, the \textsc{map} under both models coincides with the ridge--regression estimator and hence can be computed via basic \texttt{R} functions.

 \begin{table}[t]
 \vspace{-5pt}
\renewcommand{\arraystretch}{0.9}
\centering
\caption{\footnotesize For the probit and logistic regression, comparison among the accuracy of the skew--modal approximation (\textsc{skew--m}) and the classical Gaussian one from the Laplace method (\textsc{gm}). Performance is measured in terms of (i) \textsc{tv} distances from the target joint posterior and its marginals, (ii) error (\textsc{err}) in approximating the posterior means and (iii) average error (\textsc{ave--pr}) in the approximation of the posterior probabilities  of being affected by bilateral hyperplasia for each patient. Bold values indicate best performance under each measure.} 
\small
  \begin{adjustbox}{max width=\textwidth}
\begin{tabular}{lrrrrrrrr}
  \hline
 \quad & \qquad  \qquad $\textsc{tv}_{\theta}$ &  $\textsc{tv}_{\theta_{0}}$  &  $\textsc{tv}_{\theta_{1}}$ &  $\textsc{tv}_{\theta_{2}}$  & $\textsc{err}_{\theta_{0}}$&  $\textsc{err}_{\theta_{1}}$ &  $\textsc{err}_{\theta_{2}}$ & \textsc{ave--pr} \\ 
   \hline
 Probit & &   &  &  & &  &  &    \\ 
  \hline
    \textsc{skew--m} \qquad & $\bf 0.11$ &  $\bf 0.03$ & $\bf 0.04$ & $\bf0.05$  &$\bf 0.004$& $\bf 0.002$ &$ \bf 0.015$& $\bf 0.006$  \\ 
     \textsc{gm} \qquad& $0.19$  & $0.09$ & $0.08$ & $0.11$  & $-0.092$  & $0.008$ & $0.051$&$0.026$    \\ 
    \hline
 Logit & &   &  &  & &  &  &    \\ 
  \hline
    \textsc{skew--m} \qquad & $\bf 0.14$  & $\bf 0.05$ & $ \bf 0.06$ & $\bf 0.07$&$\bf 0.069$ & $\bf -0.001$ & $\bf -0.008$&$\bf 0.009$  \\ 
     \textsc{gm} \qquad & $0.23$  & $0.11$ & $0.10$ & $0.14$& $-0.116$ & $0.010$ & $0.060$ & $0.064$ \\ 
       \hline
\end{tabular}
\end{adjustbox}
\label{table:tv}
\vspace{-10pt}
\end{table}

Table \ref{table:tv} displays Monte Carlo estimates of \textsc{tv} distances from the target posterior distribution and its marginals, along with errors in approximating the posterior means for~the~three regression parameters and the posterior probabilities  of being affected by a bilateral hyperplasia.  Under probit, the latter quantity is defined as $\textsc{Ave--pr} = \sum_{i = 1}^n | \mathrm{pr}_i - \mathrm{\hat{pr}}_{\textsc{app},i}  |/n$ with $\mathrm{pr}_i = \int \Phi(\theta_0 + \theta_1 z_{i1} + \theta_2 z_{i2} ) \pi_n(\theta)d\theta$  and $ \mathrm{\hat{pr}}_{\textsc{app},i} = \int \Phi(\theta_0 + \theta_1 z_{i1} + \theta_2 z_{i2} ) \hat{p}^n_{\textsc{app}}(\theta) d\theta$, for  each $i = 1,\dots,n$, where $ \hat{p}^n_{\textsc{app}}(\theta)$ is any generic approximation~for~$\pi_n(\theta)$. The logistic case follows by replacing $\Phi(\cdot)$ with $g(\cdot)$. The Monte Carlo estimate of such a measure and of~all those reported within Table \ref{table:tv} rely on $10^5$  i.i.d. samples from both \textsc{skew--m} and \textsc{gm}, and~on $2$ chains of length $10^5$ of Hamiltonian Monte Carlo realizations from the target posterior~obtained with the \texttt{R} function \texttt{stan\_glm} from the \texttt{rstanarm} package.
  
As illustrated in Table  \ref{table:tv}, the proposed \textsc{skew--m} solutions generally yield remarkable accuracy improvements relative to \textsc{gm}, under both models. More specifically, \textsc{skew--m} almost halves, on average, the \textsc{tv} distance associated with \textsc{gm}, while providing a much more accurate approximation for the posterior means and posterior probabilities. This is an important accuracy gain provided that the ratio between the absolute error made by \textsc{gm} in posterior means approximation and the actual value of these posterior means is, on average, $\approx 0.25$.

As discussed in the Supplementary Materials,  \textsc{skew--m} outperforms also state--of--the--art mean--field  \textsc{vb} \citep{consonni2007mean,durante2019conditionally, fasanoscalable}, and is competitive with \textsc{ep} \citep{Chopin_2017}. The latter result is particularly remarkable since our proposed approximation only leverages  the local behavior of the posterior distribution in a neighborhood of its mode, whereas \textsc{ep}  is known to provide an accurate global solution aimed at matching the first two moments of the target posterior.

\subsection{High--dimensional logistic regression} \label{sec_hd_logistic}
We conclude with a final real--data application which is useful to assess more in detail the marginal skew--modal approximation from Section \ref{sec_marginal}, while studying the performance of the proposed class of skewed approximations in a high--dimensional context that partially departs from the regimes we have studied from a theoretical perspective. To this end, we consider a clinical study that investigates whether biological measurements from cerebrospinal fluid  collected on $n = 333$ subjects can be used to diagnose the Alzheimer's disease \citep{craig2011multiplexed}. The dataset~is~available in the \texttt{R} package \texttt{AppliedPredictiveModeling} and comprises $130$ explanatory variables along with a response $X_i \in \{0,1\}, \,i = 1, \dots,n$, which takes the value 1 if patient $i$ is affected by the Alzheimer's disease, and 0 otherwise.

Bayesian inference relies on logistic regression with independent Gaussian priors~$\mbox{N}(0,4)$ for the coefficients. Here we consider a lower variance than in the previous application to induce shrinkage in this higher--dimensional context. The inclusion of the intercept and the presence of a categorical variable with 6 levels imply that the number of parameters in the model is $d = 135$. As a consequence, although the sample size is not small in absolute terms, since $n/d \approx 2.5$ and $d>n^{1/3}$ the behavior of the posterior  in this example is not necessarily closely described by the  asymptotic and non--asymptotic theory developed in Section \ref{sec_3}. 

\begin{table}[t]
\vspace{-8pt}
	\caption{\footnotesize For the logit model in Section \ref{sec_hd_logistic}, mean and median of the approximation error (\textsc{err}) and \textsc{tv} distance from the target posterior under both the marginal \textsc{skew--m} and \textsc{gm}. Bold values indicate best performance.}
	\label{table:high:dim}
	\centering
	\begin{tabular}{lrrrrr}
		\hline
		&  \textsc{err} (mean) &  \textsc{err} (median) &  \textsc{tv}  (mean) &  \textsc{tv} (median) \\ 
		\hline
		\textsc{skew--m} & {\bf 0.139} & {\bf 0.068 } & {\bf 0.104} & {\bf 0.078} \\ 
		\textsc{gm} & 0.425  & 0.347 & 0.145  & 0.120  \\ 
		\hline
	\end{tabular}
	\vspace{-8pt}
\end{table}

Nonetheless, as clarified in Table \ref{table:high:dim}, \textsc{skew--m} still yields remarkable improvements relative to \textsc{gm} also in this challenging regime. These gains are visible both in the absolute~difference between the exact posterior mean and its approximation (\textsc{err}), and also in the \textsc{tv} distances between each marginal posterior density and its approximation (\textsc{tv}). Such quantities are computed via Monte Carlo as in Section~\ref{sec_322} for each of the $d=135$ coefficients.  Table \ref{table:high:dim} reports the means and medians over these $135$ different values. As for the results in Section~\ref{sec_322}, also these improvements are particularly relevant provided that the absolute error of \textsc{gm} is not negligible when compared with the actual posterior means ($95\%$ of these means are between $-2.68$ and $2.66$). These findings provide further empirical evidence in favor of the proposed \textsc{skew--m}, and clarify that it can yield substantial accuracy improvements whenever the shape of the posterior departs from Gaussianity, either because of low sample size, or also in situations where $n$ is large in absolute terms, but not in relation to $d$.

\section{Discussion} \label{sec_4}
Through a novel treatment of a third order version of the Laplace method, this article shows that it is possible to derive valid, closed--form and  tractable skew symmetric approximations of posterior distributions. Under general assumptions which also account for both misspecified models and non--i.i.d. settings, such a novel family of approximations is shown to admit a Bernstein--von Mises type result that establishes remarkable improvements in convergence rates to the target posterior relative to those of the classical Gaussian limiting approximation. The specialization of this general theory to regular~parametric models yields skew--symmetric approximations with a direct methodological impact and immediate applicability under a novel skew--modal solution which is obtained by replacing the unknown $\theta_*$ entering the theoretical version with its \textsc{map} estimate \smash{$\hat{\theta}$}. The empirical studies on both simulated data~and real applications confirm that the remarkable accuracy improvements dictated by our asymptotic and non--asymptotic theory are visible also~in~practice,~even for small--sample regimes. This provides further support to the superior theoretical, methodological and practical performance of the proposed class of approximations.

The above advancements open new avenues that stimulate research in the field of Bayesian inference based on skewed deterministic approximations. As shown in a number of contributions appeared after our article and referencing to our theoretical results, interesting directions~include the introduction of skewness in other deterministic approximations, such~as \textsc{vb} \citep[e.g.,][]{tan2023variational},  and further refinements of the high--dimensional results implied by~the non--asymptotic bounds we derive for the proposed skew--modal approximation.  \citet{katsevich2023tight} provides an interesting contribution along such a latter direction, which leverages~a novel theoretical approach based on Hermit polynomial expansions to show that $d$ can possibly grow faster than $n^{1/3}$, under suitable models. However, unlike for our results, the focus is on studying non--valid skewed approximating densities. The notion of effective dimension \smash{$\tilde{d}$} introduced by \citet{spok2021} and \citet{spokoiny2023inexact} for the study of the classical Gaussian Laplace approximation in high dimensions is also worth further~investigations under our skewed extension provided that \smash{$\tilde{d}$} can be possibly $o(d)$.

 Semiparametric settings \citep[e.g.,][]{bickel2012semiparametric,castillo2015bernstein} are also of interest. Moreover, although the inclusion of skewness is arguably sufficient to yield an accurate approximation of posterior distributions, accounting for kurtosis might provide additional improvements both in theory and in practice. To this end, a relevant research direction is to seek for an alternative to the Gaussian density in the symmetric part, possibly obtained from an extension to the fourth order of our novel treatment of the Laplace method. Our conjecture is that this generalization would provide an additional order--of--magnitude improvement in the rates, while yielding an approximation  still within the \textsc{sks} class.  

\vspace{20pt}
\renewcommand{\baselinestretch}{1.15} 
\bibliographystyle{imsart-nameyear}
\bibliography{example}

\vspace{19pt}
\renewcommand{\baselinestretch}{1} 
\setcounter{equation}{23}
\numberwithin{equation}{section}
\numberwithin{table}{section}
\numberwithin{figure}{section}

\begin{center}
\large {\bf  Supplementary Materials}
\end{center}

\begin{appendix}
\section{Proofs of Lemmas, Theorems and Corollaries}\label{sec_proof_ltc}
Section~\ref{sec_proof_ltc} contains the proofs of the Lemmas, Theorems and Corollaries stated in the~main article. The proof of Theorem~\ref{thm:modal} is discussed in Section~\ref{sec:proof:finite:sample}, and follows as a direct consequence of the non--asymptotic bound we derive for the \textsc{tv} distance among the skew--modal approximation and the target posterior.


\begin{proof}[proof of Lemma~\ref{lemma:distr:inv}] The proof of Lemma \ref{lemma:distr:inv} follows directly from Proposition 6~in \citet{wang2004skew} which states the distributional invariance of \textsc{sks} densities with respect to even functions.
\end{proof}


\begin{proof}[proof of Corollary~\ref{corol:1}] 	To prove \eqref{corol1:out1} notice that the general Assumptions \ref{cond:uni}--\ref{cond:LAN} and  \ref{cond:concentration}, introduced in Section \ref{sec_general_thm}, are implied by Assumptions \ref{cond:uni} and \ref{cond:1}--\ref{cond:4} together with Lemma~\ref{lemma:taylor:fixed:d} and Lemma \ref{lemma:post:contr} in Section~\ref{sec_logp}. In addition, Assumption \ref{cond:3} implies that Assumption \ref{cond:prior}~is~verified with \smash{$\log\pi_{\theta_*}^{(1)} = (\partial/\partial \theta ) \log \pi(\theta)_{| \theta = \theta_*}$}. Hence, all the conditions of Theorem \ref{thm:1} are~satisfied with $\delta_n = 1/\sqrt{n}$, proving the validity of the statement in Equation~\eqref{corol1:out1}.
	
	To prove~\eqref{corol1:out2}, recall that $K_n = \{ h \, : \, \| h\| < M_n \} = \{ \theta \, : \, \| \theta - \theta_*\| <  M_n/\sqrt{n} \} $. In addition, since $ | G | \lesssim \|h\|^r $, it is sufficient to prove the statement for $\|h\|^r$. Leveraging~the~triangle inequality we have
	\begin{equation}
		\begin{aligned}
			&\int \|h\|^r | \pi_n(h) - p_{\textsc{sks}}^n(h)  |  dh     \\
			& \qquad  \leq \int_{K_n^c}  \|h\|^r  \pi_n(h) dh + \int_{K_n^c}  \|h\|^r p_{\textsc{sks}}^n(h) dh +\int_{K_n} \|h\|^r | \pi_n(h) - p_{\textsc{sks}}^n(h) | dh. 
		\end{aligned}
		\label{ineq_A1}
	\end{equation} 
    Recall that, $A_{n,0} = \{ \lambda_{\textsc{min}}(J_{\theta_*}/n) > \eta_1^*\} \cap  \{ \lambda_{\textsc{max}}(J_{\theta_*}/n) < \eta_2^*\}$, for some positive constants \smash{$ \eta_1^*, \eta_2^*$}, and \smash{$A_{n,1} = A_{n,0} \cap \{\|\xi\| < \tilde{M}_n\} $} for some \smash{$\tilde{M}_n$} going to infinity arbitrary slow. Now, note that from Assumptions \ref{cond:1}--\ref{cond:2}, Lemma \ref{lemma:taylor:fixed:d} and Lemma \ref{lemma:eigen:J} it follows $P_0^n A_{n,1} = 1 - o(1)$.
  
	To bound the element \smash{$\int_{K_n^c}  \|h\|^r  \pi_n(h) dh$} in ~\eqref{ineq_A1} we use the fact that, from Assumptions \ref{cond:3}--\ref{cond:4} and Equation \eqref{contraction:denom} of Lemma \ref{lemma:post:contr}, the event
	$ A_{n,3} = A_{n,1} \cap \{ \sup_{\|\theta - \theta_* \|>  M_n/\sqrt{n}} \{ \ell(\theta) - \ell(\theta_*)  \} < - c_5 M_n^2 \} \cap \smash{\{ \int_{K_n} e^{\ell(\theta_* + h/\sqrt{n})-\ell(\theta_*)} \pi(\theta_* + h/\sqrt{n})dh > \tilde c_1 \},  }$ 	
	with $\tilde c_1 $ denoting an arbitrary small and fixed positive constant, satisfies $P_0^n A_{n,3} = 1-o(1)$.	
	By combining this~result with $\int \|h\|^r \pi(\theta_* + h/\sqrt{n}) dh < \infty$ and Jensen's inequality we obtain 
	\begin{equation*}
		\begin{aligned}
		\int_{K_n^c} \|h\|^r \pi_n(h) dh \mathbbm{1}_{A_{n,3}} & \leq \int_{K_n^c}   \|h\|^r  \frac{ e^{\ell(\theta_* + h/\sqrt{n})-\ell(\theta_*)} \pi( \theta_* + h/\sqrt{n}) }{\int_{K_n} e^{\ell(\theta_* + h'/\sqrt{n})-\ell(\theta_*)} \pi(\theta_* + h'/\sqrt{n})dh'}dh \mathbbm{1}_{A_{n,3}}\\
			& \lesssim \frac{1}{n^{c_0 c_5}} \int \|h\|^r \pi(\theta_* + h/\sqrt{n}) dh = O(n^{-1}), 			
		\end{aligned}
	\end{equation*}
	for a sufficiently large choice of $c_0$ in $M_n$. Since $P_0^n A_{n,3} = 1-o(1)$, this implies
	\begin{equation} \label{corol1:functionals:help1}
		\int_{K_n^c} \|h\|^r \pi_n(h) dh = O_{P_0^n}(n^{-1}).
	\end{equation}  
	Similarly, the boundedness of $w(\cdot)$
	and the tail behavior of the Gaussian distribution ensure
	\begin{equation*}
		\begin{aligned}
			\int_{K_n^c} \|h\|^r p_{\textsc{sks}}^n(h) dh \mathbbm{1}_{A_{n,1}}			\leq  2 \int_{K_n^c} \|h\|^r \phi_d(h; \xi, \Omega) dh \mathbbm{1}_{A_{n,1}} = O(n^{-1}),
		\end{aligned}
	\end{equation*}
	for a sufficiently large choice of $c_0$. In turn, this implies
	\begin{equation}\label{corol1:functionals:help2}
		\int_{K_n^c} \|h\|^r p_{\textsc{sks}}^n(h) dh  =  O_{P_0^n}(n^{-1}).
	\end{equation}
	Finally, Equation \eqref{corol1:out1} gives
	\begin{equation}\label{corol1:functionals:help3}
		\int_{K_n} \|h\|^r \big| \pi_n(h) - p_{\textsc{sks}}^n(h) \big | dh  \leq  M_n^r \int | \pi_n(h) - p_{\textsc{sks}}^n(h) | dh = O_{P_0^n}(M_n^{c_6 + r}/n).
	\end{equation}
 Combining \eqref{ineq_A1}, \eqref{corol1:functionals:help1}, \eqref{corol1:functionals:help2} and \eqref{corol1:functionals:help3} proves  Equation \eqref{corol1:out2}.     
\end{proof}


\begin{proof}[proof of Lemma~\ref{lemma:taylor:fixed:d}] 
	The proof of the lemma follows directly from Assumptions \ref{cond:1}--\ref{cond:2}. Such assumptions allow to take, in $K_n = \{ h \, : \,  \|h\| \leq M_n \}$, the following Taylor expansion 
	\begin{equation*} 
		\log \frac{p_{\theta_* + h/\sqrt{n} }^n}{p_{\theta_*}^n}(X^n) = h_s \frac{\ell^{(1)}_{\theta_*,s}}{\sqrt{n}} - \frac{1}{2} \frac{j_{st}}{n} h_s h_{t} + \frac{1}{6 \sqrt{n}} \frac{\ell^{(3)}_{\theta_*,stl}}{n} h_s h_t h_l +  r_{n,1}(h),
	\end{equation*}
	with  $\ell^{(1)}_{\theta_*,s}/\sqrt{n}= O_{P_0^n}(1)$, $j_{st}/n =O_{P_0^n}(1)$, $ \ell^{(3)}_{\theta_*,stl}/n = O_{P_0^n}(1)$ and
	$$ \sup_{h \in K_n} r_{n,1}(h) = \sup_{h \in K_n} \frac{1}{24 n} \frac{\ell^{(4)}_{\theta_*,stlk}(\beta h)}{n} h_s h_t h_l h_k  =  O_{P_0^n}(M_n^4/n),$$
	for some $\beta \in (0,1)$. To conclude, we only need to check that the first term can be written as \smash{$h_s (j_{st}/n)\Delta_{\theta_*, t}^n $}  with \smash{$\Delta_{\theta_*, t}^n = j^{-1}_{st} \sqrt{n} \ell^{(1)}_{\theta_{*},s} = O_{P_0^n}(1)$}. To this end, note that, in view of Assumption \ref{cond:2}, Lemma \ref{lemma:eigen:J} implies that $\lambda_{\textsc{max}}(J_{\theta_*}/n)$ and $\lambda_{\textsc{min}}(J_{\theta_*}/n)$ are bounded from above and below, respectively, with probability tending to 1 as $n \to \infty$. Since, by the eigendecomposition (we assume the eigenvectors are normalized) it follows that the entries of $(J_{\theta_{*}}/n)^{-1}$ are bounded, in absolute value, by $d /\lambda_{\textsc{min}}(J_{\theta_*}/n),$ we get $n j^{-1}_{st} = O_{P_0^n}(1)$, which implies, in turn, $\Delta_{\theta_*, t}^n=O_{P_0^n}(1)$.
	
\end{proof}


\begin{proof}[proof of Lemma~\ref{lemma:post:contr}] 
	 Let us start by writing 
	\begin{equation} \label{help:lemma:conc:1}
		\Pi_n( K_n^c)  \leq \frac{\int_{K_n^c} p_{\theta_* + h/\sqrt{n}}^n(X^{n}) \pi(\theta_* + h/\sqrt{n}) dh }{\int_{K_n} p_{\theta_* + h/\sqrt{n}}^n(X^{n}) \pi(\theta_* + h/\sqrt{n}) dh}. 	
	\end{equation}
	Recall that under Assumption \ref{cond:4} it holds
	$$\textstyle \lim_{n \to \infty } P_0^{n} \lbrace \sup_{\|h\|> M_n}\{ \ell(\theta_* + h/\sqrt{n}) - \ell(\theta_*)  \} < - c_5 M_n^2 \rbrace = 1.$$
	As a consequence, for every $D>1$, 
	\begin{equation} \label{contraction:numerator}
		\begin{aligned}
			\int_{  K_n^c} (p_{\theta_* + h/\sqrt{n}}^n/p_{\theta_*}^n)(X^{n}) \pi(\theta_* + h/\sqrt{n})/\pi(\theta_*) dh = O_{P_0^n}(n^{-D}),
		\end{aligned}
	\end{equation}
	given a sufficiently large constant $c_0$ in $M_n$ and the boundedness condition within Assumption \ref{cond:3}. 
	For the denominator of the right--hand--side of \eqref{help:lemma:conc:1}, we use Assumptions \ref{cond:1}--\ref{cond:2}--\ref{cond:3} to consider the Taylor expansions reported in \eqref{taylor:lprior:star} and \eqref{taylor:llik:star}. Recall that from Assumption \ref{cond:2} and Lemma \ref{lemma:eigen:J} there exist two positive constants $\eta^*_1$ and $\eta^*_2$ such that  $A_{n,0} = \{ \lambda_{\textsc{min}}(V_{\theta_*}^n  ) > \eta^*_1 \} \cap \{ \lambda_{\textsc{max}}( V_{\theta_*}^n ) < \eta^*_2\}$ holds with probability $P_{0}^n A_{n,0}  = 1-o(1)$.
	As a consequence, if we collect the third order term in \eqref{taylor:llik:star} and the prior effect in the remainder, it follows
		\begin{equation*}
			\begin{aligned}
				\frac{p_{\theta_* + h/\sqrt{n}}^n}{p_{\theta_*}^n}(X^{n}) \frac{\pi(\theta_* + h/\sqrt{n})}{\pi(\theta_*)}\frac{1}{\mathbbm{1}_{A_{n,0}}}	=  \exp \lbrace h_s v_{st}^n \Delta_{\theta_*, t}^n - (1/2)  v_{st}^n h_s h_t + r_{n,5}(h) \rbrace \frac{1}{\mathbbm{1}_{A_{n,0}}} &\\
				=  \exp\lbrace - (1/2)  v_{st}^n (h-\Delta_{\theta_*}^n)_s (h-\Delta_{\theta_*}^n )_{t} + \gamma_n +r_{n,5}(h) \rbrace \frac{1}{\mathbbm{1}_{A_{n,0}}}&,
			\end{aligned} 
		\end{equation*}
	with $\gamma_n =  v_{st}^n \Delta_{\theta_*,s}^n \Delta_{\theta_*,t}^n/2 >0$ since  $V_{\theta_*}^n$ is positive definite when conditioned on $A_{n,0}$, 
	$$r_{n,5}(h) =r_{n,1}(h) + r_{n,2}(h) + (1/6 \sqrt{n}) a^{(3),n}_{\theta_*,stl} h_s h_t h_l +  (1/\sqrt{n})\log \pi^{(1)}_{\theta_*,s} h_s,  $$ and $r_{n,5}= \sup_{h \in K_n} r_{n,5}(h) = O_{P_0^n}(M_n^{3}/\sqrt{n})$. Notice that we consider $1/\mathbbm{1}_{A_{n,0}}$ instead~of~$\mathbbm{1}_{A_{n,0}}$ since we are currently studying the quantity at the denominator of \eqref{help:lemma:conc:1}.
	
	Define now the event $A_{n,4} = A_{n,0}  \cap \{\|\Delta_{\theta_*}^n\| < \tilde{M}_n\} \cap \{ |r_{n,5}| < \gamma_1 \} $ for some $\tilde{M}_n$ going~to infinity arbitrary slow and $\gamma_1>0$ a fixed positive constant. Since $P_0^n(A_{n,4}) = 1-o(1)$ and $\gamma_n>0$, we can equivalently study the asymptotic behavior of the following lower bound
		\begin{equation} \label{contraction:denom} 
			\begin{aligned}
				& \int_{K_n} (p_{\theta_* + h/\sqrt{n}}^n/p_{\theta_*}^n)(X^{n}) \pi(\theta_* + h/\sqrt{n})/\pi(\theta_*) dh  \frac{1}{\mathbbm{1}_{A_{n,4}}} \\
				 &\quad \geq  e^{-\gamma_1}	\int_{K_n} \exp\{ -   v_{st}^n(h-\Delta_{\theta_*}^n)_s (h-\Delta_{\theta_*}^n)_{t}/2 \} dh  \frac{1}{\mathbbm{1}_{A_{n,4}}}\\
				& \quad \quad   =  \frac{e^{-\gamma_1} (2 \pi )^{d/2}}{|V_{\theta_*}^n|^{1/2}}	\int_{ K_n} \frac{ |V_{\theta_*}^n|^{1/2}}{ (2 \pi )^{d/2}} \exp\{ -  v_{st}^n (h-\Delta_{\theta_*}^n)_s (h-\Delta_{\theta_*}^n)_{t}/2 \} dh  \frac{1}{\mathbbm{1}_{A_{n,4}}}.
			\end{aligned}
		\end{equation}
	
	Due to the fact that, in $K_n$, $M_n \to \infty$, $\Delta_{\theta_*}^n = O_{P_0^n}(1)$ and that, in $\mathbbm{1}_{A_{n,4}}$, the eigenvalues of $V_{\theta_*}^n$ lay on a bounded and positive range, the quantity in the last display is positive and bounded away from zero. This, together with \eqref{contraction:numerator}, proves Lemma~\ref{lemma:post:contr}.
\end{proof}


\begin{proof}[proof of Lemma~\ref{lemma:cond4}]
	To prove Lemma~\ref{lemma:cond4},  let us deal with the cases $\|\theta - \theta_* \|> \delta_1$ and  $ M_n/\sqrt{n}  < \|\theta - \theta_* \| < \delta_1 $ separately. For $\|\theta - \theta_* \|> \delta_1$ the claim trivially follows from \eqref{lemma:5:cond:1}.
	We are left to deal with the case $M_n/\sqrt{n} < \|\theta - \theta_* \| < \delta_1 $. To this end, let us write 
	\begin{equation*}
		\begin{aligned}
			\{ \ell(\theta) - \ell(\theta_*) \}/n = D_{n}(\theta,\theta_*) + \bar{D}_{n}(\theta,\theta_*).
		\end{aligned}
	\end{equation*}
    where $D_n(\theta,\theta_*) = [ \{ \ell(\theta) - \ell(\theta_*)  \} - \E_{0}^n \{ \ell(\theta) - \ell(\theta_*)  \}  ]/n $ and $\bar{D}_n(\theta,\theta_*) =  \E_{0}^n \{ \ell(\theta) - \ell(\theta_*)  \} / n $.
	Note that Assumption R1 implies that $\bar{D}_{n}(\theta,\theta_*)$ is concave in $\|\theta - \theta_*\| < \delta_1$ with Hessian $- I_{\theta_*}/n$. Since $I_{\theta_*}/n$ is positive definite from Assumption \ref{cond:2}, for sufficiently small choice of $\rho>0$ it follows 
	\begin{equation*}
		\begin{aligned}
			\bar{D}_{n}(\theta,\theta_*) \leq &  - \rho \cdot i_{st}(\theta - \theta_*)_s (\theta - \theta_*)_t/n\\
			\leq & - \rho \lambda_{\textsc{min}}( I_{\theta_*}/n)  (\theta - \theta_*)_s  (\theta - \theta_*)_s \leq - \rho \eta_1 \| \theta - \theta_* \|^2.
		\end{aligned}
	\end{equation*}
	In addition, define the event \smash{$A_{n,5} = \{ \sup_{0<\|\theta - \theta_*\| < \delta_1} D_n(\theta,\theta_*)/\|\theta - \theta_*\| < \tilde c_1 \tilde M_n /\sqrt{n} \}$} for a sufficiently large constant $\tilde c_1$, and a sequence $\tilde M_n$ going to infinity arbitrary slow. Notice that $P_0^n( A_{n,5}   ) = 1 -  o(1)$ from Assumption R2. 
	As a consequence, conditioned on $A_{n,5}$, we have for every $\theta$ which meets $0<\|\theta - \theta_*\| < \delta_1$, that 
	\begin{equation*}
		\begin{split}
			D_{n}(\theta,\theta_*) + \bar{D}_{n}(\theta,\theta_*) &\leq \tilde c_1 \|\theta - \theta_*\| \tilde M_n/\sqrt{n} - \rho \eta_1 \| \theta - \theta_* \|^2 \\
			&= \{ \tilde c_1 \tilde{M}_n/(\| \theta - \theta_* \| \sqrt{n}) - \rho \eta_1 \} \| \theta - \theta_* \|^2. 
		\end{split}
	\end{equation*}
	Since $\tilde M_n$ can be chosen such that $\tilde M_n/M_n \to 0 $, the first component in the right--hand--side of the last display becomes always negative for $n$ large enough and the whole expression is asymptotically maximized when $\| \theta - \theta_* \|$ is at its minimum. This implies
	$$ \textstyle \sup_{ M_n/\sqrt{n} < \|\theta - \theta_*\| <  \delta_1  } 	D_{n}(\theta,\theta_*) + \bar{D}_{n}(\theta,\theta_*) \leq  - c_5 M_n^2/n, $$ 
	for some $c_5>0$, with $P_0^n$--probability tending to one. This concludes the proof.
\end{proof}



\begin{proof}[Proof of Theorem~\ref{thm:marginal}]
	Define	\smash{$\hat K_{n,\cc} = \{ \hat h_{\cc} \, : \, \|\hat h_{\cc}\| < 2 M_n \} $}. 
	The \textsc{tv} distance among \smash{$\pi_{n,\cc}$ and $\hat p_{\textsc{sks},\cc}^n$} is \smash{$ (1/2) \int |\pi_{n,\cc}(\hat h_{\cc}) - \hat p_{\textsc{sks},\cc}^n(\hat h_{\cc}) |d  \hat h_{\cc}.$} 
	Adding and subtracting~$\int\hat  p_{\textsc{sks}}^n(\hat h)d \hat h_{\bar \cc}$, and by exploiting Jensen's and triangle inequality, we obtain the following upper bound
	\begin{equation*}
		\begin{aligned}
			\int |\pi_{n}(\hat h) - \hat p_{\textsc{sks}}^n(\hat h)|d  \hat h 
			 			+ \int |\int \hat p_{\textsc{sks}}^n(\hat h)d \hat h_{\bar \cc} - \hat p_{\textsc{sks},\cc}^n(\hat h_{\cc}) |d \hat h_{\cc}.
		\end{aligned}
	\end{equation*} 
	It follows from Theorem \ref{thm:modal} that 
	\begin{equation} \label{margin:help1}
		\int |\pi_{n}(\hat h) - \hat p_{\textsc{sks}}^n(\hat h)|d  \hat h = O_{P_0^n}(M_n^{c_8}/n),
	\end{equation}
	for some $c_8 >0$. Therefore, it is sufficient to study $ \int |\int \hat p_{\textsc{sks}}^n(\hat h)d \hat h_{\bar \cc} - p_{\textsc{sks},\cc}^n(\hat h_{\cc}) |d \hat h_{\cc}.$ Note that, as a direct consequence of Equation~\eqref{skew:marginal}, we have
	$$ \int \hat p_{\textsc{sks}}^n(\hat h)d \hat h_{\bar \cc} - \hat p_{\textsc{sks},\cc}^n(\hat h_{\cc})  = 2 \phi_{d_{\cc}}(\hat h_{\cc}; 0, \hat \Omega_{ \cc \cc}) \E_{\hat h_{\bar \cc }| \hat h_{\cc}}[ F(\hat{\alpha}_{\eta}(\hat h)) - F(\E_{\hat h_{\bar \cc } | \hat h_{\cc} }\{\hat{\alpha}_{\eta}(\hat h)\}) ].$$
	Let $C_{n,0}  = \{\lambda_{\textsc{min}}( \hat \Omega_{ \cc \cc} ) > \eta_{1,\cc} \} \cap \{\lambda_{\textsc{max}}(  \hat \Omega_{ \cc \cc} ) < \eta_{2,\cc} \} $ for some fixed $\eta_{1,\cc},\eta_{2,\cc}>0$. From Assumption \ref{M2} and Lemma \ref{lemma:eigen:restricted:matrix} it follows $P_0^n C_{n,0} = 1-o(1)$. Hence, let us condition on $C_{n,0}$, and split the integral 
	\begin{equation*}
		\begin{aligned}
			& \int \Big|2 \phi_{d_{\cc}}(\hat h_{\cc}; 0,\hat \Omega_{ \cc \cc}) \E_{\hat h_{\bar \cc }| \hat h_{\cc}}[ F(\hat{\alpha}_{\eta}(\hat h)) - F(\E_{\hat h_{\bar \cc } | \hat h_{\cc} }\{\hat{\alpha}_{\eta}(\hat h)\}) ] \Big| d\hat h_{\cc} \mathbbm{1}_{ C_{n,0}},
		\end{aligned}
	\end{equation*}
	between $\hat K_{n,\cc}$ and $\hat K_{n,\cc}^c$. From the  tail behavior of the  Gaussian distribution and the  boundedness of \smash{$ \E_{\hat h_{\bar \cc }| \hat h_{\cc}}[ F(\hat{\alpha}_{\eta}(\hat h)) - F(\E_{\hat h_{\bar \cc } | \hat h_{\cc} }\{\hat{\alpha}_{\eta}(\hat h)\})]$} if follows that 
		$$ \int_{\hat h_{\cc} \in \hat K_{n,\cc}^c } \Big|2 \phi_{d_{\cc}}(\hat h_{\cc}; 0,\hat \Omega_{ \cc \cc}) \E_{\hat h_{\bar \cc }| \hat h_{\cc}}[ F(\hat{\alpha}_{\eta}(\hat h)) - F(\E_{\hat h_{\bar \cc } | \hat h_{\cc} }\{\hat{\alpha}_{\eta}(\hat h)\}) ] \Big| d\hat h_{\cc}  \mathbbm{1}_{ C_{n,0}} \leq  4 e^{- \tilde c_1 M_n^2},$$
	for some constant $\tilde c_1 >0$, which in turn implies 
	\begin{equation} \label{margin:help2}
	 \scalemath{0.97}{	 \int_{\hat h_{\cc} \in \hat K_{n,\cc}^c } \Big|2 \phi_{d_{\cc}}(\hat h_{\cc}; 0,\hat \Omega_{ \cc \cc}) \E_{\hat h_{\bar \cc }| \hat h_{\cc}}[ F(\hat{\alpha}_{\eta}(\hat h)) - F(\E_{\hat h_{\bar \cc } | \hat h_{\cc} }\{\hat{\alpha}_{\eta}(\hat h)\}) ] \Big| d\hat h_{\cc}  = O_{P_0^n}(n^{-1}),}
	\end{equation}
	for a sufficiently large constant $c_0$ in $M_n$.
	In addition, from Lemma \ref{lemma:joint:margin:diff} it follows  that
	$$ \textstyle \sup_{\hat h_{\cc} \in \hat K_{n,\cc}}|  \E_{\hat h_{\bar \cc }| \hat h_{\cc}}[ F(\hat{\alpha}_{\eta}(\hat h)) - F(\E_{\hat h_{\bar \cc } |\hat h_{\cc} }\{\hat{\alpha}_{\eta}(\hat h)\}) ] |= O_{P_0^n}(M_n^{c_{10}}/n),$$
	for some $c_{10} > 0$, which implies
		\begin{equation} \label{margin:help3}
			\begin{aligned}
					\int_{\hat h_{\cc} \in \hat K_{n,\cc} }  \Big |2 \phi_{d_{\cc}}(\hat h_{\cc}; 0,\hat \Omega_{ \cc \cc}) \E_{\hat h_{\bar \cc }| \hat h_{\cc}}[ F(\alpha_{\eta}(\hat h)) - F(\E_{\hat h_{\bar \cc } | \hat h_{\cc} }\{\alpha_{\eta}(\hat h)\}) ] \Big| d\hat h_{\cc}  \mathbbm{1}_{ C_{n,0}} = O_{P_0^n}(M_{n}^{c_{10}}/n).
			\end{aligned}
		\end{equation}
	The combination of \eqref{margin:help1}, \eqref{margin:help2} and \eqref{margin:help3} concludes the proof with $c_9 = c_8 \vee c_{10}$.
\end{proof}


\section{Technical lemmas}
In the following, we state and prove the technical lemmas required for the proofs of the theoretical results in the article. 

\vspace{-4pt}
\begin{lemma} \label{lemma:1}
	Let $F$ be the cdf of a univariate random variable on $\mathbbm{R}$ such that $F(-x) = 1-F(x)$, $F(0) =  1/2$ and $F(x) = F(0) + \eta x + O(x^2)$ for some $\eta \in \mathbbm{R}$.	
	Under Assumptions \ref{cond:LAN} and \ref{cond:prior} it follows that 
	\begin{equation}
	 \scalemath{0.92}{	\begin{aligned}
			&\log \Big[\frac{p_{\theta_* + \delta_n h }^n}{p_{\theta_*}^n}(X^n)\frac{ \pi(\theta_* + \delta_n h )}{\pi(\theta_*)} \Big]+ \frac{\omega^{-1}_{st} }{2}(h-\xi)_s  (h-\xi)_t - \log 2 w(h -\xi) + \delta = r_{n,4}(h),
		\end{aligned}}
		\vspace{-5pt}
	\end{equation}
	with $\delta$ a constant not depending on $h,$ $\xi = \Delta_{\theta_*}^n + \delta_n (V_{\theta_*}^n)^{-1} \log \pi^{(1)} ,$ $\Omega^{-1}=[\omega^{-1}_{st}] = [v_{st}^n - \delta_n \smash{a^{(3),n}_{\theta_*,stl} \xi_l}]$, and $ w(h -\xi)=F(\alpha_{\eta}(h-\xi))$, where $\alpha_{\eta}(h-\xi) = (\delta_n/12 \eta)\smash{a^{(3),n}_{\theta_*, stl}} \{  (h-\xi)_s   (h-\xi)_t  (h-\xi)_l + 3  (h-\xi)_s\xi_t \xi_l    \}$. Moreover,
	\begin{equation}
	\textstyle	r_{n,4} := \sup_{h \in K_n}r_{n,4}(h) = O_{P_0^n}(\delta_n^2 M_n^{c_3}),
			\vspace{-3pt}
	\end{equation}  
	for some constant $c_3>0$.
\end{lemma}

\vspace{-16pt}
\begin{proof}
	We start by noting that Assumptions \ref{cond:LAN} and \ref{cond:prior} imply
	\begin{equation} \label{lemma1:help1}
		\begin{aligned}
			& \log\Big[\frac{p_{\theta_* + \delta_n h }^n}{p_{\theta_*}^n}(X^n)\frac{ \pi(\theta_* + \delta_n h )}{\pi(\theta_*)}\Big] - h_s v_{st}^n \Delta_{\theta_*, t}^n - \delta_n h_s \log \pi^{(1)}_s  \\
			& \qquad \qquad \qquad   + \frac{1}{2}   v_{st}^n h_s h_{t} - \frac{\delta_n}{6} a^{(3),n}_{\theta_*,stl} h_s h_t h_l  = r_{n,1}(h) + r_{n,2}(h).
		\end{aligned}
			\vspace{-5pt}
	\end{equation}
	Furthermore, note that
	\begin{equation}  \label{lemma1:help2}
		h_s v_{st}^n \Delta_{\theta_*, t}^n + \delta_n h_s \log \pi^{(1)}_s - (1/2)v_{st}^n h_s  h_{t} = -  v_{st}^n(h-\xi)_s   (h-\xi)_t/2 + \delta_1,
	\end{equation}
	where $\xi = \Delta_{\theta_*}^n + \delta_n (V_{\theta_*}^n)^{-1} \log \pi^{(1)}$ and $\delta_1$ is a quantity not depending on $h$.
	
	Let us now add and subtract $\xi$ from $h$ in the three dimensional array part, obtaining
	\begin{equation}  \label{lemma1:help3}
		\begin{aligned}
			\frac{\delta_n}{6} a^{(3),n}_{\theta_*,stl} h_s h_t h_l  = &\  \frac{ \delta_n}{6} a^{(3),n}_{\theta_*,stl}(h-\xi)_s (h-\xi)_t (h-\xi)_l+\frac{3 \delta_n}{6} a^{(3),n}_{\theta_*,stl} (h-\xi)_s \xi_t \xi_l\\
			&+\frac{3 \delta_n}{6}  a^{(3),n}_{\theta_*,stl} (h-\xi)_s (h-\xi)_t \xi_l+ \delta_2,
		\end{aligned}
	\end{equation}  
	where $\delta_2$ does not depend on $h$. Combining \eqref{lemma1:help2} and \eqref{lemma1:help3} it is possible to rewrite \eqref{lemma1:help1} as
	\begin{equation} \label{lemma1:help4}
		\begin{aligned}
			& \log \Big[ \frac{p_{\theta_* + \delta_n h }^n}{p_{\theta_*}^n}(X^n)\frac{ \pi(\theta_* + \delta_n h )}{\pi(\theta_*)} \Big] +  \omega^{-1}_{st}(h-\xi)_s   (h-\xi)_t/2  \\
			&\quad  - \frac{\delta_n}{6} \Psi^{(3)}_{stl} (h-\xi)_s (h-\xi)_t (h-\xi)_l-\frac{3 \delta_n}{6} \Psi^{(1)}_{s} (h-\xi)_s  + \delta   = r_{n,1}(h) + r_{n,2}(h),
		\end{aligned}
	\end{equation}
	with $ \Omega^{-1} = V_{\theta_*}^n - \delta_n \Psi^{(2)}, $ $\Psi^{(3)} = [a^{(3),n}_{\theta_*,stl}]$, $\Psi^{(2)} = [  a^{(3),n}_{\theta_*,stl}  \xi_l ] $, $\Psi^{(1)} =  [a^{(3),n}_{\theta_*,stl} \xi_t \xi_l ] $,  $\delta = - \delta_1 - \delta_2 $. 
	
	To conclude the proof of the lemma note that, Assumption \ref{cond:LAN}, the fact that the parameter dimension $d$ is fixed, and the Cauchy--Schwarz inequality imply 
	$$ \frac{\delta_n}{6} \Psi^{(3)}_{stl} (h-\xi)_s (h-\xi)_t (h-\xi)_l + \frac{3 \delta_n}{6} \Psi^{(1)}_{s} (h-\xi)_s = O_{P_0^n}(\delta_n\{\|h\|^3 \vee 1 \}). $$
	By exploiting the conditions imposed on $F(\cdot)$, let us write
	\begin{equation*}
		\begin{aligned}
			&F[  (\delta_n/6) \big \lbrace \Psi^{(3)}_{stl} (h-\xi)_s (h-\xi)_t (h-\xi)_l + 3 \Psi^{(1)}_{s} (h-\xi)_s \big \rbrace  ]   \\
			&=  \frac{1}{2}\left[ 1 + 2 \eta \frac{\delta_n}{6} \big \lbrace \Psi^{(3)}_{stl} (h-\xi)_s (h-\xi)_t (h-\xi)_l + 3 \Psi^{(1)}_{s} (h-\xi)_s \big \rbrace  + O_{P_0^n}(\delta_n^2 \{\|h\|^6 \vee  1 \}) \right] 
		\end{aligned}
	\end{equation*}
	 since the argument of $F(\cdot)$ converges to zero in probability. An additional Taylor expansion, this time $\log(1+x) = x +O(x^2)$ for $x \to 0,$ gives 
	\begin{equation} \label{lemma1:help5}
		\begin{aligned}
			& \log 2F	[ (\delta_n/12 \eta) \big \lbrace \Psi^{(3)}_{stl} (h-\xi)_s (h-\xi)_t (h-\xi)_l + 3  \Psi^{(1)}_{s} (h-\xi)_s \big \rbrace ] \\
			& = \log [  1 + (\delta_n/6) \big \lbrace \Psi^{(3)}_{stl} (h-\xi)_s (h-\xi)_t (h-\xi)_l + 3 \Psi^{(1)}_{s} (h-\xi)_s \big \rbrace  + O_{P_0^n}(\delta_n^2 \{\|h\|^6 \vee  1 \}) ] \\
			&=   (\delta_n/6) \big \lbrace \Psi^{(3)}_{stl} (h-\xi)_s (h-\xi)_t (h-\xi)_l + 3 \Psi^{(1)}_{s} (h-\xi)_s \big \rbrace  + \tilde r_{n,1}(h)
		\end{aligned}
	\end{equation}
    where the remainder term $ \tilde r_{n,1}(h)$ is $O_{P_0^n}(\delta_n^2 \{\|h\|^6 \vee  1 \})$.
     Note that, when restricted on $K_n$,  $\tilde r_{n,1} = \sup_{h \in K_n}\tilde r_{n,1}(h) = O_{P_0^n}(\delta_n^2 M_n^6)$. This fact combined with \eqref{lemma1:help4} and \eqref{lemma1:help5} gives
	\begin{equation*}
		\begin{aligned}
			&\log  \Big[  \frac{p_{\theta_* + \delta_n h }}{p_{\theta_*}}(X^n)\frac{ \pi(\theta_* + \delta_n h )}{\pi(\theta_*)} \Big]+ \omega^{-1}_{st} (h-\xi)_s  (h-\xi)_t/2 \\
			& \quad - \log 2 F[ (\delta_n/12 \eta) \{\Psi^{(1)}_s (h-\xi)_s + \Psi^{(3)}_{stl} (h-\xi)_s   (h-\xi)_t  (h-\xi)_l \} ] + \delta = r_{n,4}(h),
		\end{aligned}
	\end{equation*}
	where $r_{n,4}(h) = r_{n,1}(h) + r_{n,2}(h) - \tilde r_{n,1}(h).$ Then, Assumptions \ref{cond:LAN}--\ref{cond:prior}  imply that \begin{equation}
		\textstyle \sup_{h \in K_n} r_{n,4}(h) = O_{P_0^n}(\delta_n^2 M_n^{c_3}),
	\end{equation}  
	for some constant $c_3>0$, concluding the proof.
\end{proof}


\begin{lemma} \label{lemma:eigen:J}
    Let $\hat A$ and $A$ denote two $d \times d$ real symmetric matrices. Suppose that  $\hat A$ is random with entries $ 	\hat a_{st} = O_{P_0^n}(1)$, whereas $A$ is non--random and has elements $	a_{st} = O(1) $. Moreover, assume
    \begin{equation} \label{cond:diff:matrix}
    	a_{st} - \hat a_{st} = O_{P_0^n}(\delta_n),
    \end{equation}  
   for some norming rate $ \delta_n \to 0$ and $s,t \in \{1,\dots,d\}$.
    If there exist two positive constants $\eta_1$ and $\eta_2$ such that $\lambda_{\textsc{min}}(A) > \eta_1$ and $\lambda_{\textsc{max}}(A) < \eta_2$ then, with $P_{0}^n$--probability tending to $1$, there exist two positive constants $\eta_1^*$ and $\eta_2^*$ such that $\lambda_{\textsc{min}}(\hat A)> \eta_1^*$ and $\lambda_{\textsc{max}}(\hat A) < \eta_2^*$. 
\end{lemma}

\begin{proof}
	Leveraging \eqref{cond:diff:matrix}, let us first notice that  \smash{$\hat A  = A + R$} with $R$ having entries of order \smash{$\mathcal{O}_{P_{0}^n} ( \delta_n).$} As a consequence, there exist constants $\tilde c_1 >0 $ and $\tilde c_2 > 1$   such that 
	$$ P_{0}^n( | R_{{st}} | > \tilde c_1 \delta_n^{\tilde c_2} ) = o(1),$$
	for every $s,t=1, \ldots, d$. Define now the matrix $M$ having entries $ M_{st} = | R_{st} | \land  \tilde c_1 \delta_n^{\tilde c_2} $ for  $s,t=1, \ldots, d$.
	From Wielandt’s theorem \citep{zwillinger2007table}, with probability $1-o(1),$ the spectral radius of $M$ is an upper bound of the spectral radius of $R.$ Moreover, since $M$ is a non--negative matrix, the Perron--Frobenius theorem \citep{perron1907theorie,frobenius1912matrizen} implies that the largest eigenvalue in absolute value is bounded by constant times $ \delta_n^{\tilde c_2} .$
	Since both $A$ and $R$ are real symmetric matrices, in view of the Weyl's inequalities \citep[e.g.,][equation (1.54)]{tao2011topics}, the eigenvalues of $\hat A$ and $A$ can differ at most by constant times $\delta_n^{\tilde c_2}$ with probability $1-o(1).$ As a consequence, since the lemma assumes the existence of two positive constants $\eta_1$ and $\eta_2$ such that $\lambda_{\textsc{min}}(A) > \eta_1$ and $\lambda_{\textsc{max}}(A) < \eta_2$, it follows that there exist  $\eta_1^*,\eta_2^* > 0 $ such that, with probability  $1-o(1),$ $\lambda_{\textsc{min}}(\hat A) > \eta_1^*$ and $\lambda_{\textsc{max}}(\hat A) < \eta_2^*$.
\end{proof}



\vspace{4pt}
\begin{lemma} \label{lemma:eigen:restricted:matrix}
	Let $A$ be a $d \times d$ symmetric positive definite matrix satisfying $\lambda_{\textsc{min}}(A) \ge \eta_{1A}$ and $\lambda_{\textsc{max}}(A) \le \eta_{2A}$. Let $\mathcal{S} \subseteq \{1,\dots d\}$ be a set of indexes having cardinality $d_*$, and $B$ be the $d_* \times d_*$ submatrix obtained by keeping only rows and columns of $A$ whose position is in $\mathcal{S}$. Then it holds that $\lambda_{\textsc{min}}(B) \ge \eta_{1A}$ 
	and $\lambda_{\textsc{max}}(B) \le \eta_{2A}$.
\end{lemma}
\begin{proof}
	Without loss of generality assume that the elements of $\mathcal{S}$ are in increasing order. Note, also, that $B = S A S^T $ where $S$ is a $d_* \times d$ matrix having entries $s_{ij} = 1$ if $j = \mathcal{S}_i$, where $\mathcal{S}_i$ denotes the $i$-th element of $\mathcal{S}$, and $s_{ij} = 0$ otherwise. Recall that, from the relation between minimum and maximum eigenvalues and the Rayleigh quotient it follows that 
	\begin{equation*}
		\lambda_{\textsc{min}}(A) = \min_{ x \in \mathbb{R}^d } \frac{x^T A x}{x^T x}, \quad \text{and} \quad 	\lambda_{\textsc{max}}(A) = \max_{ x \in \mathbb{R}^d } \frac{x^T A x}{x^T x}.
	\end{equation*}
	Similarly, leveraging $S S^T = I_{d_*}$, it holds for $B$ that
	\begin{equation*}
		\begin{aligned}
			\lambda_{\textsc{min}}(B) = \min_{ x_* \in \mathbb{R}^{d_*} } \frac{x_*^T B x_*}{x_*^T x_*} =  \min_{ x_* \in \mathbb{R}^{d_*} } \frac{x_*^T S A S^T  x_*}{x_*^T  S  S^T x_*} \ge  \min_{ x \in \mathbb{R}^d } \frac{x^T A x}{x^T x} = 	\lambda_{\textsc{min}}(A),
		\end{aligned}
	\end{equation*}
	with the inequality that follows from the fact that $\{ x \in \mathbbm{R}^d \, : \, x =  x_*^T  S \, \text{for}\, x_* \in \mathbb{R}^{d_*} \} \subseteq \mathbb{R}^{d}.$ Following the same line of reasoning it is possible to prove $\lambda_{\textsc{max}}(A) \geq \lambda_{\textsc{max}}(B)$. This~concludes the proof.	
\end{proof}



\vspace{4pt}
\begin{lemma} \label{lemma:joint:margin:diff}
	Under the assumptions stated in Theorem \ref{thm:marginal}, for every 	univariate cdf $F(\cdot)$ satisfying $F(-x) = 1-F(x)$, $F(0) =  1/2$ and $F(x) = F(0) + \eta x + O(x^2)$,  it holds
	\begin{equation*}
	 \sup_{\hat h_{\cc} \in \hat K_{n,\cc}} \big| \E_{\hat h_{\bar \cc }| \hat h_{\cc}} F(\hat{\alpha}_{\eta}(\hat h)) - F(\E_{\hat h_{\bar \cc } |\hat  h_{\cc} }\{\hat{\alpha}_{\eta}(\hat h)\}) \big| =  O_{P_0^n}(M_n^{c_{10}}/n),
	\end{equation*}
	where  $\hat K_{n,\cc} = \{ \hat h_{\cc} \, : \, \|\hat h_{\cc}\| < 2 M_n \}$, $c_{10} $ denotes a positive constant and $\hat{\alpha}_{\eta}(\hat h)$ is defined as  in Section \ref{sec_31}. 
\end{lemma}

\begin{proof}
	Recall that the covariance matrix  $\bar \Omega$, associated to the Gaussian measure \smash{$P_{\hat h_{\bar \cc} |\hat h_{ \cc}} $}, is the Schur complement of the block $\hat \Omega_{\cc\cc}$ of the matrix $\hat \Omega$. This implies that, in view of Assumption \ref{M2}, Lemma \ref{lemma:eigen:restricted:matrix} and the properties of the Schur complement, there exist constants $0<\tilde c_1<\tilde c_2<\infty$ such that the eigenvalues of $\bar\Omega$ are bounded from below by $\tilde c_1$ and above by $\tilde c_2$ with probability tending to one.
	
	As a consequence, there exists a large enough constant $c_{0,\bar \cc}>0$ such that the complement of the set \smash{$\hat K_{n, \bar \cc } = \{ \hat h_{\bar \cc} \, : \, \|\hat h_{\bar \cc} - \E_{\hat h_{\bar \cc } | \hat h_{\cc} }(\hat h_{\bar \cc}) \| < 2 \sqrt{  c_{0,\bar \cc} \log n}   \},$} has negligible mass, i.e.,  the event \smash{$L_{n,0} =  \big \lbrace P_{\hat h_{\bar \cc} |\hat h_{ \cc  }}(\hat h_{ \bar\cc} \in \hat K_{n, \bar \cc }^c) < \tilde c_3/ n ,\, \text{for} \, \hat{h}_{ \cc} \in \hat K_{n,\cc} \big \rbrace$} has probability $P_0^n(L_{n,0}) = 1-o(1)$.
    This fact and the boundedness of $F(\cdot)$ imply 
    \begin{equation} \label{help:margin:ord0} 
    \begin{aligned}
    	 &\textstyle\sup_{\hat h_{\cc} \in \hat K_{n,\cc}} \big| \E_{\hat h_{\bar \cc }| \hat h_{\cc}}[ F(\hat{\alpha}_{\eta}(\hat h)) - F(\E_{\hat h_{\bar \cc } |\hat  h_{\cc} }\{\hat{\alpha}_{\eta}(\hat h)\}) \big|    \mathbbm{1}_{L_{n,0}} \\
    	 &\textstyle \leq \sup_{\hat h_{\cc} \in \hat K_{n,\cc}}  \big| \E_{\hat h_{\bar \cc }| \hat h_{\cc}} \big \lbrace F(\hat{\alpha}_{\eta}(\hat h)) - F(\E_{\hat h_{\bar \cc } |\hat  h_{\cc} } \hat{\alpha}_{\eta}(\hat h)) \big \rbrace  \mathbbm{1}_{ \hat h_{\bar \cc } \in \hat K_{n,\bar \cc} }  \big| \mathbbm{1}_{ L_{n,0}} + O(n^{-1}).
    \end{aligned}
    \end{equation}
    We further decompose the first term on the right--hand--side of the last display by adding and subtracting the~first~order Taylor expansion of $F(\hat{\alpha}_{\eta}(\hat h))$. Moreover, recall that under the assumptions of Theorem \ref{thm:marginal}, there exists $\tilde c_3>0$ such that  \smash{$L_{n,1} = L_{n,0} \cap \{ \max_{s,t,l} |a^{(3),n}_{\hat \theta, stl}| < \tilde c_3 \}$} holds with probability $P_0^n L_{n,1} = 1-o(1)$.
    Since  $\|\hat h \| \leq \|\hat h_{\cc}\| + \|\hat h_{\bar \cc}\| $ it follows also that, conditioned on $L_{n,1}$, both \smash{$\sup_{ \hat h \in  \hat K_{n,\cc} \cap \hat K_{n,\bar \cc} }|\hat{\alpha}_{\eta}(\hat h)|$} and \smash{$ \sup_{ \hat h_{\cc} \in  \hat K_{n,\cc}}|\E_{\hat h_{\bar \cc }| \hat h_{\cc}}\{ \hat{\alpha}_{\eta}(\hat h)\}|$} are of order $O(M_n^{\tilde c_4}/\sqrt{n})$, while $ \sup_{ \hat h_{\cc} \in  \hat K_{n,\cc}} \E_{\hat h_{\bar \cc }| \hat h_{\cc}}\{ \hat{\alpha}_{\eta}(\hat h)^2\} =O(M_n^{\tilde c_5}/n)$  for some $\tilde c_4,\tilde c_5 >0$. This, together with \smash{$F(x) = 1/2+\eta x + O(x^2), \, x \to 0$}, implies for $n$ sufficiently large
    \begin{equation*}
    	\begin{aligned}
    	&\textstyle \sup_{\hat h_{\cc} \in \hat K_{n,\cc}}	\big| \E_{\hat h_{\bar \cc }| \hat h_{\cc}} \big \lbrace F(\hat{\alpha}_{\eta}(\hat h)) -  1/2 - \eta \hat{\alpha}_{\eta}(\hat h)  \big \rbrace   \mathbbm{1}_{\ \hat h_{\bar \cc } \in \hat K_{n,\bar \cc} } \big| \mathbbm{1}_{L_{n,1}} \\
    	&\textstyle \lesssim \, \sup_{\hat h_{\cc} \in \hat K_{n,\cc}} \E_{\hat h_{\bar \cc }| \hat h_{\cc}}\{ \hat{\alpha}_{\eta}(\hat h)^2\} \mathbbm{1}_{L_{n,1}}   = O(M_n^{\tilde c_5}/n).
    	\end{aligned}
    \end{equation*}
    
    As a consequence, 
    \begin{equation} \label{help:margin:ord1}
 	\textstyle   \sup_{\hat h_{\cc} \in \hat K_{n,\cc}}\big|	\E_{\hat h_{\bar \cc }| \hat h_{\cc}} \big \lbrace F(\hat{\alpha}_{\eta}(\hat h)) -  1/2 - \eta \hat{\alpha}_{\eta}(\hat h)  \big \rbrace   \mathbbm{1}_{\hat h_{\bar \cc } \in \hat K_{n,\bar \cc} } \big| = O_{P_0^n}(M_n^{\tilde c_5}/n). 
    \end{equation}
    
    To conclude, note that for $n$ sufficiently large,  
     \begin{equation*}
 \scalemath{0.97}{	    	\begin{aligned}
    		&\textstyle \sup_{\hat h_{\cc} \in \hat K_{n,\cc}}  | \E_{\hat h_{\bar \cc }| \hat h_{\cc}} \big \lbrace  F(\E_{\hat h_{\bar \cc } |\hat  h_{\cc} } \hat{\alpha}_{\eta}(\hat h)) -  1/2 - \eta \hat{\alpha}_{\eta}(\hat h)  \big \rbrace   \mathbbm{1}_{ \hat h_{\bar \cc } \in \hat K_{n,\bar \cc} }   \mathbbm{1}_{ L_{n,1}} | \\
    		&\leq \textstyle \sup_{\hat h_{\cc} \in \hat K_{n,\cc}} | \E_{\hat h_{\bar \cc }| \hat h_{\cc}}[ \eta \big \lbrace  \E_{\hat h_{\bar \cc } |\hat  h_{\cc} } (\hat{\alpha}_{\eta}(\hat h) ) - \hat{\alpha}_{\eta}(\hat h)  \big \rbrace   \mathbbm{1}_{ \hat h_{\bar \cc } \in \hat K_{n,\bar \cc} }  ] \mathbbm{1}_{L_{n,1}} |\\
		&\qquad \qquad \qquad \qquad \qquad \qquad \qquad \qquad \qquad \qquad \qquad + O(\{\E_{\hat h_{\bar \cc }| \hat h_{\cc}} \hat{\alpha}_{\eta}(\hat h)\}^2  \mathbbm{1}_{L_{n,1}}  )\\
    		&\textstyle = \sup_{\hat h_{\cc} \in \hat K_{n,\cc}} | \E_{\hat h_{\bar \cc }| \hat h_{\cc}}[ \eta \big \lbrace  \E_{\hat h_{\bar \cc } |\hat  h_{\cc} } (\hat{\alpha}_{\eta}(\hat h))  - \hat{\alpha}_{\eta}(\hat h)  \big \rbrace   \mathbbm{1}_{ \hat h_{\bar \cc } \in \hat K_{n,\bar \cc}^c }  ] \mathbbm{1}_{L_{n,1}} | + O(M_n^{2 \tilde c_4}/n  )\\
    		&\textstyle \leq \sup_{\hat h_{\cc} \in \hat K_{n,\cc}} \eta [ \E_{\hat h_{\bar \cc }| \hat h_{\cc}}  \big \lbrace  \E_{\hat h_{\bar \cc } |\hat  h_{\cc} }( \hat{\alpha}_{\eta}(\hat h))  - \hat{\alpha}_{\eta}(\hat h)  \big \rbrace^2]^{\frac{1}{2}} [ \E_{\hat h_{\bar \cc }| \hat h_{\cc}}  \mathbbm{1}_{ \hat h_{\bar \cc } \in \hat K_{n,\bar \cc}^c } ]^{\frac{1}{2}}    \mathbbm{1}_{ L_{n,1}} + O(M_n^{2 \tilde c_4}/n ),\\
    	\end{aligned}}
    \end{equation*} 
    with the first inequality that follows from the Taylor expansion of $F(\cdot)$ at 0, the equality from $ \mathbbm{1}_{ \hat h_{\bar \cc } \in \hat K_{n,\bar \cc} }   = 1 -  \mathbbm{1}_{ \hat h_{\bar \cc } \in \hat K_{n,\bar \cc}^c } $ and the last line from the Cauchy--Schwarz inequality.
    Finally, note that \smash{$[\E_{\hat h_{\bar \cc }| \hat h_{\cc}}  \mathbbm{1}_{ \hat h_{\bar \cc } \in \hat K_{n,\bar \cc}^c } ]^{1/2}    \mathbbm{1}_{ L_{n,1}} = O(n^{-1/2})$} while  
    $$\textstyle \sup_{\hat h_{\cc} \in \hat K_{n,\cc}}  [ \E_{\hat h_{\bar \cc }| \hat h_{\cc}}  \big \lbrace  \E_{\hat h_{\bar \cc } |\hat  h_{\cc} }( \hat{\alpha}_{\eta}(\hat h))  - \hat{\alpha}_{\eta}(\hat h)  \big \rbrace^2]^{\frac{1}{2}} \mathbbm{1}_{ L_{n,1}}= O(M_n^{\tilde c_6}/\sqrt{n}), $$  for some $\tilde c_6>0$ large enough. From the previous display it follows
    \begin{equation}\label{help:margin:ord2}
    	\begin{aligned}
    		\textstyle \sup_{\hat h_{\cc} \in \hat K_{n,\cc}} \big|	\E_{\hat h_{\bar \cc }| \hat h_{\cc}}[ \big \lbrace  F(\E_{\hat h_{\bar \cc } |\hat  h_{\cc} } \hat{\alpha}_{\eta}(\hat h)) -  1/2 - \eta \hat{\alpha}_{\eta}(\hat h)  \big \rbrace   \mathbbm{1}_{ \hat h_{\bar \cc } \in \hat K_{n,\bar \cc} } \big|= O_{P_0^n}(M_n^{\tilde c_{7}}/n). 
    	\end{aligned}
    \end{equation}
    for a $\tilde c_{7}$ large enough. Combining \eqref{help:margin:ord0}, \eqref{help:margin:ord1} and \eqref{help:margin:ord2} concludes the proof.
\end{proof}

\section{Non--Asymptotic bounds and proof of Theorem \ref{thm:modal}}\label{sec:proof:finite:sample}
The proof of Theorem \ref{thm:modal} follows as a direct consequence of a non--asymptotic version of this result, which we derive in Theorem \ref{thm:modal:fs} below. More precisely, we show~that, for a large enough $n$, the \textsc{tv} distance between the skew--modal approximation and the target posterior is upper bounded by a constant times $M_n^6d^3/n$, on an event $\hat{A}_{n,4}$ with $P_0^n(\hat{A}_{n,4})=1-o(1)$. When $d$ is fixed and $n \to \infty$, this result implies the $O_{P^n_0}(M_n^{c_8}/n)$ rate stated in Theorem \ref{thm:modal}, with $c_8=6$. Similarly, we provide also a non--asymptotic upper bound for the approximation error of functionals of the posterior that implies the asymptotic rate stated in \eqref{result:corol:functionals}.

To prove the above result, let us introduce some additional notations. For $\delta>0$, define
\begin{align}
\scalemath{0.97}{	 L_{\pi, \delta} := \sup\nolimits_{\theta \, : \, \|\theta - \theta_*\| < \delta} | \log \pi(\theta_*) - \log \pi(\theta )|<\infty \ \ \text{and}\ \
C_{\pi,\delta} = \inf\nolimits_{\theta : \|\theta-\theta_{*}\|< \delta} \pi(\theta)>0,\label{def:L_pi}}
\end{align}
along with the event
	\begin{equation}\label{def:A_3}
		 \hat{A}_{n,3} =  \lbrace \sup\nolimits_{\|\theta - \hat \theta\|> 2 M_n\sqrt{d}/\sqrt{n}}\{ (\ell(\theta) - \ell(\hat \theta))/n  \} < - c_5 M_n^2 d/n +  L_{\pi, \delta}/n \rbrace,
	\end{equation}
	providing bound on the logarithm of the likelihood--ratio outside of the $ 2 M_n\sqrt{d}/\sqrt{n}$--radius ball centered around the \textsc{map}. Then, for convenience, let us denote the intersection of the events in Assumptions \ref{M1}--\ref{M2} and the one above by
	\begin{equation}\label{def:A_4}
		\hat{A}_{n,4} =   \hat{A}_{n,0} \cap \hat{A}_{n,1} \cap \hat{A}_{n,2} \cap \hat{A}_{n,3}.
	\end{equation}

Finally, recall that in the definition of the skew--symmetric approximation in \eqref{limiting:distribution}, we consider a univariate cdf $F(\cdot)$ satisfying $F(-x) = 1-F(x)$ and $F(0) =  1/2$, while admitting a suitable expansion. By making more explicit such an expansion, let us assume  that for~some $\eta \in \mathbbm{R}$ and a sufficiently small  $\delta>0$ it holds 
\begin{align}
F(x) -1/2 - \eta x = r_{F,\delta}(x), \qquad \text{with $| r_{F,\delta}(x)  |< L_{F,\delta} x^2$.} \label{assump:F}
\end{align}
 for $\{x \,: \, |x|< \delta\}$.
\begin{theorem} \label{thm:modal:fs}
Suppose that Assumptions \ref{cond:uni}, \ref{cond:3}--\ref{cond:4}--\ref{M1}--\ref{M2} hold and that there exists $\delta>0$ small enough such that \eqref{def:L_pi} and \eqref{assump:F} are satisfied. Then \smash{$P_0^n(\hat{A}_{n,4})=1-o(1)$}. Furthermore, let \smash{$\hat{h}= \sqrt{n}(\theta - \hat{\theta})$} and $M_n = \sqrt{c_0 \log n}$, with $c_0$ satisfying
		\begin{align}
	c_0\geq 2[c_1^*/d+ L_{\pi, \delta}/d+\log (\bar\eta_2/(2\pi))/2-\log (C_{\pi,\delta}/2)/d\big]/c_5\vee  2/\bar\eta_1, \label{cond_c0:2}
	\end{align}
	where $c_1^*$ is defined as
	\begin{equation} \label{def:star:c1}
		c_1^* = 4L_3/3 + 2L_4/3 +  2L_{\pi,2}.
	\end{equation}
	Then, conditioned on $\hat{A}_{n,4}$,  for  $n$ large enough satisfying
\begin{align}
	 \frac{ M_n^3d^{3/2}}{\sqrt{n}} \leq 1\wedge\frac{\delta}{2}\wedge \frac{3}{4L_3}\Big(\delta\wedge \frac{1}{2\sqrt{L_{F,\delta}}}\wedge \frac{1}{4} \Big)\wedge \frac{1}{2c_2^*},  \label{assump:n:large}
\end{align}
with 
\begin{equation}\label{def:c2*}
	c^*_2 =16(2+L_{F,\delta}) L_3^2  /9+ 2 \cdot 16^2L_{F,\delta}^2 L_3^4 /9^2+2 L_4/3+ 2L_{\pi,2},
\end{equation}
and $L_{\pi,\delta}, C_{\pi,\delta}>0$ given in \eqref{def:L_pi}, we have that
	\begin{eqnarray} \label{tot:variatindist:map:fs}
		\| \Pi_n( \cdot ) - \hat{P}^n_{\textsc{sks}}(\cdot)  \|_{\textsc{tv}} \leq (M_n^6 d^3/n) \hat{r}_{\textsc{s-tv}}(n,d), 
	\end{eqnarray}
where $\hat{P}_{\textsc{sks}}^n(S) \, = \, \int_{S} \hat{p}_{\textsc{sks}}^n(\hat{h}) d\hat{h}$ for $S \subset \mathbbm{R}^{d}$ with $\hat{p}_{\textsc{sks}}^n(\hat{h})$ denoting the skew--modal approximating density defined in~\eqref{limiting:distribution:map}, while 
	\begin{equation}\label{def:remainders}
		\begin{split}
			&\hat{r}_{\textsc{s-tv}}(n,d) =   2c_2^*+ 2(c_2^*)^2e M_n^6d^3/n+4n^{1-\bar{\eta}_1 c_0}	
			+2n^{1-(c_0 c_5/2)d}.
		\end{split}
	\end{equation}

	In addition, let \smash{$G \, : \, \mathbbm{R}^d \to \mathbbm{R} $} be a function satisfying \smash{$ | G(\hat h) | \leq C_G \| \hat h\|^r $}, for some constant $C_G>0$. Then for $n$ satisfying also $n\geq 2^{(r+1)/(2\bar{\eta}_1 c_0)}$,  it holds, conditioned on $\hat{A}_{n,4}$, that
	\begin{equation} \label{result:corol:functionals:fs}
		\begin{aligned}
		\int  G(\hat h) | \pi_n(\hat h) - \hat{p}_{\textsc{sks}}^n(\hat h) | d\hat h /C_G
		\leq  \frac{2^r M_n^{6+r}d^{3+r/2}}{n} \hat{r}_{\textsc{s-tv}}(n,d)+ \frac{\E_{\pi} \|\hat{h}\|^r}{n^{c_0 c_5d/2}}
		+  \frac{2^{2r+2} (M_n \sqrt{d})^r}{n^{ 2 \bar{\eta}_1 c_0 }} 	 .
		\end{aligned} 
	\end{equation}  
\end{theorem} 

\begin{remark}
Let us discuss briefly the conditions and results of the previous theorem. The condition \eqref{assump:n:large} on $n$ is needed in order to provide a finite sample control for the total variation distance. Under such a condition, for $n$ large enough, all the summands within the expression for $\hat{r}_{\textsc{s-tv}}(n,d)$ in \eqref{def:remainders}, except $2c_2^*$, can be always bounded by an arbitrarily small constant $\bar{c}$ not depending on $d$ and $n$. As such, the term $\hat{r}_{\textsc{s-tv}}(n,d)$ in \eqref{tot:variatindist:map:fs} can be itself upper bounded by $C=2c_2^*+\bar{c}$, thereby yielding the constant in Remark~\ref{rem_31} within the article. In addition, investigating the proof reveals that $c^*_2 $ could be replaced in the limit by $16L_3^2(2+L_{F,\delta})/9$. Similarly, $c_1^*$ can be replaced asymptotically by $4L_3/3$. Then the derived upper bound $M_n^6d^3/n$  on the total variation distance in fact extends our results to the high dimensional setting in which $d$ grows with $n$ at a rate that is, up to a poly–log term, slower than $n^{1/3}$. Leveraging~this discussion, it can also be noticed that the upper bound for the functional in \eqref{result:corol:functionals:fs} is asymptotically equal to $2^{r+1} c^*_2 M_n^{6+r}d^{3+r/2}/n.$ Finally, as already discussed, by fixing $d$ and letting $n \to \infty$, Theorem~\ref{thm:modal:fs} yields directly to Theorem \ref{thm:modal}.
\end{remark}

\begin{proof}[proof of Theorem \ref{thm:modal:fs}]
First note that, in view of Assumption \ref{cond:4},
\begin{align*}
 \bar{A}_{n,0} := \lbrace \sup\nolimits_{\|\theta - \theta_*\|> M_n\sqrt{d}/\sqrt{n}}\{ (\ell(\theta) - \ell(\hat \theta))/n  \} < - c_5 dM_n^2/n - (\ell(\hat \theta) - \ell( \theta_*)) /n \rbrace. \end{align*}
hold with $P_{0}^n$-probability $P_0^n(\bar{A}_{n,0})\geq 1-\hat{\epsilon}_{n,3}$, where $\hat{\epsilon}_{n,3}=o(1)$. Note also that under  $\hat{A}_{n,0}$, in view of $\hat\theta$ maximizing the posterior, for any $\delta>M_n\sqrt{d}/\sqrt{n}$, we have
	\begin{align*}
	- [\ell(\hat \theta) - \ell( \theta_*)] /n=  -[\log \pi_n(\hat\theta)-\log \pi_n(\theta_*) ]/n+[\log \pi(\hat\theta)-\log \pi(\theta_*)]/n\leq  L_{\pi, \delta}/n.  
	\end{align*}
Therefore, $P_0^{n}(\hat{A}_{n,0}\cap \hat{A}_{n,3})\geq P_0^{n}(\hat{A}_{n,0}\cap \bar{A}_{n,0})\geq 1-\hat\epsilon_{n,0}-\hat\epsilon_{n,3}$. Hence by the union bound
	\begin{align*}
P_0^{n}(\hat{A}_{n,4})\geq 1-\sum\nolimits_{i=0}^3\hat{\epsilon}_{n,i}=1-o(1).
\end{align*}

To finish the proof we follow a similar lines of reasoning as in Theorem \ref{thm:1} and Corollary \ref{corol:1} with the main additional difference of restricting ourselves to the event  $\hat{A}_{n,4}$ and tracking the constant terms and dimension dependence explicitly in the rest of the proof.

Let us first split  the problem into three parts using triangle inequality 
	\begin{equation} \label{triangle:tv:map:fs}
		\scalemath{0.97}{	\begin{aligned}
				&\int| \pi_n(\hat h) - \hat{p}_{\textsc{sks}}^n(\hat h)  | d\hat h \\
				&\leq 	{\int}| \pi_n(\hat h) -  \pi_n^{\hat K_n}(\hat h) | d\hat h + {\int}| \pi_n^{\hat K_n}( \hat h) - \hat{p}_{\textsc{sks}}^{n, \hat K_n}( \hat h) |d\hat h \  {+}  {\int}| \hat{p}_{\textsc{sks}}^n(\hat h) - \hat{p}_{\textsc{sks}}^{n, \hat K_n}( \hat h) | d\hat h,  
		\end{aligned}}
	\end{equation}
	where \smash{$\pi_n^{\hat K_n}( \hat h)= \pi_{n}(\hat h)\mathbbm{1}_{\hat h \in \hat K_n}/\Pi_n(\hat K_n)$} and 	\smash{$\hat{p}_{\textsc{sks}}^{n, \hat K_n}( \hat h)=\hat{p}_{\textsc{sks}}^n(\hat h) \mathbbm{1}_{\hat h \in \hat K_n}/\hat P_{\textsc{sks}}^n(\hat K_n)$}	are the~versions of $\pi_n(\hat h)$ and $\hat{p}_{\textsc{sks}}^n(\hat h)$ restricted to 
	\begin{equation}
	\hat K_n=\{\hat{h}:\, \|\hat{h}\|\leq 2\sqrt{d}M_n \}.\label{def:hat_Kn}
	\end{equation}
	A standard inequality of the \textsc{tv} norm together with Lemma \ref{lemma:post:contr:map} implies that, on $\hat{A}_{n,4}$, for a large enough choice of $c_0$, satisfying \eqref{cond_c0:2},
	\begin{equation}\label{bound:out:pi:map:fs}
		\int| \pi_n(\hat h) -  \pi_n^{\hat K_n}(\hat h) | d\hat h \leq 2 \Pi_n(\hat K_n^c) \leq 2n^{-(c_0 c_5/2)d}.
	\end{equation}
	We deal now with the third term in a similar manner. The same \textsc{tv} inequality as above,~together with the invariance of \textsc{sks} with respect to even functions  (leveraged in the proof of Lemma~\ref{lemma:distr:inv}), and Lemma \ref{lemma:gauss:contr:map} give 
	\begin{equation*} 
	\begin{split}
		\int| \hat{p}_{\textsc{sks}}^n(\hat h) - \hat{p}_{\textsc{sks}}^{n, \hat K_n}( \hat h)| d\hat h 
		&\leq 2 \int_{ \hat h \,:\, \|\hat h\|>2 M_n\sqrt{d} }\hat{p}_{\textsc{sks}}^n(\hat h) d\hat h \\
		&= 2 \int_{ \hat h \,:\, \|\hat h\|>2 M_n\sqrt{d} } \phi_d(\hat h;0,\hat \Omega) d\hat h
		\leq 4 n^{-\bar{\eta}_1 c_0}.
	\end{split}
	\end{equation*}

	We are left to deal with the second term on the right side of \eqref{triangle:tv:map:fs}. To this end, define
\begin{align}
\hat r_{n,4}( \hat h) :=&\log \Big[ \frac{p_{ \hat \theta + \hat h/\sqrt{n} }^n}{p_{\hat \theta}^n}(X^n)\frac{ \pi(\hat \theta +  \hat h/\sqrt{n} )}{\pi(\hat \theta)} \Big] + \hat \omega^{-1}_{st} \hat h_s \hat h_t/2 - \log \big(2  \hat{w}( \hat h)\big),\label{def:hat_rn4}	
\end{align}
	where	$ \hat \Omega^{-1} = \hat V^n$ with $V^n=[\hat{v}^n_{st}]=[j_{\hat \theta, st}/n]$, while  the skewing function $\hat{w}( \hat h)$ is defined as  \smash{$\hat{w}( \hat h) =F(\hat{\alpha}_{\eta}(\hat h))$, with  $\hat{\alpha}_{\eta}(\hat h) = \{1/(12 \eta \sqrt{n})\}( \ell_{\hat \theta, stl}^{(3)}/n) \hat h_s \hat h_t  \hat h_l$}. Furthermore, note that, conditioned on $\hat{A}_{n,4}$, it follows from Lemmas \ref{lemma:gauss:contr:map} and \ref{lemma:post:contr:map} that $\Pi_n(\hat K_n) > 0 $  and $\hat{P}^n_{\textsc{sks}}(\hat K_n) > 0 $.  Hence, similarly as in the proof of Theorem \ref{thm:1}, using Lemma \ref{lemma:modal:taylor} and $|1-e^x| \leq  |x| + e^{ x}x^2/2$
	we obtain for $n$ large enough satisfying $M_n^6d^3/n\leq1\wedge1/(2 c_2^*)$ that
	\begin{equation*} 
		\scalemath{1}{\begin{split}
				{\int}| \pi_n^{\hat K_n}( \hat h) - \hat{p}_{\textsc{sks}}^{n, \hat K_n}( \hat h) |d\hat h   &\leq  \int_{\hat K_n \times \hat K_n}| 1 - \exp[ \hat r_{n,4}( \hat h') -  \hat r_{n,4}( \hat h)]| \pi_n^{ \hat K_n}( \hat h)\hat{p}_{\textsc{sks}}^{n, \hat K_n}( \hat h') d \hat h d \hat h'  \\
				&  \leq 2c_2^*( M_n^6 d^3/n) + 2 (c_2^*)^2( M_n^{12} d^6 /n^2) \exp(2 c_2^* M_n^6 d^3 /n)\\
				&\leq (M_n^6 d^3/n) [2c_2^*+ 2(c_2^*)^2e( M_n^6 d^3/n)].
		\end{split}}
	\end{equation*} 
Let us now combine the above results to obtain \eqref{tot:variatindist:map:fs}.

In particular, as a direct consequence of the above derivations, an upper bound for the \textsc{tv} distance among $\Pi_n( \cdot )$ and  $\hat{P}^n_{\textsc{sks}}(\cdot)  $ is 
\vspace{-7pt}

\begin{equation*}
\begin{split}
&\smash{2n^{-(c_0 c_5/2)d}+4 n^{-\bar{\eta}_1 c_0}+(M_n^6 d^3/n) [2c_2^*+ 2(c_2^*)^2e( M_n^6 d^3/n)]}\\
&\smash{=(M_n^6 d^3/n)[2c_2^*+ 2(c_2^*)^2e( M_n^6 d^3/n)+2n^{1-(c_0 c_5/2)d}/(M_n^6 d^3)+4 n^{1-\bar{\eta}_1 c_0}/(M_n^6 d^3)]}\\
&\smash{ \leq (M_n^6 d^3/n)[2c_2^*+ 2(c_2^*)^2e( M_n^6 d^3/n)+2n^{1-(c_0 c_5/2)d}+4 n^{1-\bar{\eta}_1 c_0}]=(M_n^6 d^3/n)\hat{r}_{\textsc{s-tv}}(n,d)},
\end{split}
\end{equation*}

\vspace{3pt}
\noindent as in \eqref{tot:variatindist:map:fs}, with $\hat{r}_{\textsc{s-tv}}(n,d)$ defined in \eqref{def:remainders}.

	It remains to deal with the upper bound in Equation \eqref{result:corol:functionals:fs}. Similarly to Corollary~\ref{corol:1}, it is sufficient to prove the statement for \smash{$\|\hat h \|^r$}. Using the  triangle inequality we can again split the problem~into~three parts
	\begin{equation}\label{a.13.mom:fs}
		\scalemath{0.99}{		\begin{aligned}
				& \int  \|\hat h \|^r | \pi_n(\hat h) - \hat{p}_{\textsc{sks}}^n(\hat{h})  |  d\hat h \\
				&  \qquad \leq \int_{\hat K_n^c}  \|\hat h \|^r \pi_n(\hat h) d\hat h + \int_{\hat K_n^c}  \|\hat h \|^r \hat{p}_{\textsc{sks}}^n(\hat{h}) d\hat h  +\int_{\hat K_n} \|\hat h \|^r | \pi_n(\hat h) - \hat{p}_{\textsc{sks}}^n(\hat{h}) | d\hat h. 
		\end{aligned}}
	\end{equation} 
	For the first term above note that as in Lemma \ref{lemma:post:contr:map}, we have on the event $\hat{A}_{n,4}$ that
	\begin{equation} \label{lemma:functionals:help1:fs}
		\begin{split}
			\int_{\hat K_n^c}   \|\hat h\|^r \pi_n(\hat h) d\hat h & \leq  \int_{\hat K_n^c}   \|\hat h\|^r  \frac{ e^{\ell(\hat \theta + \hat h/\sqrt{n})-\ell(\hat \theta)} \pi(\hat \theta + \hat h/\sqrt{n}) }{\int_{\hat K_n} e^{\ell(\hat \theta + \hat h'/\sqrt{n})-\ell(\hat \theta)} \pi(\hat \theta + \hat h'/\sqrt{n})d\hat h'}d\hat h \\
			&\leq  \E_{\pi} \|\hat{h}\|^r n^{-(c_0 c_5/2)d}.			
		\end{split}
	\end{equation}
Similarly, conditioned on $\hat A_{n,4}$, the boundedness of $\hat{w}(\cdot)$ and Lemma \ref{lemma:gauss:contr:map} imply
	\begin{equation*} 
		\begin{split}
			&\int_{\hat K_n^c} \|\hat h\|^r \hat{p}_{\textsc{sks}}^n(\hat{h}) d\hat h  \leq   2 \int_{\hat K_n^c}   \|\hat h\|^r \phi_d(\hat{h}; 0, \hat \Omega) d\hat h \\
			&\qquad \leq  2(2M_n \sqrt{d})^r \sum\nolimits_{i=2}^{\infty} i^r P_{\hat{\Omega}}\big(2(i-1) M_n\sqrt{d}<\|\hat{h}\|\leq 2i M_n\sqrt{d} \big)\\
			&\qquad \leq   2(2M_n \sqrt{d})^r \sum\nolimits_{i=2}^{\infty} i^r e^{ -2 \bar{\eta}_1 (i-1)^2 M_n^2}
			\leq   2(4M_n \sqrt{d})^r\sum\nolimits_{i=1}^{\infty} i^r n^{ -2 \bar{\eta}_1 c_0  i}. 
		\end{split}
	\end{equation*}
	Let us introduce the notation $a_i=i^r e^{ -2 \bar{\eta}_1 c_0 i \log (n) } $. Then by noting that $a_{i+1}/a_i\leq 2^r n^{-2 \bar{\eta}_1 c_0 }\leq 1/2$ for $n\geq 2^{(r+1)/(2\bar{\eta}_1 c_0)}$, the preceding display can be further bounded by 
		\begin{equation}\label{lemma:functionals:help2:fs}
		\begin{aligned}
		\int_{\hat K_n^c}   \|\hat h\|^r \hat{p}_{\textsc{sks}}^n(\hat{h}) d\hat h\leq  2(4M_n \sqrt{d})^r  \sum\nolimits_{i=1}^{\infty} \frac{i^r}{n^{2 c_0 \bar{\eta}_1 i }} \leq   2^{2r+2} (M_n \sqrt{d})^r n^{ -2 \bar{\eta}_1 c_0  }. 	 
		\end{aligned}
	\end{equation}
	
	Finally, Equation \eqref{tot:variatindist:map:fs} implies
	\begin{equation}\label{lemma:functionals:help3:fs}
		\scalemath{0.94}{	\int_{\hat K_n} \|\hat h \|^r \big | \pi_n(\hat h) - \hat{p}_{\textsc{sks}}^n(\hat{h}) \big | d\hat h \leq  (2M_n\sqrt{d})^r	\int| \pi_n(\hat h) - \hat{p}_{\textsc{sks}}^n(\hat{h}) | d\hat h \leq  \frac{2^r M_n^{6+r}d^{3+r/2}}{n} \hat{r}_{\textsc{s-tv}}(n,d)}.
	\end{equation}
	Combining \eqref{a.13.mom:fs}, \eqref{lemma:functionals:help1:fs}, \eqref{lemma:functionals:help2:fs} and \eqref{lemma:functionals:help3:fs} provides  \eqref{result:corol:functionals:fs}.
\end{proof}

\vspace{5pt}
\begin{lemma} \label{lemma:modal:taylor}
	Suppose that conditions  \ref{M2}, \eqref{assump:F} and \eqref{assump:n:large} hold. Then on the event $\hat{A}_{n,2}$, we have that
	\begin{equation}\label{def:r_n4}
		\begin{aligned}
			\hat r_{n,4}:=\sup\nolimits_{\hat h \in \hat{K}_n}|\hat r_{n,4}( \hat h)|\leq c^*_2\frac{d^3 M_n^6}{n},
		\end{aligned}
	\end{equation}
where $c^*_2$ is given in \eqref{def:c2*}, while $\hat{K}_n$ and \smash{$\hat r_{n,4}( \hat h)$} are defined as in \eqref{def:hat_Kn} and  \eqref{def:hat_rn4}, respectively.	
\end{lemma}

\begin{proof}
	First note that since $\hat \theta$ is the \textsc{map}, the first derivative of the log--posterior is zero  by definition when evaluated at \smash{$\hat \theta$}. As a consequence using Assumption \ref{M2}, together with the combination of a third order Taylor expansion of the log--likelihood with a first order Taylor expansion of the log--prior gives  
	\begin{equation}\label{help:fin:map0}
		\begin{aligned}
			&\log \Big[ \frac{p_{ \hat \theta + \hat h/\sqrt{n} }}{p_{\hat \theta}}(X^n)\frac{ \pi(\hat \theta +  \hat h/\sqrt{n} )}{\pi(\hat \theta)}  \Big] + \frac{\hat \omega^{-1}_{st}}{2}  \hat h_s  \hat h_t -  \frac{1}{6 \sqrt{n}}  \frac{\ell^{(3)}_{\hat \theta, stl}}{n} \hat h_s \hat h_t  \hat h_l \\
			  &\qquad= \frac{1}{24 n}  \frac{\ell^{(4)}_{\hat \theta+ \beta \hat h/\sqrt{n}, stlk}}{n} \hat h_s \hat h_t  \hat h_l  \hat h_k +  \frac{1}{2 n}\log \pi_{\hat \theta+ \beta' \hat h/\sqrt{n}, st}^{(2)} \hat h_s \hat h_t,
		\end{aligned}
	\end{equation} 
	for some $\beta, \beta' \in [0,1]$.  As a consequence, conditioned on  $\hat{A}_{n,2}$, by the Cauchy--Schwarz~inequality and the upper bounds on the spectral norms of \smash{$\ell^{(3)}_{\theta}/n,$ $\ell^{(4)}_{\theta}/n$, and $\log \pi^{(2)}_{\theta}$} (see Assumption \ref{M2}), we have, for $n$ satisfying $2M_n\sqrt{d}/\sqrt{n}\leq \delta$, that
	\begin{equation}  \label{rem:3}
	\begin{split}
		&\sup_{\hat h \in \hat  K_n} \Big| \frac{1}{6 \sqrt{n}} \frac{\ell^{(3)}_{\hat \theta, stl}}{n} \hat h_s \hat h_t  \hat h_l \Big|  \leq \frac{4d^{3/2} L_3 M_n^{3}}{3 \sqrt{n}},\\
		&\sup_{\hat h \in \hat  K_n} \Big| \frac{1}{24 n} \frac{\ell^{(4)}_{\hat \theta+ \beta' \hat h/\sqrt{n}, stlk}}{n} \hat h_s \hat h_t  \hat h_l  \hat h_k \Big|  \leq \frac{2 d^2 L_4 M_n^4}{3 n},\\
	&\sup_{\hat h \in \hat  K_n} \Big| \frac{1}{2 n}\log \pi_{\hat \theta+ \beta' \hat h/\sqrt{n}, st}^{(2)} \hat h_s \hat h_t\Big|  \leq \frac{2d L_{\pi,2} M_n^2}{n}.	
	\end{split}
	\end{equation} 
	Furthermore, in view of condition \eqref{assump:F}, we also have that
	\begin{equation} \label{help:fin:00}
		\begin{aligned}
			\log\{ 2 F(\hat{\alpha}_{\eta}(\hat h)) \} & =  \log \Big[ 1 +  \frac{1}{6 \sqrt{n}}  \frac{\ell^{(3)}_{\hat \theta, stl}}{n} \hat h_s \hat h_t  \hat h_l + {r}_{F,\delta}\Big(  \frac{1}{6 \sqrt{n}}  \frac{\ell^{(3)}_{\hat \theta, stl}}{n} \hat h_s \hat h_t  \hat h_l \Big)  \Big],
		\end{aligned}
	\end{equation}
where in the remainder part the additional $1/(2\eta)$ term is incorporated within the function ${r}_{F,\delta}(\cdot)$. 
	
	By combining the above displays we get the following upper bound for $|\hat{r}_{n,4}|$,
	\begin{equation} \label{help:fin:map1}
	\scalemath{1}{	\begin{aligned}
			|\hat{r}_{n,4}|		& \leq \sup_{\hat h \in \hat  K_n} \Big|  \frac{1}{6 \sqrt{n}}  \frac{\ell^{(3)}_{\hat \theta, stl}}{n} \hat h_s \hat h_t  \hat h_l  -  \log \Big \lbrace 1 +  \frac{1}{6 \sqrt{n}}  \frac{\ell^{(3)}_{\hat \theta, stl}}{n} \hat h_s \hat h_t  \hat h_l +  {r}_{F,\delta}\Big(  \frac{1}{6 \sqrt{n}}  \frac{\ell^{(3)}_{\hat \theta, stl}}{n} \hat h_s \hat h_t  \hat h_l \Big)  \Big \rbrace	 \Big| \\
			& \qquad + \sup_{\hat h \in \hat  K_n}  \Big| \frac{1}{24 n}  \frac{\ell^{(4)}_{\hat \theta+ \beta \hat h/\sqrt{n}, stlk}}{n} \hat h_s \hat h_t  \hat h_l  \hat h_k \Big| + \sup_{\hat h \in \hat  K_n}  \Big| \frac{1}{2 n}\log \pi_{\hat \theta+ \beta' \hat h/\sqrt{n}, st}^{(2)} \hat h_s \hat h_t\Big|.
		\end{aligned}}
	\end{equation} 
	Notice that, by \eqref{rem:3} the last two summands in the above display can be upper--bounded by $2 d^2 L_4 M_n^4/(3 n)$ and $2d L_{\pi,2} M_n^2/n$, respectively. As for the first summand, note that for $n$ large enough satisfying $\eqref{assump:n:large}$ with $\delta<1/4$, we have,  in view of the inequalities $|x-\log(1+x)|\leq x^2$ if $|x|<0.5$ and $(a+b)^2\leq 2(a^2+b^2)$, Assumption \eqref{assump:F}, and Cauchy--Schwarz inequality, that such a summand can be upper--bounded by
	\begin{equation*}
		\begin{aligned}
			&\sup_{\hat h \in \hat  K_n} \Big| \frac{1}{6 \sqrt{n}}  \frac{\ell^{(3)}_{\hat \theta, stl}}{n} \hat h_s \hat h_t  \hat h_l +{r}_{F,\delta}\Big(  \frac{1}{6 \sqrt{n}}  \frac{\ell^{(3)}_{\hat \theta, stl}}{n} \hat h_s \hat h_t  \hat h_l \Big)  \Big|^2 + \Big| {r}_{F,\delta}\Big(  \frac{1}{6 \sqrt{n}}  \frac{\ell^{(3)}_{\hat \theta, stl}}{n} \hat h_s \hat h_t  \hat h_l \Big) \Big| \\
			&\qquad \qquad \leq \frac{16(2+L_{F,\delta}) L_3^2  d^3 M_n^6 }{9n}+ \frac{ 2 \cdot 16^2L_{F,\delta}^2 L_3^4  d^6 M_n^{12} }{9^2n^2}. 
		\end{aligned}
	\end{equation*}
Combining the above upper bound with those for the last two summands in \eqref{help:fin:map1} yields, for $n$ large enough satisfying $\eqref{assump:n:large}$, 
	\begin{equation*}
		\begin{split}
			|\hat{r}_{n,4}|	 &\leq \frac{16(2+L_{F,\delta}) L_3^2  d^3 M_n^6 }{9n}+ \frac{2 \cdot 16^2L_{F,\delta}^2 L_3^4  d^6 M_n^{12} }{9^2n^2}+\frac{2 d^2 L_4 M_n^4}{3 n}+ \frac{2d L_{\pi,2} M_n^2}{n}\\
			& \leq \frac{d^3 M^6_n}{n} \left[ \frac{16(2+L_{F,\delta}) L_3^2  }{9}+ \frac{2 \cdot 16^2L_{F,\delta}^2 L_3^4 }{9^2}+\frac{2 L_4}{3}+ 2L_{\pi,2}\right]=  \frac{d^3 M^6_n}{n} c^*_2,
		\end{split}
	\end{equation*}
thereby concluding the proof.
\end{proof}

\begin{lemma}[Concentration Gaussian modal approximation]  \label{lemma:gauss:contr:map}
	Let  \smash{$M_n = \sqrt{c_0 \log n}$} and denote with \smash{$P_{\hat{\Omega}}$} the centered Gaussian distribution for $\hat{h}$ with covariance matrix \smash{$\hat{\Omega}$}. Conditioned on the event $\hat{A}_{n,1} = \{\lambda_{\textsc{min}}(\hat{\Omega}^{-1}) > \bar{\eta_1}\}\cap \{\lambda_{\textsc{max}}(\hat{\Omega}^{-1}) < \bar{\eta_2}\}$ it holds that
\begin{equation} \label{Gaus:modal:conc}
		P_{\hat{\Omega}}(\hat h:\, \| \hat h\| > 2 M_n\sqrt{d}) \leq 2 n^{-2\bar{\eta}_1 c_0}.
	\end{equation} 
\end{lemma}
\begin{proof}
Let us write the covariance matrix in the form $\hat{\Omega}=\Gamma \Lambda \Gamma^\intercal$, where  $\Gamma$ comprises~the eigenvectors of \smash{$\hat{\Omega}$} as columns, whereas $ \Lambda$ is the diagonal matrix of eigenvalues. Then, we have that \smash{$\|\hat{h}\|\stackrel{d}{=}\|\Gamma \Lambda^{1/2} Z\| $}, with $Z \sim \mbox{N}_d(0,\mathrm{I}_d)$. Note that for every fixed $Z \in \mathbbm{R}^d$, Cauchy--Schwarz inequality gives
	\begin{equation*}
		\|\Gamma \Lambda^{1/2} Z\| \leq \|\Gamma \Lambda^{1/2}\|_{F} \| Z\| \leq (d/\bar{\eta}_1)^{1/2} \| Z\|,
	\end{equation*} 
where $\|\cdot\|_F$ denotes the Frobenius norm, while  the last inequality follows from $\|\Gamma \Lambda^{1/2}\|_{F} = (\sum_{i=1}^d \Lambda_{ii})^{1/2} \leq (d\lambda_{\textsc{max}}(\hat{\Omega}))^{1/2} $ and the conditioning on $\hat{A}_{n,1}$. Finally, an application of Hoeffding's inequality provides
\begin{equation*}
P_{\hat{\Omega}}(\hat h:\, \| \hat h\| > 2 M_n\sqrt{d}) \leq
P\Big( \|Z\| > \sqrt{\bar{\eta}_1} 2M_n \Big) \leq 2 \exp (   -2 \bar{\eta}_1 M_n^2) = 2 n^{-2 \bar{\eta}_1 c_0},
\end{equation*}
concluding the proof of the lemma.
\end{proof}

\begin{lemma}[Posterior contraction about \textsc{map}] \label{lemma:post:contr:map}
	Under the conditions of Theorem \ref{thm:modal:fs}, on the event $\hat{A}_{n,4}$, for $2\vee M_n^6d^3\leq n$ and $c_0$ satisfying \eqref{cond_c0:2} we have that
	\begin{eqnarray*}
		  \Pi_n(\hat{K}_n^c )  \leq  n^{-(c_0 c_5/2)d}.
	\end{eqnarray*} 
\end{lemma}

\begin{proof}[proof of Lemma~\ref{lemma:post:contr:map}] 
 First note that
	\begin{equation} \label{help:lemma:conc:1:map}
		\Pi_n(\hat{K}_n^c)  \leq \frac{\int_{\hat{K}_n^c} (p_{\hat{\theta} + \hat{h}/\sqrt{n}}/p_{\hat{\theta}})(X^{n}) \pi(\hat{\theta} + \hat{h}/\sqrt{n}) d\hat{h} }{\int_{\hat{K}_n}(p_{\hat{\theta} + \hat{h}/\sqrt{n}}/p_{\hat{\theta}})(X^{n}) \pi(\hat{\theta} + \hat{h}/\sqrt{n}) d\hat{h}}. 	
	\end{equation}
    Then, in view of $M_n = \sqrt{c_0 \log n}$, on the event $\hat{A}_{n,4}$, 
	\begin{equation} \label{contraction:numerator:map}
		\begin{aligned}
			\int_{  \hat{K}_n^c} (p_{\hat{\theta} + \hat{h}/\sqrt{n}}/p_{\hat{\theta}})(X^{n}) \pi(\hat{\theta} + \hat{h}/\sqrt{n}) d\hat{h} \leq n^{-c_0 c_5d} e^{L_{\pi, \delta}}.
		\end{aligned}
	\end{equation}
	For the denominator of the right--hand--side of \eqref{help:lemma:conc:1:map}, we use part of the results in the proof of Lemma \ref{lemma:modal:taylor} and the fact that, conditioned on $\hat{A}_{n,4}$, $\{ \bar{\eta}_1 < \lambda_{\textsc{min}}(\hat \Omega^{-1}) \} \cap \{\lambda_{\textsc{max}}(\hat \Omega^{-1}) < \bar{\eta}_2 \}.$ In particular, it follows from \eqref{help:fin:map0} and \eqref{rem:3} that 
		\begin{equation*}
		\begin{aligned}
			&\log \Big[ \frac{p_{ \hat \theta + \hat h/\sqrt{n} }}{p_{\hat \theta}}(X^n)\frac{ \pi(\hat \theta +  \hat h/\sqrt{n} )}{\pi(\hat \theta)}  \Big] = -\frac{\hat \omega^{-1}_{st}}{2}  \hat h_s  \hat h_t + \hat{r}_{n,2}(\hat h),
		\end{aligned}
	\end{equation*} 
	where 
	\begin{equation*}
		 \sup_{ \hat h \in  \hat K_n} | \hat{r}_{n,2}(\hat h)| \leq  \frac{d^{3/2} M_n^3}{\sqrt{n}} \Big( \frac{ 4L_3}{3} + \frac{2L_4 \sqrt{d} M_n }{3 \sqrt{n}} +  \frac{ 2L_{\pi,2}}{\sqrt{d}M_n\sqrt{n}}\Big)
		 \leq   c_1^* \frac{d^{3/2}M_n^3}{\sqrt{n}},
	\end{equation*}
with $c_1^*$ is defined in \eqref{def:star:c1}, for $\sqrt{d}M_n/\sqrt{n}\leq 1$.
    As a consequence, conditioned on $\hat{A}_{n,4}$,
    \begin{equation} \label{contraction:den:map}
    	\begin{aligned}
    		&\int_{\hat{K}_n}(p_{\hat{\theta} + \hat{h}/\sqrt{n}}/p_{\hat{\theta}})(X^{n}) \pi(\hat{\theta} + \hat{h}/\sqrt{n}) d\hat{h}\\
    		&\qquad \geq \pi(\hat \theta) (2\pi)^{d/2} |\hat{\Omega}|^{1/2} \exp( - c_1^*  \frac{M_n^3d^{3/2}}{\sqrt{n}}) P_{\hat{\Omega}}(\|\hat h\| \leq 2\sqrt{d} M_n).
    	\end{aligned}
    \end{equation}
 Then using \eqref{contraction:numerator:map} and \eqref{contraction:den:map}, the inequalities $|\hat{\Omega}|^{-1/2} \leq \bar{\eta}_2^{d/2} $ and $P_{\hat{\Omega}}(\|\hat h\| > 2 M_n\sqrt{d}) \leq 2n^{-2\bar{\eta}_1 c_0}\leq1/2 $ for $ c_0\geq 1/\bar\eta_1$ and $n\geq 2$,  Lemma \ref{lemma:gauss:contr:map}, Assumptions \ref{cond:3} and \eqref{def:L_pi}, we get~for~$n$ large enough, satisfying $2\vee d^3M_n^6\leq n$, and $c_0$ large enough satisfying \eqref{cond_c0:2}, that
  \begin{equation} \label{contraction:map}
  	\scalemath{0.94}{\begin{aligned}
  			&\frac{\int_{\hat{K}_n^c} (p_{\hat{\theta} + \hat{h}/\sqrt{n}}/p_{\hat{\theta}})(X^{n}) \pi(\hat{\theta} + \hat{h}/\sqrt{n}) d\hat{h} }{\int_{\hat{K}_n}(p_{\hat{\theta} + \hat{h}/\sqrt{n}}/p_{\hat{\theta}})(X^{n}) \pi(\hat{\theta} + \hat{h}/\sqrt{n}) d\hat{h}}\\ 
  			&\leq  \exp\Big[-c_0 c_5d\log n+   L_{\pi, \delta}+ c_1^* \frac{d^{3/2} M_n^3}{\sqrt{n}} +\frac{d\log (\bar\eta_2/2\pi)}{2}-\log \frac{C_{\pi,\delta}}{2}  \Big]\leq n^{-(c_0 c_5/2)d},
  	\end{aligned}}
  \end{equation}
concluding the proof of the lemma.
\end{proof}


\begin{theorem} \label{thm:modal:gauss:lower}
	Let us consider the assumptions and notations of Theorem \ref{thm:modal:fs}. Furthermore, assume that there exists a $\delta>0$ such that, for all $\theta, \theta' \in \{\theta \, : \, \|\theta -\theta_{*}\| < \delta \}$,~it holds \smash{$| \ell^{(3)}_{\theta,stl}-\ell^{(3)}_{\theta',stl}|/n \leq L_{3,2} \|\theta - \theta'\| $} for a positive constant $L_{3,2}>0$ and for every $s,t,l \in \{1,\dots,d\}$. Moreover, let
	\begin{equation*}
	h^*  = \argmax_{\hat{h} \, :\, \|\hat{h}\| = 1} | (\ell^{(3)}_{\theta_*, stl}/n) \hat{h}_s \hat{h}_t \hat{h}_l |,
	\end{equation*}
	and assume that 
	\begin{align}
	M^*:=\inf_n |(\ell^{(3)}_{\theta_*, stl}/n) h^*_s h^*_t h^*_l |> 0.\label{cond:asym}
	\end{align}
	 Then, conditioned on the event $\hat{A}_{n,4}$, for $n$ large enough satisfying $ M_n\sqrt{d}/\sqrt{n}\leq \delta$, the total variation distance between the posterior distribution and the Gaussian Laplace approximation has the following lower bound
	\begin{equation*}
		\| \Pi_n(\cdot) - \hat{P}_{\textsc{gm}}^n(\cdot) \|_{\textsc{tv}} \geq \frac{1}{\sqrt{n}} C_{d} - \frac{M_n^6 d^3}{n} \left(2\hat{r}_{\textsc{s-tv}}(n,d) + \frac{4L_{3,2}  }{3 d M^2_n} + \frac{ 16L_{F,\delta} L_3^2 }{9}\right),
	\end{equation*} 
	where $C_{d}>0$ is a constant possibly depending on $d$,  whereas $\hat{P}_{\textsc{gm}}^n(S) \, = \, \int_{S} \phi_d(\hat{h}; 0, \hat{\Omega}) d\hat{h}$ for measurable $S \subset \mathbbm{R}^{d}$. 
\end{theorem}
\begin{proof}
	We start by noting that, conditioned on the event $\hat{A}_{n,4}$ and in view of Theorem \ref{thm:modal:fs}, an application of triangle inequality gives
	\begin{equation*}
		\begin{aligned}
		&\int | \pi_{n}(\hat{h}) - \phi_d(\hat{h}; 0, \hat{\Omega})  |d\hat{h}\\
		 &\qquad \geq  \int | 2 \phi_d(\hat{h}; 0, \hat{\Omega}) \hat{w}(\hat{h})  - \phi_d(\hat{h}; 0, \hat{\Omega})  |d\hat{h} 
			 -\int | \pi_{n}(\hat{h}) - 2 \phi_d(\hat{h}; 0, \hat{\Omega}) \hat{w}(\hat{h})  |d\hat{h} \\
			&\qquad \geq \int_{\hat{h} \, : \, \|\hat{h}\| < 2\sqrt{d}M_n} | 2 \phi_d(\hat{h}; 0, \hat{\Omega}) \hat{w}(\hat{h})  - \phi_d(\hat{h}; 0, \hat{\Omega})  |d\hat{h}
			- \frac{2M_n^6 d^3}{n} \hat{r}_{\textsc{s-tv}}(n,d). 
		\end{aligned}
	\end{equation*}
	Next, notice that for $n$ satisfying \eqref{assump:n:large}, Assumption \eqref{assump:F} on the cdf $F$ and by triangle~inequality, we have that
	\begin{equation} \label{help:lower:0}
		\begin{aligned}
			& \int_{\hat{h} \, : \, \|\hat{h}\| < 2\sqrt{d}M_n} | 2 \phi_d(\hat{h}; 0, \hat{\Omega}) \hat{w}(\hat{h})  - \phi_d(\hat{h}; 0, \hat{\Omega})  |d\hat{h} \\
			 &\qquad = \int_{\hat{h} \, : \, \|\hat{h}\| <  2\sqrt{d}M_n} \Big| \frac{1}{6 \sqrt{n}} \frac{\ell^{(3)}_{\hat \theta, stl}}{n} \hat{h}_s \hat{h}_t \hat{h}_l + {r}_{F,\delta}\Big( \frac{1}{6 \sqrt{n}} \frac{\ell^{(3)}_{\hat \theta, stl}}{n} \hat{h}_s \hat{h}_t \hat{h}_l  \Big)  \Big| \phi_d(\hat{h}; 0, \hat{\Omega}) d\hat{h}\\
			& \qquad \geq \int_{\hat{h} \, : \, \|\hat{h}\| <  2\sqrt{d}M_n} \Big| \frac{1}{6 \sqrt{n}} \frac{\ell^{(3)}_{\hat \theta, stl}}{n} \hat{h}_s \hat{h}_t \hat{h}_l \Big| \phi_d(\hat{h}; 0, \hat{\Omega}) d\hat{h} - \frac{ 16L_{F,\delta} L_3^2   d^3M_n^6}{9 n}.  
		\end{aligned}
	\end{equation}
	Note also that, since by assumption \smash{$\ell^{(3)}_{\theta, stl}/n$} is Lipschitz continuous in a $\delta$--neighborhood of $\theta_*$ and, conditioned on $\hat{A}_{n,4}$ and $\|\hat{\theta} - \theta_{*}\| < M_n\sqrt{d}/\sqrt{n}$, the quantity in the last line of \eqref{help:lower:0} can be lower bounded by
	\begin{equation} \label{LB:final}
		\begin{aligned}
			 \int_{\hat{h} {:} \|\hat{h}\| < 2\sqrt{d}M_n} \Big| \frac{1}{6 \sqrt{n}} \frac{\ell^{(3)}_{\theta_*, stl}}{n} \hat{h}_s \hat{h}_t \hat{h}_l \Big| \phi_d(\hat{h}; 0, \hat{\Omega}) d\hat{h} - \frac{4L_{3,2}  d^2 M_n^4}{3 n} - \frac{ 16L_{F,\delta} L_3^2   d^3M_n^6}{9 n}.  
		\end{aligned}
	\end{equation}
	
	It remains to show that the first term above is lower bounded by a positive constant. Let us introduce the notation 
	 $$B_{\epsilon}(h^*) = \{  \hat{h} \, : \, \hat{h}  = h^* + \epsilon \tilde{h}, \, \, \|\tilde{h}\| \leq1 \}.$$
	 Then, on the event $\hat{A}_{n,4}$, the Cauchy--Schwartz and triangle inequalities together with Assumption~\ref{M2} and the fact that $\|\tilde{h}\| \leq1$ and $\|h^*\| =1$ imply, for every $\hat{h} \in B_{\epsilon}(h^*)$ with $(3\epsilon + 3 \epsilon^2  + \epsilon^3) < M^*/(2 L_3)$, that
	\begin{equation*}
		\begin{aligned}
			\Big| \frac{\ell^{(3)}_{\theta_*, stl}}{n} \hat{h}_s \hat{h}_t \hat{h}_l \Big| &= \Big| \frac{\ell^{(3)}_{\theta_*, stl}}{n} h^*_s h^*_t h^*_l + 3 \epsilon \frac{\ell^{(3)}_{\theta_*, stl}}{n} h^*_s h^*_t \tilde{h}_l + 3  \epsilon^2 \frac{\ell^{(3)}_{\theta_*, stl}}{n} h^*_s \tilde{h}_t \tilde{h}_l +  \epsilon^3\frac{\ell^{(3)}_{\theta_*, stl}}{n} \tilde{h}_s \tilde{h}_t \tilde{h}_l \Big| \\
			&\geq  M^* - L_3  (3\epsilon + 3 \epsilon^2  + \epsilon^3)>M^*/2.
		\end{aligned}
	\end{equation*}
	Since $B_{\epsilon}(h^*) \subset \{\hat{h} : \|\hat{h}\| < 2\sqrt{d}M_n \}$ and  $P_{\hat{\Omega}}( \hat{h} \in B_{\epsilon}(h^*))$ is lower bounded by a positive constant (depending on $d$), the first term in \eqref{LB:final} is also lower bounded by a positive constant times $1/\sqrt{n}$. Including such a lower bound in  \eqref{LB:final} concludes the proof.
\end{proof}


\vspace{2pt}
\section{Empirical studies}
In the following, we provide additional results related to the empirical studies considered in Sections~\ref{sec_24} and \ref{sec_32} of the main article. 

\subsection{Misspecified exponential model from Section~\ref{sec_242}}

Table \ref{tab01} reproduces the same analyses reported within Table~\ref{tab} of the main article, but now with a focus on the misspecified exponential model described in detail in Section~\ref{sec_242}. The focus is again on comparing the accuracy of the approximations arising from the classical (\textsc{BvM}) and the skewed (\textsc{s--BvM}) Bernstein--von Mises theorem, respectively.

As discussed in Section~\ref{sec_242}, Table~\ref{tab01} yields the same conclusions as those obtained for the correctly--specified setting in Table~\ref{tab}. These results further stress that the proposed \textsc{sks} class of approximating distributions outperforms remarkably the classical Gaussian arising from Bernstein--von Mises theorem, also in misspecified settings. The magnitude of these improvements is again in line with the expected gains encoded within the rates we derived in Section~\ref{sec_fixd_thm} from a theoretical perspective.

\begin{table}[t] 
	\renewcommand{\arraystretch}{1}
	\centering
	\caption{\footnotesize Empirical comparison, averaged over $50$ replicated studies, between the  classical (\textsc{BvM}) and skewed (\textsc{s--BvM}) Bernstein--von Mises theorem in the misspecified exponential example. The first table shows, for different sample sizes from $n=10$ to $n=1500$, the  log--\textsc{tv} distances (\textsc{tv}) and log--approximation errors for the posterior mean  (\textsc{fmae}) under \textsc{BvM} and \textsc{s--BvM}. The bold values indicate the best performance for each $n$. The second table shows,  for each $n$ from $n=10$ to $n=100$, the sample size $\bar{n}$ required by the classical Gaussian \textsc{BvM} to achieve the same \textsc{tv} and \textsc{fmae} attained by the proposed \textsc{sks} approximation with that $n$.
	\vspace{10pt}} 
	\small
	\begin{adjustbox}{max width=\textwidth}
		\begin{tabular}{lrrrrrr}
			\hline
			\quad & $n=10$  &  $n=50$  &   $n=100$ & $n=500$  &  $n=1000$ &  $n=1500$  \\ 
			 \hline
			$\log \textsc{tv}^n_{\textsc{BvM}}$ \quad & $-1.28$& $-2.16$& $-2.53$& $-3.28$& $-3.60$& $-3.86$ \\ 
			$\log \textsc{tv}^n_{\textsc{s--BvM}}$ \quad  & $\bf-2.32 $& $ \bf-3.59$& $ \bf-4.17$& $\bf -4.49$& $\bf-5.07$& $\bf -5.36$ \\
			\hline
			$\log \textsc{fmae}^n_{\textsc{BvM}}$ \quad &$0.15$& $-0.81$ &$-1.27$& $-2.13$& $-2.18$& $-2.64 $\\ 
			$\log \textsc{fmae}^n_{\textsc{s--BvM}}$ \quad  &$\bf-0.56$& $\bf-2.35$& $\bf-3.30$& $\bf -5.05$& $\bf -6.15$& $\bf -6.80$  \\ 
			\hline
		\end{tabular}
	\end{adjustbox}
	\label{tab01}
\end{table}
\vspace{-8pt}
\begin{table}[t!]
	\centering
	\begin{tabular}{lrrrrrrrr}
		\hline
		& $n=10$ & $n=15$ & $n=20$ & $n=25$ & $n=50$ & $n=75$ & $n=100$ \\ 
		\hline
		$\bar{n}: \ \textsc{tv}^{\bar{n}}_{\textsc{BvM}}=\textsc{tv}^{n}_{\textsc{s--BvM}} $ & 75 & 140 & 210 & 250 & 980 & 1560 & $>2500$ \\ 
		$\bar{n}: \ \textsc{fmae}^{\bar{n}}_{\textsc{BvM}}=\textsc{fmae}^{n}_{\textsc{s--BvM}} $  & 35 & 65 & 140 & 150 & 830 & 1560 & $>2500$ \\ 
		\hline
	\end{tabular}
	\vspace{20pt}
\end{table}

\vspace{20pt}
\subsection{Gamma--Poisson model}
Below we study in detail an additional important example that meets the conditions required to guarantee the validity of Corollary \ref{corol:1} and~Theorem~\ref{thm:modal}. Let $x_1,\dots,x_n$, be independent and identically distributed realizations of a Poisson random variable with mean $\theta$, and consider the case in which a Gamma prior $\mbox{Ga}(\alpha,\beta)$ on $\theta$ is assumed. In this framework, the log--likelihood of the model is $\ell(\theta) = \log(\theta)\sum_{i=1}^n x_i  - n \theta  - \sum_{i=1}^n \log(x_i!)$ while $\pi(\theta) \propto \theta^{\alpha-1} \exp(-\beta \theta)$ for $\alpha, \beta >0$. 

Let us verify that the conditions of Corollary \ref{corol:1} and Theorem \ref{thm:modal} are fulfilled by starting from Assumptions \ref{cond:1}--\ref{cond:2}--\ref{cond:3} and \ref{cond:4} (since we are in a correctly--specified setting, Assumption \ref{cond:uni} holds with $\theta_*=\theta$).  The first four log--likelihood derivatives are 
\begin{eqnarray*}
\ell^{(1)}_{\theta} =  n\bar{x}/\theta - n, \quad \ell^{(2)}_{\theta} = -  n\bar{x}/\theta^2, \quad \ell^{(3)}_{\theta} = 2n\bar{x}/\theta^3 \quad \ell^{(4)}_{\theta} = -  6 n\bar{x}/\theta^4,
\end{eqnarray*}
where $\bar{x} = \sum_{i=1}^n x_i/n$. Since \smash{$\E_0^n \ell^{(1)}_{\theta_*}= 0$}, it immediately follows that \smash{$\ell^{(1)}_{\theta_*} = O_{P_0^n}(n^{1/2})$}.~In addition, in view of \smash{$\bar{x} - \theta = O_{P_0^n}(n^{-1/2})$}, we have that \smash{$\ell^{(2)}_{\theta_{*}} = O_{P_0^n}(n)$, $\ell^{(3)}_{\theta_{*}} = O_{P_0^n}(n)$~and} $\sup_{\theta \, : \, |\theta- \theta_*|< \delta} \ell^{(4)}_{\theta}  = O_{P_0^n}(n)$ for any fixed $\delta>0$, and also $0<1/(\theta_{*}+ \epsilon) < I_{\theta_{*}}/n = 1/\theta_{*}<1/(\theta_{*}-\epsilon) $ for any $\epsilon>0$ sufficiently small and $J_{\theta_{*}}/n - I_{\theta_{*}}/n = (\bar{x}-\theta_{*})/\theta_*^2 = O_{P_0^n}(n^{-1/2})$. This implies that both Assumptions \ref{cond:1} and \ref{cond:2} are satisfied. Similarly, also Assumption \ref{cond:3} is easily seen to hold in view of the fact that, for every $\theta_{*}>0$, the Gamma density is bounded in a neighborhood of $\theta_{*}$. To ensure the validity of Corollary \ref{corol:1} the last condition to be checked is Assumption \ref{cond:4}. We prove it leveraging Lemma \ref{lemma:cond4}. First note that 
\begin{equation} \label{help:gamma:pois:0}
	\E_{0}^n (\ell(\theta) - \ell(\theta_{*}))/n = \theta_* \log(\theta/\theta_*) -(\theta- \theta_{*}),
\end{equation}
and 
\begin{equation} \label{help:gamma:pois:1}
	(\ell(\theta) - \ell(\theta_{*}))/n - 	\E_{0}^n (\ell(\theta) - \ell(\theta_{*}))/n = (\bar{x} - \theta_*) \log(\theta/\theta_*). 
\end{equation}
The right--hand--side of Equation \eqref{help:gamma:pois:0} is a non--positive function, having maximum at $\theta_{*}$, which is two times differentiable and concave. This implies that Assumption $R1$ of Lemma \ref{lemma:cond4}  is fulfilled. Similarly, since $\log(\theta/\theta_{*})/|\theta - \theta_{*}|$ is bounded for every $0<|\theta- \theta_*|< \delta$,
\begin{equation*}
	\sup\nolimits_{0<|\theta - \theta_{*}|< \delta}(\ell(\theta) - \ell(\theta_{*}))/n - 	\E_{0}^n (\ell(\theta) - \ell(\theta_{*}))/n = O_{P_0^n}(n^{-1/2}),
\end{equation*}
implying that also Assumption $R2$ of Lemma \ref{lemma:cond4}  is fulfilled. Note that these results for the quantities in the right--hand--side of \eqref{help:gamma:pois:0} and \eqref{help:gamma:pois:1} imply also that, for every $\delta>0$, there exists $c_{\delta}>0$ such that 
$P_0^n( \sup_{|\theta-\theta_*|>\delta} (\ell(\theta)- \ell(\theta_{*}))/n < c_{\delta}) \to 1$ as $n\to \infty$. Therefore, Lemma \ref{lemma:cond4} holds for the model under consideration. This concludes the part regarding~Corollary \ref{corol:1}. 

To demonstrate the validity of Theorem \ref{thm:modal}, note that in view of the conjugacy between Gamma prior and Poisson likelihood the \textsc{map} estimator takes the form 
\begin{equation*}
	\hat{\theta} = \frac{\alpha-1}{\beta + n} + \frac{n}{\beta + n} \bar{x}, 
\end{equation*}
for $\alpha + \sum_{i=1}^n x_i>1$. Thus,  $\hat{\theta} - \theta = O_{P_0^n}(n^{-1/2})$ and, as a direct consequence, Assumption \ref{M1} is fulfilled. Finally note that, for every $\delta>0$, the event \smash{$\hat{E}_n = \{ |\hat{\theta} - \theta_*| < \delta\} \cap \{|\bar{x}- \theta_*| < \delta\}$} has probability converging to one as $n \to \infty$. Conditioned on  \smash{$\hat{E}_n$}, in view of $\smash{\log \pi^{(2)}_{\theta}}= -(\alpha-1)/\theta^2$ and of the previously--derived results for the log--likelihoods derivatives, it follows that \smash{$|\ell^{(3)}_{\theta}/n|$, $|\ell^{(4)}_{\theta}/n|$ and $|\log \pi^{(2)}_{\theta} |$} are bounded by positive constants in a sufficiently small neighborhood of $\hat{\theta}$. This last observation implies Assumption \ref{M2} and, therefore, the validity of Theorem \ref{thm:modal} for the Gamma--Poisson model.

\subsection{Exponential model revisited from Section~\ref{sec_321}} Tables~\ref{tab1_supp} and \ref{tab_exp_miss_n}, together with Table~\ref{tab1} in the article, reproduce the same outputs reported in Tables \ref{tab} and \ref{tab01}, but now~with  focus on the practical skew--modal (\textsc{skew--m}) approximation in Section~\ref{sec_31}, rather than its population version which assumes knowledge of $\theta_*$. Consistent with this focus, the performance of the skew--modal approximation in Equation~\eqref{post:theta:map} is compared against the Gaussian  \smash{$\mbox{N}(\hat{\theta}, J_{\hat{\theta}}^{-1})$} arising from the Laplace method  (\textsc{gm})  \citep[see e.g.,][p. 318]{gelman2013bayesian}.

As discussed in detail within  Section~\ref{sec_321}, the remarkable improvements of \textsc{skew--m} in Tables~\ref{tab1_supp}--\ref{tab_exp_miss_n} are in line with those reported for its theoretical \textsc{s--BvM} counterpart~in~Section~\ref{sec_24}. Interestingly, by comparing the results in Tables~\ref{tab1_supp} and \ref{tab_exp_miss_n} with those  in Tables \ref{tab} and \ref{tab01} it is possible to notice that \textsc{skew--m}  approximates even more accurately the target posterior than its theoretical \textsc{s--BvM} counterpart. This is because the practical skew--modal approximation is located, by definition, at the actual \textsc{map} \smash{$\hat{\theta}$} of the target posterior, whereas its theoretical counterpart relies on $\theta_*$. Therefore, in practical implementations,  \textsc{skew--m} is expected to be closer to the actual posterior of interest than its theoretical version, since, in finite samples, there might be a non--negligible difference between $\theta_*$ and the \textsc{map} \smash{$\hat{\theta}$} of  the target posterior to be approximated.

\begin{table}[t]
\renewcommand{\arraystretch}{1}
\centering
\caption{\footnotesize 
For both the correctly--specified and misspecified exponential example, empirical comparison, averaged over $50$ replicated studies, between the  classical Gaussian modal approximation  from the Laplace method (\textsc{gm}) and skew--modal one developed in Section~\ref{sec_31}  (\textsc{skew--m}). The table shows, for different sample sizes from $n=10$ to $n=1500$, the  log--\textsc{tv} distances (\textsc{tv}) and log--approximation errors for the posterior mean  (\textsc{fmae}) under \textsc{gm} and \textsc{skew--m}, respectively. The bold values indicate the best performance for each $n$.  
\vspace{5pt}} 
\small
  \begin{adjustbox}{max width=\textwidth}
\begin{tabular}{lrrrrrr}
  \hline
 \quad & $n=10$  &  $n=50$  &   $n=100$ & $n=500$  &  $n=1000$ &  $n=1500$  \\ 
   \hline
 {\scriptsize \bf Correctly--specified}  &  &  &   &   &   &   \\ 
  	\hline
  $\log \textsc{tv}^n_{\textsc{gm}}$ \quad & $-2.48$ &$-3.28$& $-3.63$& $-4.43$& $-4.78$& $-4.98$ \\ 
  $\log \textsc{tv}^n_{\textsc{skew--m}}$ \quad  &$\bf-3.71$ &$\bf-5.33$& $\bf-6.03$& $\bf-7.65$& $\bf-8.34$& $\bf-8.74$\\
  \hline
  $\log \textsc{fmae}^n_{\textsc{gm}}$ \quad &$-0.61$&$ -1.30$& $-1.63$& $-2.41$& $-2.76$& $-2.96$ \\ 
  $\log \textsc{fmae}^n_{\textsc{skew--m}}$ \quad  &$\bf -1.91$ &$\bf -3.52$& $\bf -4.35$& $\bf -6.50$& $\bf -7.50$& $\bf -8.09$ \\ 
  \hline
   {\scriptsize \bf Misspecified}  &  &  &   &   &   &   \\ 
  \hline
  $\log \textsc{tv}^n_{\textsc{gm}}$ \quad & $-2.48$& $-3.28$& $-3.63$& $-4.43$& $-4.78$& $-4.98$\\ 
  $\log \textsc{tv}^n_{\textsc{skew--m}}$ \quad & $\bf-3.71$& $\bf-5.33$& $\bf-6.03$& $\bf-7.65$ & $\bf-8.35$ &$\bf-8.75$ \\
  \hline
  $\log \textsc{fmae}^n_{\textsc{gm}}$ \quad &$-0.41$& $-1.05$ &$-1.36$& $-2.12$& $-2.46$& $-2.66$\\ 
  $\log \textsc{fmae}^n_{\textsc{skew--m}}$ \quad  &$\bf-1.71 $&$\bf-3.28$& $\bf-4.08$ &$\bf-6.21 $&$\bf-7.21$& $\bf-7.79$  \\ 
  \hline
\end{tabular}
\end{adjustbox}
\label{tab1_supp}
\end{table}

\begin{table}[t]
	\centering
	\caption{\footnotesize  Under the misspecified exponential example, Table~\ref{tab_exp_miss_n} reports, for each $n$ from $n=10$ to $n=50$, the sample size $\bar{n}$ required by the classical Gaussian--modal (\textsc{gm}) approximation from the Laplace method to obtain the same \textsc{tv} and \textsc{fmae} achieved by the proposed skew--modal  approximation (\textsc{skew--m}) with that~$n$.
\vspace{5pt}}
\label{tab_exp_miss_n}
	\begin{tabular}{lrrrrrrr}
		\hline
	 & $ n=10$ & $ n=15$ & $ n=20$ & $ n=25$ & $ n=50$  \\ 
		\hline
			$\bar{n}: \ \textsc{tv}^{\bar{n}}_{\textsc{gm}}=\textsc{tv}^{n}_{\textsc{skew--m}}$    & 150 & 260 & 470 & 730 & $>2500 $ \\ 
	$\bar{n}: \ \textsc{fmae}^{\bar{n}}_{\textsc{gm}}=\textsc{fmae}^{n}_{\textsc{skew--m}} $   & 220 & 450 & 760 & 1120 & $>2500$ \\ 
		\hline
	\end{tabular}
\end{table}

\begin{figure}[b!]
	\centering
		\includegraphics[width=1\textwidth,height=0.46\textwidth]{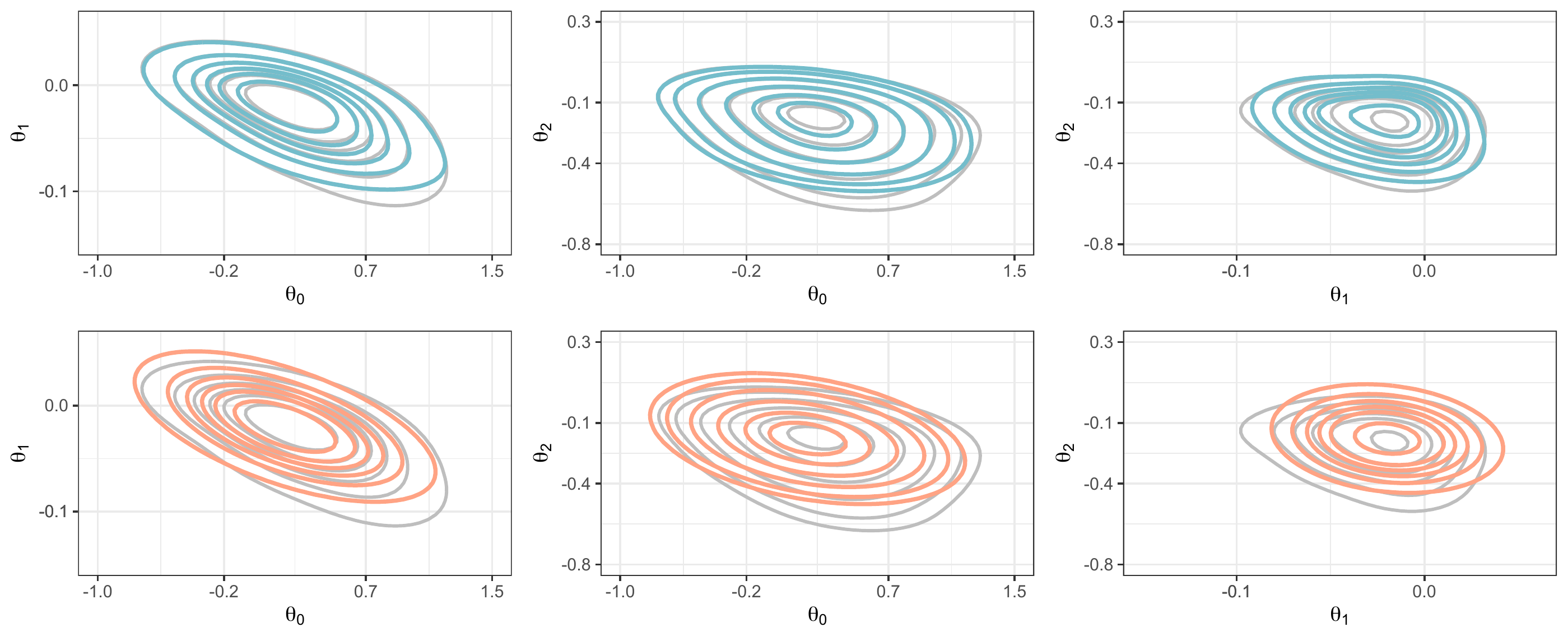}
		\caption{ \footnotesize Visual comparison between skew--modal (blue) and Gaussian (orange) approximations of the target bivariate posteriors (grey) for the three coefficients of the probit regression model in the Cushings application. }	
		\vspace{-10pt}
	\label{fig:probit_sm}
\end{figure}
\begin{figure}[b!]
	\centering
		\includegraphics[width=1\textwidth,height=0.46\textwidth]{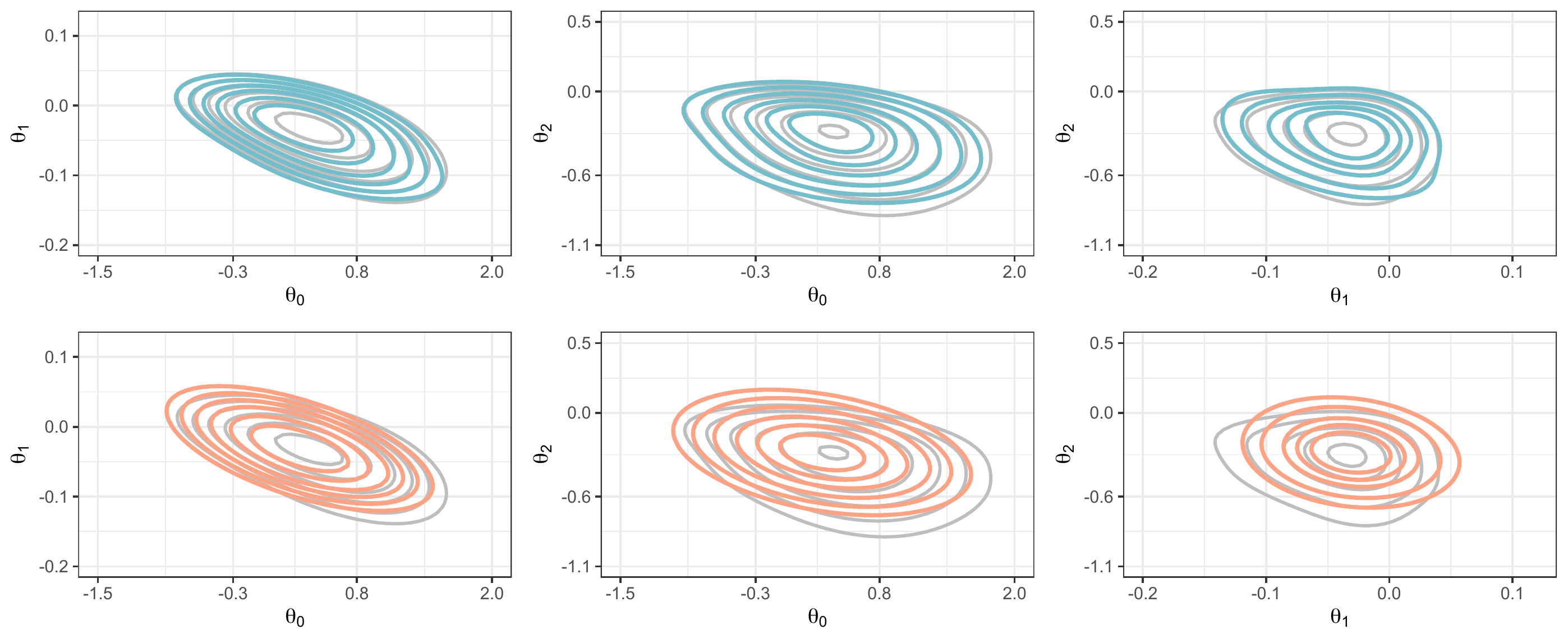}
		\caption{ \footnotesize Visual comparison between skew--modal (blue) and Gaussian (orange) approximations of the target bivariate posteriors (grey) for the three coefficients of the logistic regression model in the Cushings application. }	
		\vspace{-10pt}
	\label{fig:logit_sm}
\end{figure}

\subsection{Probit and logistic regression model from Section \ref{sec_322}}

This section reports some additional results for the real--data analysis described in Section \ref{sec_322}. Figures~\ref{fig:probit_sm} and \ref{fig:logit_sm} strengthen the results within Table~\ref{table:tv} in the main article by providing a graphical comparison between  the accuracy of the newly--developed skew--modal approximation for the bivariate posteriors in the  Cushings dataset, and the one obtained under the classical Gaussian--modal solution. Results confirm again the improved ability of  the proposed skew--modal in approximating the target posterior through an accurate characterization of its skewness.

 \begin{table}[t]
	\renewcommand{\arraystretch}{1}
	\centering
	\caption{\footnotesize For probit and logistic regression, estimated joint, bivariate and marginal total variation distances between the target posterior and the deterministic approximations under analysis in the Cushings application. The bold values indicate the best performance for each subset of parameters.} 
	\small
	\begin{adjustbox}{max width=\textwidth}
		\begin{tabular}{lrrrrrrr}
			\hline
			\quad & \qquad  \qquad $\textsc{tv}_{\theta}$ & $\textsc{tv}_{\theta_{01}}$  &  $\textsc{tv}_{\theta_{02}}$  &   $\textsc{tv}_{\theta_{12}}$ & $\textsc{tv}_{\theta_{0}}$  &  $\textsc{tv}_{\theta_{1}}$ &  $\textsc{tv}_{\theta_{2}}$  \\ 
			\hline
			Probit & &   &  &  &   & & \\ 
			\hline
			\textsc{skew--m} \qquad & $\bf 0.11$ & $\bf 0.05$ & $\bf 0.06$ & $\bf 0.09$ & $0.03$ & $\bf0.04$ & $\bf0.05$  \\ 
			\textsc{gm} \qquad& $0.19$ & $0.10$ & $0.13$ & $0.18$ & $0.09$ & $0.08$ & $0.11$    \\ 
			\textsc{ep} \qquad & $0.13$ & $0.07$ & $0.09$ & $0.11$ & $\bf 0.01$ & $0.07$ & $0.09$  \\ 
			\textsc{mf}--\textsc{vb} \qquad & $0.50$ & $0.32$ & $0.41$ & $0.47$ & $0.18$ & $0.28$ & $0.35$    \\ 
			\textsc{pfm}--\textsc{vb}\qquad & $0.25$ & $0.12$ & $0.22$ & $0.23$ & $0.06$ & $0.09$ & $0.19$    \\ 
			\hline
			Logit & &   &  &  &   & & \\ 
			\hline
			\textsc{skew--m} \qquad & $\bf 0.14$ & $0.08$ & $\bf 0.10$ & $0.13$ & $0.05$ & $ \bf 0.06$ & $\bf 0.07$\\ 
			\textsc{gm} \qquad & $0.23$ & $0.13$ & $0.17$ & $0.22$ & $0.11$ & $0.10$ & $0.14$  \\ 
			\textsc{ep} \qquad& $\bf 0.14$ & $ \bf 0.07$ & $0.11$ & $\bf 0.12$ & $\bf 0.01$ & $0.07$ & $0.10$   \\ 
			\textsc{mf}--\textsc{vb} \qquad& $0.25$ & $0.13$ & $0.21$ & $0.24$ & $0.07$ & $0.10$ & $0.19$    \\ 
			\hline
		\end{tabular}
	\end{adjustbox}
	\label{table:tv:app}
\end{table}

Table \ref{table:tv:app} concludes the analysis by assessing the behavior of the newly--proposed  \textsc{skew--m}~solution when compared against other advanced techniques within the class of deterministic approximations for binary regression models, beyond the classical Gaussian--modal (\textsc{gm}) alternative. State--of--the--art methods under this framework are mean--field variational Bayes (\textsc{mf}--\textsc{vb}) \citep{consonni2007mean,durante2019conditionally} and expectation--propagation (\textsc{ep}) \citep{Chopin_2017}, while partially--factorized variational Bayes   (\textsc{pfm}--\textsc{vb}) \citep{fasanoscalable} is available only for probit regression.  \textsc{mf}--\textsc{vb} and \textsc{pfm}--\textsc{vb} for probit regression leverage the implementation in the GitHub repository \texttt{Probit-PFMVB} \citep{fasanoscalable}, while in the logistic setting we rely on the codes in the repository 
\texttt{logisticVB} \citep{durante2019conditionally}. Note that \textsc{pfm}--\textsc{vb} is designed for probit regression only. Finally, \textsc{ep} is implemented under both models using the code in the GitHub repository \texttt{GaussianEP.jl} by Simon Barthelmé; see also the \texttt{R} library \texttt{EPGLM}  \citep{Chopin_2017} for a previous implementation.

The results in Table \ref{table:tv:app} highlight how the use of \textsc{skew--m} ensures noticeable improvements relative to  \textsc{mf}--\textsc{vb} and \textsc{pfm}--\textsc{vb}. The advantages over \textsc{pfm}--\textsc{vb} in the probit model are remarkable since also such a strategy leverages a skewed approximation of the target posterior distribution. This yields a higher accuracy than  \textsc{mf}--\textsc{vb}, but the improvements are not as noticeable as \textsc{skew--m}. A reason for this result is that  \textsc{pfm}--\textsc{vb}  has been originally developed to provide high accuracy in high--dimensional $p>n$ settings \citep{fasanoscalable}, whereas in this study $p=3$ and $n=27$. 
Remarkably, \textsc{skew--m} yields results competitive also with \textsc{ep}. This fact is noteworthy for at least two reasons. First, Gaussian \textsc{ep} methods are aimed at matching the  first two posterior moments. Being global characteristics of the target posterior, these objectives lead to approximations that, albeit symmetric, are often difficult to improve in practice. On the contrary, the proposed \textsc{skew--m} focuses on the local behavior of the posterior distribution in a neighborhood of its mode. It is therefore interesting to notice how inclusion of skewness can dramatically improve the global quality of an approximation, even when targeting its local behavior at the mode. For the Cushing~application, this translates into an average for the \textsc{tv} distances in Table \ref{table:tv:app} of $0.11$, lower than the average of $0.14$ achieved by \textsc{ep}. Second, \textsc{ep} techniques typically rely on a convenient factorization of the target density \citep[e.g.,][p. 338]{gelman2013bayesian}. Such a condition is not required for the adoption of \textsc{skew--m}, making it applicable to a wider range of models.

\end{appendix}

\end{document}